\documentclass[a4paper,11pt,intlimits,oneside]{amsart}

\usepackage{enumerate, color}
\usepackage{amsfonts,amsmath}

\usepackage{latexsym,amssymb}

\usepackage{mathrsfs}
\usepackage{hyperref}

\allowdisplaybreaks

\newcommand{\comment}[1]{}

\newcommand{\R}{{\mathbb R}}

\newcommand{\f}{\frac}
\newcommand{\vc}{\infty}

\def\G{{\mathcal G}}

\def\K{{\mathcal K}}

\def\X{{\mathcal X}}

\def\a{\mathfrak a}

\def\m{\mathfrak m}

\def\M{{\mathcal M}}

\def\X{{\mathcal X}}

\def\sign{{\mbox{\rm  sign}}\,}

\newcommand{\ackname}{Acknowledgements}
\makeatletter
\if@titlepage
\newenvironment{acknowledgement}{%
	\titlepage
	\null\vfil
	\@beginparpenalty\@lowpenalty
	\begin{center}%
		\bfseries \ackname
		\@endparpenalty\@M
\end{center}}%
{\par\vfil\null\endtitlepage}
\else
\newenvironment{acknowledgement}{%
	\if@twocolumn
	\section*{\ackname}%
	\else
	\begin{center}%
		{\bfseries \ackname\vspace{0em}\vspace{\z@}}%
	\end{center}%
	\quotation
	\fi}
\fi
\makeatother

\begin{document}
	
	\title[Commutators of singular integral operators]{Endpoint estimates for commutators of singular integral operators associated with admissible functions}         
	
	\author{The Anh Bui}  
	\address{Department of Mathematics, Macquarie University, NSW 2109, Australia} 
	\email{{\tt the.bui@mq.edu.au}}
	\author{Xuan Thinh Duong}  
	\address{Department of Mathematics, Macquarie University, NSW 2109, Australia} 
	\email{{\tt xuan.duong@mq.edu.au}}
	\author{Luong Dang Ky$^*$} 
	\address{Department of Education, Quy Nhon University, 170 An Duong Vuong, South Quy Nhon, Gia Lai, Vietnam} 
	\email{{\tt luongdangky@qnu.edu.vn}}

	\thanks{$^*$Corresponding author}

	\begin{abstract}
		Let $(\mathcal X, d, \mu)$ be a space of homogeneous type and let $\rho$ be an admissible function on $\mathcal X$. In this paper, we introduce a new class of singular integral operators associated with $\rho$, including a wide range of operators arising in harmonic analysis. For such an operator $T$ and a function $b$ belonging to suitable localized $\mathrm{BMO}$ spaces associated with $\rho$, which are strictly larger than the classical space $\mathrm{BMO}(\mathcal X)$, we establish the boundedness of the commutator $[b, T]$ from the Hardy space $H^1_\rho(\mathcal X)$ into $L^{1,\infty}(\mathcal X)$, $L^1(\mathcal X)$, and $H^1_\rho(\mathcal X)$. Moreover, the boundedness of $[b,T]$ on $H^1_\rho(\mathcal X)$ is characterized by necessary and sufficient conditions. We then apply this to investigate the boundedness of commutators of singular integrals in various settings, including Schr\"odinger operators on stratified Lie groups and Laguerre operators of convolution type. Our results are new even for Schr\"odinger operators on $\mathbb R^n$.	
	\end{abstract}

	\maketitle
	\newtheorem{theorem}{Theorem}[section]
	\newtheorem{lemma}{Lemma}[section]
	\newtheorem{proposition}{Proposition}[section]
	\newtheorem{remark}{Remark}[section]
	\newtheorem{corollary}{Corollary}[section]
	\newtheorem{definition}{Definition}[section]
	\newtheorem{example}{Example}[section]
	\numberwithin{equation}{section}
	\newtheorem{Theorem}{Theorem}[section]
	\newtheorem{Lemma}{Lemma}[section]
	\newtheorem{Proposition}{Proposition}[section]
	\newtheorem{Remark}{Remark}[section]
	\newtheorem{Corollary}{Corollary}[section]
	\newtheorem{Definition}{Definition}[section]
	\newtheorem{Example}{Example}[section]
	\newtheorem*{theoremjj}{Theorem J-J}
	\newtheorem*{theorema}{Theorem A}
	\newtheorem*{theoremb}{Theorem B}
	\newtheorem*{theoremc}{Theorem C}
	\newtheorem*{conjecture}{Conjecture}
	\newtheorem*{open question}{Open question}

	\section{Introduction and statement of the results}

	Given a locally integrable function $b$ on $\mathbb{R}^n$ and a classical Calderón–Zygmund operator $T$, we consider the linear commutator $[b, T]$ defined for smooth, compactly supported functions $f$ by
	$$[b, T](f) = bT(f) - T(bf).$$
	A classical result of Coifman, Rochberg, and Weiss (see \cite{CRW}) states that the commutator $[b, T]$ is bounded on $L^p(\mathbb{R}^n)$ for $1 < p < \infty$ when $b \in \mathrm{BMO}(\mathbb{R}^n)$. Unlike the classical Calder\'on--Zygmund theory, where weak-type $(1,1)$ estimates often play a key role in endpoint results, the proofs of boundedness for commutators do not rely on such weak-type estimates. Instead, dedicated endpoint theories have been developed for such commutators (see, for example, \cite{HST, Ky13, LKY, Pe}). These classical results motivate the study of commutators in the modern setting of singular integral operators associated with Schr\"odinger operators of the form 
	$$L = -\Delta + V$$
	 on $\mathbb{R}^n$ with $n \geq 3$, where $V$ is a nonnegative potential belonging to the reverse H\"older class $\mathrm{RH}_{n/2}(\mathbb{R}^n)$. When $b \in \mathrm{BMO}(\mathbb{R}^n)$ and $T$ is one of the Riesz transforms
	 $\nabla L^{-1/2}$, $\nabla^2 L^{-1}$, or $V^{1/2} L^{-1/2}$, it was shown in
	 \cite{LP} that, in general, the commutator $[b,T]$ fails to be bounded from the
	 Hardy space $H_L^1(\mathbb{R}^n)$ to $L^1(\mathbb{R}^n)$. Instead, it is only
	 bounded from $H_L^1(\mathbb{R}^n)$ to the weak space
	 $L^{1,\infty}(\mathbb{R}^n)$, where $H_L^1(\mathbb{R}^n)$ denotes the Hardy space
	 associated with $L$ introduced in \cite{DZ99}.
	 
	 Thus, a natural question is whether there exists a non-trivial proper subspace of
	 $\mathrm{BMO}(\mathbb{R}^n)$ such that, whenever $b$ belongs to this subspace, the
	 commutator $[b,T]$ is bounded from $H_L^1(\mathbb{R}^n)$ into
	 $L^1(\mathbb{R}^n)$, or even from $H_L^1(\mathbb{R}^n)$ into
	 $H_L^1(\mathbb{R}^n)$.
	 
	 An answer to this question was given in \cite{BHS,Ky15}. More precisely, the
	 third-named author in \cite{Ky15} constructed the largest subspace
	 \[
	 H^{1}_{L,b}(\mathbb{R}^n) \subset H^{1}_{L}(\mathbb{R}^n)
	 \]
	 such that all commutators of the class of so-called
	 Schr\"odinger--Calder\'on--Zygmund operators and the associated Riesz transforms
	 are bounded from
	 \[
	 H^{1}_{L,b}(\mathbb{R}^n) \quad \text{into} \quad L^{1}(\mathbb{R}^n).
	 \]
	 In addition, a characterization of those functions $b \in
	 \mathrm{BMO}(\mathbb{R}^n)$ for which
	 \[
	 H^{1}_{L,b}(\mathbb{R}^n) \equiv H^{1}_{L}(\mathbb{R}^n)
	 \]
	 was also obtained in \cite{Ky15}.
	 
	 \bigskip
	 
	 The main aim of this paper is to establish and characterise boundedness of commutators of suitable BMO functions and a general class of operators on spaces
	  of homogeneous type without the presence of the underlying operator. To state our results precisely, we would like to set up the paper's framework. 
	  Let $(\X,d)$ be a quasi-metric space, that is, a function $d:\X\times \X\to [0,\infty)$ satisfying
	 \begin{enumerate}[(a)]
	 	\item $d(x,y)=d(y,x)$ for all $x,y\in\X$,
	 	\item $d(x,y)>0$ if and only if $x\ne y$,
	 	\item there exists a constant $\kappa\geq 1$ such that 
	 	\begin{equation}\label{16:19, 05/10/2023}
	 		d(x,z)\leq \kappa(d(x,y)+ d(y,z))\quad\text{for all } x,y,z\in \mathcal X.
	 	\end{equation}
	 \end{enumerate}
	 A triple $(\mathcal X, d,\mu)$ is called a space of homogeneous type in the sense of Coifman and Weiss \cite{CW} if $\mu$  is a regular Borel measure satisfying the doubling condition: there exists a constant $C>1$ such that
	 \begin{equation}\label{doubling property}
	 	\mu(B(x,2r))\leq C\mu(B(x,r))\quad\text{for all }x\in \X\text{ and }r>0.
	 \end{equation}

	 Following the geometric structure of spaces of homogeneous type (see \cite[Theorem 2]{MS}), throughout this paper we assume that there exist a geometric constant $\alpha_0\in (0,1]$ and a constant $C>0$ such that
	 \begin{equation}\label{Macias-Segovia}
	 	|d(x,z)-d(y,z)|\leq C (d(x,z)+d(y,z))^{1-\alpha_0} d(x,y)^{\alpha_0}\quad\text{for all }x,y,z\in \X.
	 \end{equation}
We recall the notion of \emph{admissible functions} introduced in \cite{YZ}.

\begin{definition}\label{def:admissible}
	A positive function $\rho$ on $\mathcal X$ is said to be \emph{admissible} if there exist constants $k_0>0$ and $c_0>0$ such that
	\begin{equation}\label{admissible function}
		\frac{1}{\rho(x)} \le  \frac{c_0}{\rho(y)}
		\Bigl(1+\frac{d(x,y)}{\rho(y)}\Bigr)^{k_0}
		\qquad \text{for all } x,y\in\mathcal X.
	\end{equation}
	The \emph{admissibility index} of $\rho$ is defined by
	\begin{equation}\label{eq:admissible-index}
		i(\rho)
		:= \inf\Bigl\{k_0>0:\ \eqref{admissible function} \text{ holds for some constant }
		c_0=c_0(k_0)>0\Bigr\}.
	\end{equation}
\end{definition}

\begin{remark}\label{20:49, 02/12/2023}
	If $\rho$ is admissible and \eqref{admissible function} holds for some $k_0>0$, then there exists a constant $c_1>1$ such that
	\[
	\Bigl(1+\frac{d(x,y)}{\rho(y)}\Bigr)^{\frac{1}{k_0+1}}
	\le c_1 \Bigl(1+\frac{d(x,y)}{\rho(x)}\Bigr)
	\qquad \text{for all } x,y\in\mathcal X.
	\]
\end{remark}

The concept of admissible functions was first introduced in the setting of Schr\"odinger operators on $\mathbb R^n$ in \cite{F} (see also \cite{Sh}) and later extended to spaces of homogeneous type in \cite{YZ}.

A simple example of an admissible function is $\rho\equiv 1$. Moreover, one of the most important classes of admissible functions arises from weights satisfying a reverse H\"older inequality. Recall that a nonnegative locally integrable function $w$ is said to belong to the reverse H\"older class $\mathrm{RH}_q(\mathcal X)$ with $q>1$ if there exists a constant $C>0$ such that
\begin{equation}\label{def:reverse-holder-class}
	\Biggl(\frac{1}{\mu(B)}\int_B w(x)^q\,d\mu(x)\Biggr)^{1/q}
	\le C\,\frac{1}{\mu(B)}\int_B w(x)\,d\mu(x)
\end{equation}
for all balls $B\subset \mathcal X$. It is well known that if $w\in \mathrm{RH}_q(\mathcal X)$, then $w$ is a Muckenhoupt weight; see \cite{ST}.

Now assume that $V\in \mathrm{RH}_q(\mathcal X)$ for some $q>1$. Following \cite{Sh,YZ}, we define
\begin{equation}\label{eq:critical-function}
	\rho(x)
	:= \sup\Biggl\{r>0:\ 
	\frac{r^2}{\mu(B(x,r))}
	\int_{B(x,r)} V(y)\,d\mu(y)
	\le 1\Biggr\}.
\end{equation}
In the Euclidean setting $\mathcal X=\mathbb R^n$, it was proved in \cite{Sh,YZ} that the function $\rho$ defined by \eqref{eq:critical-function} is an admissible function provided that $q>\max\{1,n/2\}$.

Associated with an admissible function $\rho$, we now define Hardy spaces $H^1_\rho(\mathcal X)$ via atomic decompositions.

\begin{definition}\label{def:Hrho-atom}
	Let $q\in(1,\infty]$. A measurable function $\mathfrak a$ is called an $(H^1_\rho,q)$-atom associated with a ball $B=B(x_0,r)$ if $0<r<\rho(x_0)$ and the following conditions hold:
	\begin{enumerate}[\rm (a)]
		\item $\operatorname{supp}\mathfrak a \subset B$;
		\item $\|\mathfrak a\|_{L^q(\mathcal X)} \le \mu(B)^{\frac{1}{q}-1}$;
		\item if $0<r<\rho(x_0)/4$, then
		\[
		\int_{\mathcal X} \mathfrak a(x)\,d\mu(x)=0.
		\]
	\end{enumerate}
\end{definition}

For $q\in(1,\infty]$, the Hardy space $H^{1,q}_\rho(\mathcal X)$ is defined as the set of all functions $f\in L^1(\mathcal X)$ such that
\[
f=\sum_{j=1}^\infty \lambda_j \mathfrak a_j
\quad \text{in } L^1(\mathcal X),
\]
where each $\mathfrak a_j$ is an $(H^1_\rho,q)$-atom and $\{\lambda_j\}_{j=1}^\infty\in \ell^1$. The norm is defined by
\[
\|f\|_{H^{1,q}_\rho(\mathcal X)}
:= \inf \Biggl\{\sum_{j=1}^\infty |\lambda_j|:\ 
f=\sum_{j=1}^\infty \lambda_j \mathfrak a_j\Biggr\},
\]
where the infimum is taken over all atomic decompositions of $f$.

\begin{remark}
	It was proved in \cite{YZ} that
	\[
	H^{1,q}_\rho(\mathcal X)=H^{1,\infty}_\rho(\mathcal X)
	\quad \text{for all } q\in(1,\infty].
	\]
	Therefore, the Hardy space $H^1_\rho(\mathcal X)$ may be defined as any $H^{1,q}_\rho(\mathcal X)$ with $q\in(1,\infty]$.
\end{remark}

	Here and in what follows, for any ball $B\subset \X$ and $f\in L^1_{\rm loc}(\X)$, we denote 
	\begin{equation}
		f_B:= \frac{1}{\mu(B)}\int_B f(y)d\mu(y)\quad\text{and}\quad \mathrm{MO}(f,B):= \frac{1}{\mu(B)}\int_B |f(x)- f_B|d\mu(x).
	\end{equation}	
	Given $\theta\in [0,\infty)$, following \cite{BHS}, we define the spaces $\mathrm{BMO}_{\rho,\theta}(\X)$ and $\mathrm{BMO}^{\log}_{\rho,\theta}(\X)$ as follows. A locally integrable function $f$ belongs to $\mathrm{BMO}_{\rho,\theta}(\X)$ if
	\begin{equation}\label{eq-BMO rho}
		\|f\|_{\mathrm{BMO}_{\rho,\theta}}=\sup_{B} \frac{1}{\left(1+\frac{r}{\rho(x_0)}\right)^\theta} \mathrm{MO}(f,B) <\infty,
	\end{equation}
	where the supremum is taken over all balls $B=B(x_0,r)\subset\X$. Clearly, $\mathrm{BMO}(\X)\subsetneq \mathrm{BMO}_{\rho,\theta}(\X)$ for $\theta>0$.
	
	A locally integrable function $f$ belongs to $\mathrm{BMO}^{\log}_{\rho,\theta}(\X)$ if
	\begin{equation}
		\|f\|_{\mathrm{BMO}^{\log}_{\rho,\theta}}=\sup_{B} \frac{\log\left(e+\frac{\rho(x_0)}{r}\right)}{\left(1+\frac{r}{\rho(x_0)}\right)^\theta} \mathrm{MO}(f,B)<\infty,
	\end{equation}
	where the supremum is taken over all balls $B=B(x_0,r)\subset\X$. When $\theta=0$, we write $\mathrm{BMO}^{\log}_{\rho}(\X)$ instead of $\mathrm{BMO}^{\log}_{\rho,0}(\X)$.
	
	\begin{remark}\label{15:41, 31/12/2025}
		\begin{enumerate}[\rm (i)]
			\item For every $\theta \in [0,\infty)$ and every $f\in \mathrm{BMO}^{\log}_{\rho,\theta}(\X)$, we have
			\[\|f\|_{\mathrm{BMO}_{\rho,\theta}}\leq \|f\|_{\mathrm{BMO}^{\log}_{\rho,\theta}}.\]
			
			\item The space $\mathrm{BMO}_{\rho,0}(\X)$ is just the classical space $\mathrm{BMO}(\X)$. Moreover, whenever $0 \leq \theta_1 < \theta_2 < \infty$, we have
			\[\mathrm{BMO}_{\rho,\theta_1}(\X)\subset \mathrm{BMO}_{\rho,\theta_2}(\X)\quad\text{and}\quad \mathrm{BMO}^{\log}_{\rho,\theta_1}(\X)\subset \mathrm{BMO}^{\log}_{\rho,\theta_2}(\X).\]
			Furthermore, these inclusions are strict in general  $($see \cite[Remark 2.8]{Ky15}$)$.
		\end{enumerate}
	\end{remark}
	
	Let $\delta\in (0,1]$. Following \cite{MSTZ},  we say that $T$ is a $(\delta,\rho)$-Calder\'on--Zygmund operator if $T$ is bounded on $L^2(\X)$ and its kernel $K(x,y)$ satisfies the following conditions: for each $N>0$ there exists a constant $C=C(N)>0$ such that, for all $x\ne y$,
	\begin{equation}\label{11:06, 13/12/2023}
		|K(x,y)|\leq  \f{C}{\mu(B(x,d(x,y)))}  \left(1+\frac{d(x,y)}{\rho(x)}\right)^{-N};
	\end{equation}
	and there exists a constant $C>0$ such that
	\begin{equation}\label{20:56, 15/12/2023}
		|K(x,y)-K(x',y)|+|K(y,x)-K(y,x')|\leq \left[\frac{d(x,x')}{d(x,y)}\right]^\delta \frac{C}{\mu(B(x,d(x,y)))}, 
	\end{equation}
	whenever  $d(x,x')\leq \frac{1}{2\kappa}d(x,y)$, where $\kappa$ is the quasi-metric constant of $d$ in \eqref{16:19, 05/10/2023}.
	
	\bigskip

	Let $s\in (1,\infty)$ and $\delta\in (0,1]$. Following \cite{BDK, GLP} (see also  \cite{BPQ, BHQ, BLL, DLPV}), we say that $T$ is an $(s,\delta,\rho)$-Calder\'on--Zygmund operator if $T$ is bounded from $L^s(\X)$ into $L^{s,\infty}(\X)$ and its kernel $K(x,y)$ satisfies the following conditions: for each $N>0$ there exists a constant $C(N)>0$ such that, for all $d(y,x_0)\leq \frac{R}{2\kappa}$,
	\begin{equation}\label{11:08, 13/12/2023}
		\Big(\int_{R\leq d(x,x_0)< 2\kappa R} \left|K(x,y)\right|^s  d\mu(x)\Big)^{\frac{1}{s}}\leq   \frac{C(N)}{\mu(B(x_0,R))^{1-\frac{1}{s}}}\Big(1+\frac{R}{\rho(x_0)}\Big)^{-N};
	\end{equation}
	and there exists a constant $C>0$ such that, for all $d(y,x_0)\leq \frac{R}{2\kappa}$,
	\begin{align}\label{11:09, 13/12/2023}
		\Big(\int_{R\leq d(x,x_0)< 2\kappa R} |K(y,x)-K(x_0,x)|^s d\mu(x)\Big)^{\frac{1}{s}}\leq \left[\frac{d(y,x_0)}{R}\right]^\delta \frac{C}{\mu(B(x_0,R))^{1-\frac{1}{s}}}.
	\end{align}
	
	We now introduce a new class of operators which is more general than the class of $(s,\delta,\rho)$-Calder\'on--Zygmund operators defined above.
	
	\begin{definition}\label{07:59, 31/12/2025}
		Let $\theta\in [0,\infty)$ and $q\in (1,\infty]$. A bounded sublinear operator $T: L^q(\X)\to L^q(\X)$ belongs to the class $\K_{\rho,\theta,q}$ if the following two conditions hold:
		\begin{enumerate}[\rm (i)]
			\item $T$ extends to a bounded operator from $H^1_\rho(\X)$ into $L^1(\X)$;
			
			\item there exists a constant $C>0$ such that
			$$\|(b-b_B)T\a\|_{L^1}\leq C \|b\|_{\mathrm{BMO}_{\rho,\theta}}$$
			for all $b\in \mathrm{BMO}_{\rho,\theta}(\X)$ and all $(H^1_\rho,q)$-atoms $\a$ related to the balls $B$.
		\end{enumerate}
	\end{definition}

	\begin{remark}\label{rem 1.4}
		\begin{enumerate}[\rm (i)]
			\item\label{10:19, 22/12/2023} If $T$ is an $(s,\delta,\rho)$-Calder\'on--Zygmund operator with $s\in (1,\infty)$ and $\delta\in (0,1]$, then the following properties hold:
			\begin{enumerate}[$\bullet$]
				\item $T$ is of weak type $(1,1)$, and thus is bounded on $L^q(\X)$ for all $q\in (1,s)$;
				
				\item $T$ is a $(q,\delta,\rho)$-Calder\'on--Zygmund operator for all $q\in (1,s)$;
				
				\item from \cite[Theorem 5.1]{Ky25} and Lemma \ref{15:47, 09/12/2023}\eqref{09:32, 15/11/2025} below, $T$ belongs to the class $\K_{\rho,\theta,s}$ for all $\theta\in [0,\frac{\delta}{i(\rho)+1})$.
			\end{enumerate}
			
			\item\label{06:22, 26/12/2023} If $T$ is a $(\delta,\rho)$-Calder\'on--Zygmund operator with $\delta\in (0,1]$, then $T$ and $T^*$ are  $(s,\delta,\rho)$-Calder\'on--Zygmund operators for all $s\in (1,\infty)$. Thus:
			\begin{enumerate}[$\bullet$]
				\item $T$ and $T^*$ are of weak type $(1,1)$, and are bounded on $L^q(\X)$ for all $q\in (1,\infty)$;
				
				\item $T$ and $T^*$ belong to the class $\K_{\rho,\theta,s}$ for all $\theta\in [0,\frac{\delta}{i(\rho)+1})$ and all $s\in (1,\infty)$.
			\end{enumerate}
		
		\item Let $L = -\Delta + V$ be a Schr\"odinger operator on $\mathbb{R}^n$ with $n \geq 3$, where $V$  is a nonnegative potential belonging to the reverse H\"older class $\mathrm{RH}_{q}(\mathbb{R}^n)$ for some $q\in (n/2,n)$. The Riesz transform $\nabla L^{-1/2}$ may not be a $(\delta,\rho)$-Calder\'on--Zygmund operator; in fact, it is only a $(2, \delta, \rho)$-Calder\'on--Zygmund operator for some $\delta \in (0,1]$.	See for example \cite{Sh}.
		\end{enumerate}
		
	\end{remark}
	
	Let $\theta\in [0,\infty)$, $q\in (1,\infty]$, $b\in \mathrm{BMO}_{\rho,\theta}(\X)$, and $T\in \K_{\rho,\theta,q}$. As usual, the (sublinear) commutator $[b, T]$ is defined for suitable functions $f$ by
	\begin{equation}
		[b,T](f)(x):= b(x)Tf(x) - T(bf)(x)  \quad\text{for all } x\in\X.
	\end{equation}

\medskip

In this paper, we aim to investigate the boundedness of the commutator $[b,T]$ with $T \in \mathcal{K}_{\rho,\theta,q}$ and 
$b \in \mathrm{BMO}_{\rho,\theta}(\mathcal{X})$, 
where $\mathrm{BMO}_{\rho,\theta}(\mathcal{X})$, defined by \eqref{eq-BMO rho}, is a larger space than the classical $\mathrm{BMO}(\mathcal{X})$. 
The space $\mathrm{BMO}_{\rho,\theta}(\mathcal{X})$ was introduced in \cite{BHS} to study the $L^p$-boundedness of the commutator of the Riesz transforms associated with the Schr\"odinger operator and functions 
$b \in \mathrm{BMO}_{\rho,\theta}(\mathcal{X})$. 
Using similar methods, one can also establish the $L^p$-boundedness of the commutator $[b,T]$ for 
$T \in \mathcal{K}_{\rho,\theta,q}$ and 
$b \in \mathrm{BMO}_{\rho,\theta}(\mathcal{X})$ whenever $1<p<\vc$. 
Therefore, in this article we restrict our attention to the endpoint case $p=1$.

Our first result shows that the commutator is bounded from the Hardy space $H^1_\rho(\X)$ to the weak space $L^{1,\vc}(\X)$. 
More precisely, we have the following theorem.

\begin{theorem}\label{thm-main thm 1 from H1 to weak L1}
	Let $\theta\in [0,\frac{1}{i(\rho)+1})$, $q\in (1,\infty]$, and $b\in \mathrm{BMO}_{\rho,\theta}(\X)$. Then the commutator $[b, T]$ is bounded from $H^1_\rho(\X)$ into $L^{1,\infty}(\X)$ for all $T\in \K_{\rho,\theta,q}$ that are of weak type $(1,1)$.
\end{theorem}

In Theorem~\ref{thm-main thm 1 from H1 to weak L1}, it is impossible to replace $L^{1,\infty}(\X)$ by $L^{1}(\X)$, even in the case $\theta=0$ (where $\mathrm{BMO}_{\rho,\theta}(\X)=\mathrm{BMO}(\X)$). Indeed, to obtain boundedness from $H^1_\rho(\X)$ into $L^{1}(\X)$, one must assume that $b\in \mathrm{BMO}^{\log}_{\rho,\theta}(\mathcal{X})\subsetneq \mathrm{BMO}_{\rho,\theta}(\X)$. 

The following theorem not only establishes this stronger boundedness for functions 
$b\in \mathrm{BMO}^{\log}_{\rho,\theta}(\mathcal{X})$, but also provides a characterization of the space 
$\mathrm{BMO}^{\log}_{\rho,\theta}(\mathcal{X})$ in terms of the boundedness of commutators of operators in 
$\mathcal{K}_{\rho,\theta,q}$ from the Hardy space $H^{1}_{\rho}(\mathcal{X})$ into $L^{1}(\mathcal{X})$.

\begin{theorem}\label{thm-main thm 2 characterization BMO log}
	There exists a constant $\sigma_0\in (0,1]$ such that if $\theta\in \big[0,\frac{\sigma_0}{i(\rho)+1}\big)$, $q\in (1,\infty]$, and $b\in \mathrm{BMO}_{\rho,\theta}(\X)$, then the following statements are equivalent:
		\begin{enumerate}[\rm (i)]
			\item\label{item:BMO-log-condition} $b\in \mathrm{BMO}_{\rho,\theta}^{\log}(\X)$;
			
			\item\label{item:commutator-L1-bounded}  The commutator $[b, T]$ is bounded from $H^1_\rho(\X)$ into $L^1(\X)$ for all $T\in \K_{\rho,\theta,q}$.
		\end{enumerate}
\end{theorem}	

\begin{remark}
	In Theorem \ref{thm-main thm 2 characterization BMO log}, if $\X$ is a metric space, then one can choose $\sigma_0 = 1$.
\end{remark}

We now consider consequences of Theorems \ref{thm-main thm 1 from H1 to weak L1} and \ref{thm-main thm 2 characterization BMO log}. Following \cite{YYZ}, we say that a function $f\in \mathrm{BMO}(\X)$ belongs to $\mathrm{BMO}_{\rho}(\X)$ if
\begin{equation}\label{eq-BMO}\|f\|_{\mathrm{BMO}_{\rho}}= \|f\|_{\mathrm{BMO}}+ \sup_{B} |f|_B <\infty,
\end{equation}
where the supremum is taken over all balls $B=B(x_0,r)\subset\X$ with $r\geq \rho(x_0)$. It is known (see \cite[Theorem 2.1]{YYZ}) that $\mathrm{BMO}_{\rho}(\X)$ is the dual space of $H^1_\rho(\X)$. As consequences of Theorems \ref{thm-main thm 1 from H1 to weak L1} and \ref{thm-main thm 2 characterization BMO log}, together with 
Remark \ref{rem 1.4}, we obtain the following two corollaries.

	\begin{corollary}\label{cor:commutator-sCZ-H1-L1}
		Let $s\in (1,\infty)$, $\delta\in (0,1]$, $\theta\in \big[0,\frac{\delta}{i(\rho)+1}\big)$, and let $T$ be an $(s,\delta,\rho)$-Calder\'on--Zygmund operator. Then:
		\begin{enumerate}[\rm (i)]
			\item If $b\in \mathrm{BMO}_{\rho,\theta}(\X)$, then the commutator $[b, T]$ is bounded from $H^1_\rho(\X)$ into $L^{1,\infty}(\X)$.
			
			\item If $b\in \mathrm{BMO}_{\rho,\theta}^{\log}(\X)$, then the commutator $[b, T]$ is bounded from $H^1_\rho(\X)$ into $L^1(\X)$.
		\end{enumerate} 
	\end{corollary}

	\begin{corollary}\label{cor:commutator-CZ-H1-BMO}
		Let $\delta\in (0,1]$, $\theta\in \big[0,\frac{\delta}{i(\rho)+1}\big)$, and let $T$ be a $(\delta,\rho)$-Calder\'on--Zygmund operator. Then:
		\begin{enumerate}[\rm (i)]
			\item If $b\in \mathrm{BMO}_{\rho,\theta}(\X)$, then the commutators $[b, T]$ and $[b,T^*]$ are bounded from $H^1_\rho(\X)$ into $L^{1,\infty}(\X)$.
			
			\item If $b\in \mathrm{BMO}_{\rho,\theta}^{\log}(\X)$, then the commutators $[b, T]$ and $[b,T^*]$ are bounded from $H^1_\rho(\X)$ into $L^1(\X)$, and also from $L^\infty(\X)$ into $\mathrm{BMO}_\rho(\X)$.		
		\end{enumerate} 
	\end{corollary}
	
	\begin{remark}		
		 Corollaries \ref{cor:commutator-sCZ-H1-L1} and \ref{cor:commutator-CZ-H1-BMO} extend several earlier results in \cite{HW, LP, LHD}. Further details and additional applications are discussed in Section \ref{16:59, 30/12/2025}. 
	\end{remark}

	As mentioned earlier, Theorem~\ref{thm-main thm 2 characterization BMO log} characterizes the space $\mathrm{BMO}^{\log}_{\rho,\theta}(\mathcal{X})$ in terms of the $H^{1}$--$L^{1}$ boundedness of commutators of all operators in $\mathcal{K}_{\rho,\theta,q}$. A natural question is whether this characterization can be obtained by testing only a selected family of operators, rather than the entire class $\mathcal{K}_{\rho,\theta,q}$. 
	To provide an affirmative answer, we introduce the following class of operators.
	
	Let $\{T_j\}_{j=1}^n$ be $(s,\delta,\rho)$-Calder\'on--Zygmund operators  for some $s\in (1,\infty)$ and $\delta\in (0,1]$. We say that the Hardy space $H^1_\rho(\X)$ admits a characterization in terms of  $\{T_j\}_{j=1}^n$ if
	\[H^1_\rho(\X)=\left\{f\in L^1(\X): T_j(f)\in L^1(\X) \text{ for all } j=1,\ldots, n\right\}\]
	and there exists a constant $C>1$ such that 
	$$C^{-1}\|f\|_{H^1_\rho}\leq \|f\|_{L^1}+\sum_{j=1}^{n} \|T_j(f)\|_{L^1}\leq C \|f\|_{H^1_\rho}\quad\text{for all } f\in H^1_\rho(\X).$$

	Our next result may be viewed as an improved version of Theorem~\ref{thm-main thm 2 characterization BMO log}, since we are able to characterize the space $\mathrm{BMO}^{\log}_{\rho,\theta}(\mathcal{X})$ by using only a finite family of operators $\{T_j\}_{j=1}^n$, which already characterizes the Hardy space $H^1_\rho(\X)$ as described above, instead of relying on all operators in $\mathcal{K}_{\rho,\theta,q}$ as required in Theorem~\ref{thm-main thm 2 characterization BMO log}. 
	
	\begin{theorem}\label{thm-main thm 3 chracterize BMO log by Tj}
		Let $\delta\in (0,1]$ and $\theta\in \big[0,\frac{\delta}{i(\rho)+1}\big)$. Let $\{T_j\}_{j=1}^n$ be $(s,\delta,\rho)$-Calder\'on--Zygmund operators  for some $s\in (1,\infty)$.  If $H^1_\rho(\X)$ admits a characterization in terms of $\{T_j\}_{j=1}^n$, and if $b\in \mathrm{BMO}_{\rho,\theta}(\X)$ is such that all commutators $[b, T_j]$ are bounded from $H^1_\rho(\X)$ to $L^1(\X)$, then $b\in \mathrm{BMO}_{\rho,\theta}^{\log}(\X)$. Moreover, there exists a constant $C>1$, independent of $b$, such that
			\[
			C^{-1}\|b\|_{{\rm BMO}_{\rho,\theta}^{\log}}\leq \|b\|_{{\rm BMO}_{\rho,\theta}} + \sum_{j=1}^{n} \|[b, T_j]\|_{H^1_\rho\to L^1}\leq C \|b\|_{{\rm BMO}_{\rho,\theta}^{\log}}.
			\]
	\end{theorem}

	Theorem~\ref{thm-main thm 3 chracterize BMO log by Tj} also provides a sufficient condition on an operator $T$ ensuring that its commutator with a function in $\mathrm{BMO}_{\rho,\theta}^{\log}(\X)$ is bounded from $H^1_\rho(\X)$ to $L^1(\X)$. We now turn to establish a sufficient condition for the boundedness of such commutators on the Hardy space $H^1_\rho(\X)$ itself.
	
	\begin{theorem}\label{thm-main thm 4 chracterize BMO log by Tj on Hardy space}
		Let $\delta\in (0,1]$ and $\theta\in \big[0,\frac{\delta}{i(\rho)+1}\big)$.  If $b\in \mathrm{BMO}_{\rho,\theta}^{\log}(\X)$, then the commutator $[b, T]$ is bounded on $H^1_\rho(\X)$ for every $(s,\delta,\rho)$-Calder\'on--Zygmund operator $T$ with $s\in (1,\infty)$ and $T^*1\in \mathrm{BMO}^{\log}_\rho(\X)$.

		Conversely, let $\{T_j\}_{j=1}^n$ be $(s,\delta,\rho)$-Calder\'on--Zygmund operators  for some $s\in (1,\infty)$ such that $(T_j)^*1\in \mathrm{BMO}^{\log}_\rho(\X)$ for all $j=1,\ldots,n$.  If $H^1_\rho(\X)$ admits a characterization in terms of $\{T_j\}_{j=1}^n$, and if $b\in \mathrm{BMO}_{\rho,\theta}(\X)$ is such that all commutators $[b, T_j]$ are bounded on $H^1_\rho(\X)$, then $b\in \mathrm{BMO}_{\rho,\theta}^{\log}(\X)$. Moreover, there exists a constant $C> 1$, independent of $b$, such that
			\[
			C^{-1}\|b\|_{{\rm BMO}_{\rho,\theta}^{\log}}\leq \|b\|_{{\rm BMO}_{\rho,\theta}} + \sum_{j=1}^{n} \|[b, T_j]\|_{H^1_\rho\to H^1_\rho}\leq C \|b\|_{{\rm BMO}_{\rho,\theta}^{\log}}.
			\]
	\end{theorem}

	 Since $\mathrm{BMO}_\rho(\X)$ is the dual space of $H^1_\rho(\X)$,  it follows from Theorem~\ref{thm-main thm 4 chracterize BMO log by Tj on Hardy space} and Remark~\ref{rem 1.4}\eqref{06:22, 26/12/2023} that the following corollary holds.
	
	\begin{corollary}\label{14:48, 15/11/2025}
		Let $\delta\in (0,1]$ and $\theta\in [0,\frac{\delta}{i(\rho)+1})$. Assume that $b\in \mathrm{BMO}_{\rho,\theta}^{\log}(\X)$ and that $T$ is a $(\delta,\rho)$-Calder\'on--Zygmund operator satisfying $T1\in \mathrm{BMO}^{\log}_\rho(\X)$. Then, the commutator $[b, T]$ is bounded on $\mathrm{BMO}_\rho(\X)$.
	\end{corollary}

	\begin{remark}
		Theorems \ref{thm-main thm 3 chracterize BMO log by Tj} and \ref{thm-main thm 4 chracterize BMO log by Tj on Hardy space}, together with Corollary \ref{14:48, 15/11/2025}, extend and improve several earlier results from \cite{BHS, Ky15}. Further details and additional applications are discussed in Section \ref{16:59, 30/12/2025}.		
	\end{remark}

	The paper is organized as follows. Section \ref{17:02, 30/12/2025} provides several auxiliary lemmas used in the proofs of the main results. Section \ref{17:04, 30/12/2025} is devoted to the proofs of these results. Finally, Section \ref{16:59, 30/12/2025} presents additional examples and applications.

	Throughout the paper, we use $C$ to denote a positive constant, independent of the main parameters involved, whose value may vary from line to line. If $f\leq C g$, we write $f\lesssim g$ and, if $f\lesssim g\lesssim f$, we write $f\simeq g$. For any ball $B=B(x,r)$ and any $t>0$, we use $tB:=B(x,tr)$.  For any set $E\subset \X$, we let $\chi_E$ denote its characteristic function. For any index $q\in [1,\infty]$, we denote by $q'$ its conjugate index, that is, $\frac{1}{q}+\frac{1}{q'}=1$. Finally, we set $\mathbb{Z}^+:=\{1, 2,\ldots\}$ and $\mathbb{Z}^+_0:=\mathbb{Z}^+\cup\{0\}$.

		\section{Some auxiliary lemmas}\label{17:02, 30/12/2025}
	
	The main purpose of this section is to provide several auxiliary lemmas. These lemmas are used in the proofs of the main results from Section \ref{17:04, 30/12/2025} and also for some of the results in Section \ref{16:59, 30/12/2025}. Here and in what follows, for any ball $B$, we denote
	\begin{equation}
		\text{$U_0(B):= 2\kappa B$ and $U_j(B):=(2\kappa)^{j+1}B\setminus (2\kappa)^j B$ for all $j\in\mathbb{Z}^+$.}
	\end{equation}
	
	\begin{definition}\label{10:59, 05/12/2023}
		Let $q\in (1,\infty]$ and $\varepsilon>0$. 
		\begin{enumerate}[\rm (i)]
			\item A measurable function $\mathfrak a$ is called an $(H^1_\rho,q)$-atom of log-type related to the ball $B=B(x_0,r)$ if 
			\begin{enumerate}[\rm (a)]
				\item supp\,$\a\subset B$,
				
				\item $\|\a\|_{L^q}\leq \mu(B)^{1/q-1}$,
				
				\item  $\left|\int_{\X}\a(x)d\mu(x)\right|\leq  \frac{1}{\log\left(e+ \frac{\rho(x_0)}{r}\right)}$.
			\end{enumerate} 
			
			\item\label{09:28, 27/11/2023} A measurable function $\mathfrak m$ is called an $(H^1_\rho,q,\varepsilon)$-molecule of log-type related to the ball $ B=B(x_0,r)$ if 
			\begin{enumerate}[\rm (a)]
				\item $\|\m\|_{L^q(U_j(B))}\leq (2\kappa)^{-j\varepsilon} \mu((2\kappa)^j B)^{1/q-1}$ for all $j\in\mathbb{Z}^+_0$,
				
				\item  $\left|\int_{\X}\m(x)d\mu(x)\right|\leq  \frac{1}{\log\left(e+ \frac{\rho(x_0)}{r}\right)}$.
			\end{enumerate}
		\end{enumerate}
	\end{definition}

	\begin{remark}\label{11:50, 05/12/2023}
		\begin{enumerate}[\rm (i)]
			\item\label{14:29, 12/10/2023} If $\a$ is an $(H^1_\rho,q)$-atom to the ball $B$, then $\frac{1}{\log(e+4)}\a$ is an $(H^1_\rho,q)$-atom of log-type related to the ball $B$.
			
			\item\label{11:51, 05/12/2023} If $\a$ is an $(H^1_\rho,q)$-atom of log-type related to the ball $B$, then $\a$ is an $(H^1_\rho,q,\varepsilon)$-molecule of log-type related to the ball $B$.
		\end{enumerate}
	\end{remark}
	
	The following result is due to Ky \cite[Lemma 3.2]{Ky25}.
	
	\begin{lemma}\label{07:01, 14/10/2023}
		Let $q\in (1,\infty]$ and $\varepsilon>0$. Then there exists a  constant $C>0$ such that, for every $(H^1_\rho,q,\varepsilon)$-molecule of log-type $\m$,
		$$\|\m\|_{H^1_\rho}\leq C.$$
	\end{lemma}

	\begin{lemma}\label{14:50, 06/10/2023}
		Let $p\in [1,\infty)$, $\theta\in [0,\infty)$, and let $k_0>i(\rho)$ be such that \eqref{admissible function} holds. Then there exists a constant $C>0$ such that
		$$\left(\frac{1}{\mu((2\kappa)^j B)}\int_{(2\kappa)^j B} |f(x)- f_{B}|^p d\mu(x)\right)^{1/p}\leq C (j+1)\frac{\left(1+ \frac{(2\kappa)^j r}{\rho(x_0)}\right)^{(k_0+1)\theta}}{\log\left(e +\frac{\rho(x_0)}{(2\kappa)^j r}\right)}\|f\|_{\mathrm{BMO}^{\rm log}_{\rho,\theta}}$$
		for all $f\in \mathrm{BMO}^{\rm log}_{\rho,\theta}(\X)$, all balls $B= B(x_0,r)$, and all $j\in \mathbb Z^+_0$. 
	\end{lemma}
	
	\begin{proof}
		We first claim that there exists a constant $C>0$ such that
		\begin{equation}\label{09:24, 07/12/2023}
			\left(\frac{1}{\mu(B_1)}\int_{B_1} |f(x)- f_{B_1}|^p d\mu(x)\right)^{1/p}\leq C \frac{\left(1+ \frac{r_1}{\rho(x_1)}\right)^{(k_0+1)\theta}}{\log\left(e +\frac{\rho(x_1)}{r_1}\right)}\|f\|_{\mathrm{BMO}^{\rm log}_{\rho,\theta}}
		\end{equation}
		for all balls $B_1= B(x_1,r_1)$. Assume that \eqref{09:24, 07/12/2023} holds for a moment. Then Lemma \ref{14:50, 06/10/2023} holds when $j=0$. For any $j\in \mathbb{Z}^+$, it follows from \eqref{09:24, 07/12/2023} that
		\begin{align*}
			&\left(\frac{1}{\mu((2\kappa)^j B)}\int_{(2\kappa)^j B} |f(y)- f_{B}|^p d\mu(y)\right)^{1/p} \\
			&\hskip0.5cm\leq \left(\frac{1}{\mu((2\kappa)^j B)}\int_{(2\kappa)^j B} |f(y)- f_{(2\kappa)^j B}|^p d\mu(y)\right)^{1/p} + \sum_{i=0}^{j-1} \left|f_{(2\kappa)^{i+1}B}- f_{(2\kappa)^i B}\right|\\
			&\hskip0.5cm\lesssim \frac{\left(1+ \frac{(2\kappa)^j r}{\rho(x_0)}\right)^{(k_0+1)\theta}}{\log\left(e +\frac{\rho(x_0)}{(2\kappa)^j r}\right)}\|f\|_{\mathrm{BMO}^{\rm log}_{\rho,\theta}}+ \sum_{i=0}^{j-1} \frac{\mu((2\kappa)^{i+1}B)}{\mu((2\kappa)^i B)} \frac{\left(1+ \frac{(2\kappa)^{i+1} r}{\rho(x_0)}\right)^{(k_0+1)\theta}}{\log\left(e +\frac{\rho(x_0)}{(2\kappa)^{i+1} r}\right)}\|f\|_{\mathrm{BMO}^{\rm log}_{\rho,\theta}}\\
			&\hskip0.5cm\lesssim (j+1) \frac{\left(1+ \frac{(2\kappa)^j r}{\rho(x_0)}\right)^{(k_0+1)\theta}}{\log\left(e + \frac{\rho(x_0)}{(2\kappa)^j r}\right)}\|f\|_{\mathrm{BMO}^{\rm log}_{\rho,\theta}},
		\end{align*}
		which completes our proof.

		It suffices to  prove \eqref{09:24, 07/12/2023}. To do this, define $h:\X\to [0,1]$ by
		\begin{equation*}
			h(x)=
			\begin{cases}
				1, & x\in B_1,\\
				\frac{2 r_1- d(x,x_1)}{r_1},  & x\in 2 B_1\setminus B_1,\\
				0, & x\notin 2 B_1.
			\end{cases}
		\end{equation*}
		It follows from \cite[Lemma 3.2]{Na} that
		\begin{equation}\label{09:07, 07/12/2023}
			|h(x)- h(y)|\lesssim \left(\frac{d(x,y)}{r_1}\right)^{\alpha_0}
		\end{equation}
		for all $x,y\in\X$,	where $\alpha_0$ is the geometric constant in \eqref{Macias-Segovia}. Let $\widetilde f:= f- f_{2 B_1}$. Then,  \cite[Theorem 3.1]{Na} yields
		\begin{align*}
			\left(\frac{1}{\mu(B_1)}\int_{B_1} |f(x)- f_{B_1}|^p d\mu(x)\right)^{1/p} &= \left(\frac{1}{\mu(B_1)}\int_{B_1} |h(x)\widetilde f(x)- (h\widetilde f)_{B_1}|^p d\mu(x)\right)^{1/p}\\
			& \lesssim \|h\widetilde f\|_{\mathrm{BMO}}.
		\end{align*}
		Thus, \eqref{09:24, 07/12/2023} reduces to proving that
		$$\|h\widetilde f\|_{\mathrm{BMO}}\lesssim \frac{\left(1+ \frac{r_1}{\rho(x_1)}\right)^{(k_0+1)\theta}}{\log\left(e + \frac{\rho(x_1)}{r_1}\right)}\|f\|_{\mathrm{BMO}^{\rm log}_{\rho,\theta}}.$$
		Namely, 
		\begin{equation}\label{09:43, 07/12/2023}
			\mathrm{MO}(h\widetilde f, B_2) \lesssim \frac{\left(1+ \frac{r_1}{\rho(x_1)}\right)^{(k_0+1)\theta}}{\log\left(e + \frac{\rho(x_1)}{r_1}\right)}\|f\|_{\mathrm{BMO}^{\rm log}_{\rho,\theta}}\quad\text{for all balls } B_2= B(x_2,r_2).
		\end{equation}
		
		Since $h(x)=0$ for all $x\notin 2 B_1$, to prove \eqref{09:43, 07/12/2023} it suffices to consider the case $B_2\cap (2 B_1)\ne\emptyset$. To this end, we consider the following two cases:
		
		{\sl Case 1:} $r_2> r_1$. Then $B_2\cap (2 B_1)\ne\emptyset$ implies that $2 B_1\subset 5\kappa^2 B_2$. Therefore, using \eqref{09:07, 07/12/2023} and the fact that $h(x)=0$ for all $x\notin 2 B_1$, we obtain
		\begin{align*}
			\mathrm{MO}(h\widetilde f, B_2) &\leq  2 \frac{1}{\mu(B_2)}\int_{B_2} |h(x) \widetilde f(x)|d\mu(x)\\
			&\leq 2 \frac{\mu(5\kappa^2 B_2)}{\mu(B_2)} \frac{1}{\mu(5\kappa^2 B_2)}\int_{5\kappa^2 B_2} |h(y) \widetilde f(y)|d\mu(y)\\
			&\lesssim  \frac{1}{\mu(2 B_1)}\int_{2 B_1}|f(x)- f_{2 B_1}|d\mu(x)\\
			&\lesssim \frac{\left(1+ \frac{2r_1}{\rho(x_1)}\right)^{\theta}}{\log\left(e +\frac{\rho(x_1)}{2 r_1}\right)}\|f\|_{\mathrm{BMO}^{\rm log}_{\rho,\theta}} \lesssim \frac{\left(1+ \frac{r_1}{\rho(x_1)}\right)^{(k_0+1)\theta}}{\log\left(e +\frac{\rho(x_1)}{ r_1}\right)}\|f\|_{\mathrm{BMO}^{\rm log}_{\rho,\theta}}.
		\end{align*}
		
		{\sl Case 2:} $r_1 \geq r_2$. Then, since $B(x_2,r_2)\cap B(x_1,2 r_1)\ne \emptyset$ and by \eqref{admissible function}, we get		
		\begin{equation}\label{10:31, 07/12/2023}
			B(x_1,2 r_1)\subset B(x_2, 5\kappa^2 r_1)\subset B(x_1, 10\kappa^3 r_1)
		\end{equation}
		and
		$$\frac{r_2}{\rho(x_2)}\leq \frac{r_1}{\rho(x_2)}\lesssim \frac{r_1}{\rho(x_1)}\left(1+ \frac{d(x_2,x_1)}{\rho(x_1)}\right)^{k_0}\lesssim \left(1+ \frac{r_1}{\rho(x_1)}\right)^{k_0+1}.$$
		Consequently,
		\begin{align}\label{10:32, 07/12/2023}
			\mathrm{MO}(f,B(x_2,r_2)) \leq \frac{\left(1+ \frac{r_2}{\rho(x_2)}\right)^{\theta}}{\log \left(e+ \frac{\rho(x_2)}{r_2}\right)}\|f\|_{\mathrm{BMO}^{\rm log}_{\rho,\theta}}
			&\lesssim \frac{\left(1+ \left(1+ \frac{r_1}{\rho(x_1)}\right)^{k_0+1}\right)^{\theta}}{\log\left(e +\frac{\rho(x_1)}{r_1}\left(1+\frac{r_1}{\rho(x_1)}\right)^{-k_0}\right)}\|f\|_{\mathrm{BMO}^{\rm log}_{\rho,\theta}}\nonumber\\
			&\lesssim \frac{\left(1+ \frac{r_1}{\rho(x_1)}\right)^{(k_0+1)\theta}}{\log\left(e +\frac{\rho(x_1)}{r_1}\right)}\|f\|_{\mathrm{BMO}^{\rm log}_{\rho,\theta}},
		\end{align}
		and
		\begin{align}\label{10:33, 07/12/2023}
			\mathrm{MO}(f, B(x_2,5\kappa^2 r_1)) \leq \frac{\left(1+ \frac{5\kappa^2 r_1}{\rho(x_2)}\right)^{\theta}}{\log \left(e+ \frac{\rho(x_2)}{5\kappa^2 r_1}\right)}\|f\|_{\mathrm{BMO}^{\rm log}_{\rho,\theta}} &\lesssim \frac{\left(1+ \left(1+ \frac{r_1}{\rho(x_1)}\right)^{k_0+1}\right)^{\theta}}{\log\left(e +\frac{\rho(x_1)}{r_1}\left(1+\frac{r_1}{\rho(x_1)}\right)^{-k_0}\right)}\|f\|_{\mathrm{BMO}^{\rm log}_{\rho,\theta}}\nonumber\\
			& \lesssim \frac{\left(1+ \frac{r_1}{\rho(x_1)}\right)^{(k_0+1)\theta}}{\log\left(e +\frac{\rho(x_1)}{r_1}\right)}\|f\|_{\mathrm{BMO}^{\rm log}_{\rho,\theta}}.
		\end{align}
		Thus, since $r_1\geq r_2$, a standard argument yields
		\begin{equation}\label{11:16, 07/12/2023}
			|f_{B(x_2,r_2)}- f_{B(x_2, 5\kappa^2 r_1)}|\lesssim \log\left(e+ \frac{r_1}{r_2}\right) \frac{\left(1+ \frac{r_1}{\rho(x_1)}\right)^{(k_0+1)\theta}}{\log\left(e +\frac{\rho(x_1)}{r_1}\right)}\|f\|_{\mathrm{BMO}^{\rm log}_{\rho,\theta}}.
		\end{equation}
		
		Finally, combining \eqref{09:07, 07/12/2023}  and \eqref{10:31, 07/12/2023}--\eqref{11:16, 07/12/2023}, we obtain that
		\begin{align*}\label{log-generalized BHS 3}
			\mathrm{MO}(h \widetilde f, B_2)	&\leq 2 \frac{1}{\mu(B_2)}\int_{B_2} |h(x) \widetilde f(x)- h_{B_2} \widetilde f_{B_2}|d\mu(x) \\
			&\leq 2 \frac{1}{\mu(B_2)}\int_{B_2} |h(x)(\widetilde f(x)- \widetilde f_{B_2})| d\mu(x)+\\
			&\hskip1cm + 2 |\widetilde f_{B_2}| \frac{1}{\mu(B_2)^2} \int_{B_2} \int_{B_2} |h(y)- h(x)|d\mu(x)d\mu(y)\\
			& \lesssim \frac{1}{\mu(B_2)}\int_{B_2} |f(x)- f_{B_2}|d\mu(x) +  \left(\frac{r_2}{r_1}\right)^{\alpha_0} |f_{B_2} - f_{2 B_1}|\\
			&\lesssim \mathrm{MO}(f,B_2) +   \left(\frac{r_2}{r_1}\right)^{\alpha_0} |f_{B(x_2,r_2)} - f_{B(x_2,5\kappa^2 r_1)}| +\\
			&\hskip1cm  + \left(\frac{r_2}{r_1}\right)^{\alpha_0} |f_{B(x_2,5\kappa^2 r_1)}- f_{B(x_1,2 r_1)}|\\
			&\lesssim \left(1+ \left(\frac{r_2}{r_1}\right)^{\alpha_0} \log\left(e+ \frac{r_1}{r_2}\right)\right) \frac{\left(1+ \frac{r_1}{\rho(x_1)}\right)^{(k_0+1)\theta}}{\log\left(e +\frac{\rho(x_1)}{r_1}\right)}\|f\|_{\mathrm{BMO}^{\rm log}_{\rho,\theta}} + \\
			&\hskip1cm + \frac{\mu(B(x_2, 5\kappa^2 r_1))}{\mu(B(x_1,10\kappa^3 r_1))}\frac{\mu(B(x_1,10\kappa^3 r_1))}{\mu(B(x_1,2 r_1))} \mathrm{MO}(f, B(x_2,5\kappa^2 r_1))\\
			&\lesssim \frac{\left(1+ \frac{r_1}{\rho(x_1)}\right)^{(k_0+1)\theta}}{\log\left(e +\frac{\rho(x_1)}{r_1}\right)}\|f\|_{\mathrm{BMO}^{\rm log}_{\rho,\theta}},
		\end{align*}
		where in the last inequality we used the estimate $$\left(\frac{r_2}{r_1}\right)^{\alpha_0}\log\left(e+ \frac{r_1}{r_2}\right)\leq \sup_{0<t\leq 1}t^{\alpha_0}\log\left(e+ \frac{1}{t}\right)\lesssim 1. 
		$$
		This proves \eqref{09:43, 07/12/2023}, and thus completes the proof of Lemma \ref{14:50, 06/10/2023}.
	\end{proof}

	Since the proof is similar to that of Lemma \ref{14:50, 06/10/2023} but uses an easier argument, we obtain the following result, whose proof is omitted.
	
	\begin{lemma}\label{14:51, 06/10/2023}
		Let $p\in [1,\infty)$, $\theta\in [0,\infty)$, and let $k_0>i(\rho)$ be such that \eqref{admissible function} holds. Then, there exists a constant $C>0$ such that
		$$\left(\frac{1}{\mu((2\kappa)^j B)}\int_{(2\kappa)^j B} |f(x)- f_{B}|^p d\mu(x)\right)^{1/p}\leq C (j+1) \left(1+ \frac{(2\kappa)^j r}{\rho(x_0)}\right)^{(k_0+1)\theta}\|f\|_{\mathrm{BMO}_{\rho,\theta}}$$
		for all $f\in \mathrm{BMO}_{\rho,\theta}(\X)$, all balls $B= B(x_0,r)$, and all $j\in \mathbb Z^+_0$. 
	\end{lemma}
	
	\begin{lemma}\label{17:53, 11/10/2023}
		Let $\theta,\eta\in [0,\infty)$, $s\in [1,\infty)$, and $q\in (s,\infty]$. Then, there exists a constant $C>0$ such that for all $(H^1_\rho,q)$-atoms $\a$ related to the balls $B$, we have
		\begin{enumerate}[\rm (i)]
			\item\label{14:01, 12/10/2023} for all $f\in \mathrm{BMO}_{\rho,\theta}(\X)$, 
			$$\|\a(f-f_B)\|_{L^s}\leq C\mu(B)^{\frac{1}{s}-1}\|f\|_{\mathrm{BMO}_{\rho,\theta}};$$
			
			\item\label{14:00, 12/10/2023} for all $f\in \mathrm{BMO}_{\rho,\theta}(\X)$ and all $g\in \mathrm{BMO}_{\rho,\eta}(\X)$, 
			$$\|\a(f-f_B)(g-g_B)\|_{L^1}\leq C \|f\|_{\mathrm{BMO}_{\rho,\theta}}\|g\|_{\mathrm{BMO}_{\rho,\eta}}.$$
		\end{enumerate}
	\end{lemma}
	
	\begin{proof}
		Since $\a$ is an $(H^1_\rho,q)$-atom related to the ball $B=B(x_0,r)$, we have
		\begin{equation}\label{13:45, 12/10/2023}
			\text{supp\,$\a\subset B$\quad and \quad $0<r<\rho(x_0)$}.
		\end{equation}	
		
		\eqref{14:01, 12/10/2023} Let $p\in [s,\infty)$ be such that $\frac{1}{s}=\frac{1}{q}+\frac{1}{p}$. Then, by the   H\"older inequality, Lemma \ref{14:51, 06/10/2023} and \eqref{13:45, 12/10/2023}, 
		\begin{align*}
			\|\a(f-f_B)\|_{L^s}&=\|\a(f-f_B)\|_{L^s(B)}\\
			&\leq \|\a\|_{L^q(B)}\|f-f_B\|_{L^p(B)}\\
			&\lesssim \mu(B)^{\frac{1}{q}-1}\mu(B)^{\frac{1}{p}}\|f\|_{\mathrm{BMO}_{\rho,\theta}}= \mu(B)^{\frac{1}{s}-1}\|f\|_{\mathrm{BMO}_{\rho,\theta}}.
		\end{align*}
		
		\eqref{14:00, 12/10/2023} It follows, from the  H\"older inequality, Lemma \ref{14:51, 06/10/2023} and \eqref{13:45, 12/10/2023}, that
		\begin{align*}
			\|\a(f-f_B)(g-g_B)\|_{L^1}&=\|\a(f-f_B)(g-g_B)\|_{L^1(B)}\\
			&\leq \|\a\|_{L^q(B)}\|f-f_B\|_{L^{2q'}(B)}\|g-g_B\|_{L^{2q'}(B)}\\
			&\lesssim \mu(B)^{\frac{1}{q}-1} \mu(B)^{\frac{1}{2q'}}\|f\|_{\mathrm{BMO}_{\rho,\theta}} \mu(B)^{\frac{1}{2q'}}\|g\|_{\mathrm{BMO}_{\rho,\eta}}\\
			&\lesssim \|f\|_{\mathrm{BMO}_{\rho,\theta}}\|g\|_{\mathrm{BMO}_{\rho,\eta}}.
		\end{align*}
		
	\end{proof}

	For any $x_0\in\X$, we define $g_{x_0}:\X\to \R$ by
	\begin{equation}\label{17:18, 11/10/2023}
		g_{x_0}(x)=\max\left\{0, 1+ \log\frac{\rho(x_0)}{d(x,x_0)}\right\}.
	\end{equation}
	Then, we have the following (see \cite[Lemma 3.5]{Ky25}).
	
	\begin{lemma}\label{17:37, 11/10/2023}
		There exists a constant $C>0$ such that 
		$$\|g_{x_0}\|_{\mathrm{BMO}_\rho}\leq C\quad\text{for all } x_0\in\X.$$
		Moreover, if $0<r<\rho(x_0)$ then $\log\left(e+\frac{\rho(x_0)}{r}\right)\leq 2\, g_{x_0}(x)$ for all $x\in B(x_0,r)$.
	\end{lemma}

	\begin{lemma}\label{07:01, 12/10/2023}
		Let $q\in (1,\infty]$, $\theta\in [0,\infty)$, and $f\in \mathrm{BMO}_{\rho,\theta}(\X)$. Then, $f\in \mathrm{BMO}^{\log}_{\rho,\theta}(\X)$ if and only if $\a(f-f_{B})\in H^1_\rho(\X)$ for all $(H^1_\rho,q)$-atoms $\a$ related to the balls $B$. Moreover,
		$$\|f\|_{\mathrm{BMO}_{\rho,\theta}^{\log}}\simeq \|f\|_{\mathrm{BMO}_{\rho,\theta}}+ \sup_{\text{\rm $(H^1_\rho,q)$-atoms $\a$ related to $B$}} \|\a(f-f_{B})\|_{H^1_\rho}.$$		
	\end{lemma}
	
	\begin{proof}
		Suppose that $f\in \mathrm{BMO}_{\rho,\theta}^{\log}(\X)$. Let $s:= \min\{2,\frac{q+1}{2}\}\in (1,q)$. Then, for any $(H^1_\rho,q)$-atom $\a$ related to the ball $B=B(x_0,r)$, Lemma \ref{17:53, 11/10/2023}\eqref{14:01, 12/10/2023} gives
		\begin{equation}\label{06:55, 12/10/2023}
			\|\a(f-f_{B})\|_{L^{s}}\lesssim \|f\|_{\mathrm{BMO}_{\rho,\theta}}\mu(B)^{\frac{1}{s}-1}\lesssim \|f\|_{\mathrm{BMO}_{\rho,\theta}^{\log}}\mu(B)^{\frac{1}{s}-1}.
		\end{equation}
		Moreover, by the H\"older inequality and Lemma \ref{14:50, 06/10/2023}, we obtain
		\begin{align*}
			\left|\int_{\X} \a(x)(f(x)-f_B)d\mu(x)\right|&\leq \|\a(f-f_B)\|_{L^1(B)}\\
			&\leq \|\a\|_{L^q(B)}\|f-f_B\|_{L^{q'}(B)}\\
			&\lesssim \mu(B)^{\frac{1}{q}-1} \mu(B)^{\frac{1}{q'}} \frac{1}{\log\left(e+\frac{\rho(x_0)}{r}\right)}\|f\|_{\mathrm{BMO}_{\rho,\theta}^{\log}}\\
			&\lesssim \|f\|_{\mathrm{BMO}_{\rho,\theta}^{\log}}\frac{1}{\log\left(e+\frac{\rho(x_0)}{r}\right)}.
		\end{align*}
		This, together with \eqref{06:55, 12/10/2023} and supp$\,\left(\a(f-f_B)\right)\subset B$, implies that $\a(f-f_B)$  is, up to a multiplicative constant bounded by
		$C\|f\|_{\mathrm{BMO}_{\rho,\theta}^{\log}}$, an $(H^1_\rho,s)$-atom of log-type related to the ball $B$. Consequently, by Remark \ref{11:50, 05/12/2023}\eqref{11:51, 05/12/2023} and Lemma \ref{07:01, 14/10/2023}, we have		
		$$\|\a(f-f_B)\|_{H^1_\rho}\lesssim \|f\|_{\mathrm{BMO}_{\rho,\theta}^{\log}}.$$
		Thus,
		$$\|f\|_{\mathrm{BMO}_{\rho,\theta}}+ \sup_{\text{$(H^1_\rho,q)$-atoms $\a$ related to $B$}} \|\a(f-f_{B})\|_{H^1_\rho}\lesssim \|f\|_{\mathrm{BMO}_{\rho,\theta}^{\log}}.$$
		
		Conversely, let $f\in \mathrm{BMO}_{\rho,\theta}(\X)$ such  that $\a(f-f_{B})\in H^1_\rho(\X)$ for all $(H^1_\rho,q)$-atoms $\a$ related to the balls $B$. It suffices to prove that
		\begin{equation}\label{12:10, 04/12/2023}
			\frac{\log\left(e+\frac{\rho(x_0)}{r}\right)}{\left(1+\frac{r}{\rho(x_0)}\right)^{\theta}} \mathrm{MO}(f, B_0)\lesssim \|f\|_{\mathrm{BMO}_{\rho,\theta}}+ \sup_{\text{$(H^1_\rho,q)$-atoms $\a$ related to $B$}} \|\a(f-f_{B})\|_{H^1_\rho}
		\end{equation}
		for all balls $B_0=B(x_0,r)$. 
		
		Indeed, the case of $r\geq \rho(x_0)$ is trivial since $\frac{\mathrm{MO}(f, B_0)}{\left(1+\frac{r}{\rho(x_0)}\right)^{\theta}}\leq \|f\|_{\mathrm{BMO}_{\rho,\theta}}$. When $0<r<\rho(x_0)$, let $g_{x_0}:\X\to \R$ be as in \eqref{17:18, 11/10/2023} and define $$\a_0=\frac{1}{2\mu(B_0)}(h-h_{B_0})\chi_{B_0}$$ 
		with $h=\sign (f-f_{B_0})$, then it is easy to see that $\a_0$ is an $(H^1_\rho,q)$-atom related to the ball $B_0$. Moreover, by Lemma \ref{17:37, 11/10/2023}, Lemma \ref{17:53, 11/10/2023}\eqref{14:00, 12/10/2023}, and \cite[Theorem 2.1]{YYZ}, we obtain 
		\begin{align*}
			\frac{\log\left(e+\frac{\rho(x_0)}{r}\right)}{\left(1+\frac{r}{\rho(x_0)}\right)^{\theta}} \mathrm{MO}(f, B_0)&=2 \frac{\log\left(e+\frac{\rho(x_0)}{r}\right)}{\left(1+\frac{r}{\rho(x_0)}\right)^{\theta}} \int_{\X} \a_0(x) (f(x)-f_{B_0}) d\mu(x)\\
			&\leq 4 (g_{x_0})_{B_0} \int_{\X} \a_0(x) (f(x)-f_{B_0}) d\mu(x)\\
			&\lesssim \int_{\X} |\a_0(x)| |f(x)-f_{B_0}| |g_{x_0}(x)- (g_{x_0})_{B_0}| d\mu(x) +\\
			&\hskip2cm +\left|\int_{\X} \a_0(x) (f(x)-f_{B_0}) g_{x_0}(x) d\mu(x)\right| \\
			&\lesssim \|f\|_{\mathrm{BMO}_{\rho,\theta}}\|g_{x_0}\|_{\mathrm{BMO}} + \|\a_0(f-f_{B_0})\|_{H^1_\rho} \|g_{x_0}\|_{\mathrm{BMO}_\rho}\\
			&\lesssim \|f\|_{\mathrm{BMO}_{\rho,\theta}} + \sup_{\text{$(H^1_\rho,q)$-atom $\a$ related to $B$}} \|\a(f-f_{B})\|_{H^1_\rho}.
		\end{align*}
		This proves \eqref{12:10, 04/12/2023}, and thus completes the proof of Lemma \ref{07:01, 12/10/2023}.
		
	\end{proof}

	\section{Proofs of the main results}\label{17:04, 30/12/2025}

	In this section, we present the proofs of the main results, Theorems \ref{thm-main thm 1 from H1 to weak L1}--\ref{thm-main thm 4 chracterize BMO log by Tj on Hardy space}. We begin by recalling the following result in \cite[Lemma 10.1]{AH}.

\begin{lemma}\label{lem:exp-ATI}
	There exist constants
	$C,\nu\in(0,\infty)$, $a\in(0,1]$, and $\Lambda, \sigma\in(0,1)$  such that, for any $k\in\mathbb{Z}$, the kernel of the operator $S_k$, still denoted by $S_k$ and viewed as a function on $\X\times \X$, satisfies the following properties:
	
	\begin{enumerate}[\rm (i)]

		\item
		For any $x,y\in \X$,
		\[
		|S_k(x,y)|
		\le C \,
		\frac{1}{\mu(B(x,\Lambda^k))\mu(B(y,\Lambda^k))}
		\exp\!\left[-\nu\left(\frac{d(x,y)}{\Lambda^k}\right)^a\right]=:E_k(x,y).
		\]
		
		\item
		For any $x,x',y\in \X$ with $d(x,x')\le \Lambda^k$,
		\[
		|S_k(x,y)-S_k(x',y)| + |S_k(y,x)-S_k(y,x')|
		\le C \left(\frac{d(x,x')}{\Lambda^k}\right)^{\sigma} E_k(x,y).
		\]
		
		
		\item
		For any $x,y\in \X$,
		\[
		\int_{\X} S_k(x,z)\,d\mu(z) = 1
		\quad \text{and} \quad
		\int_{\X} S_k(z,y)\,d\mu(z) = 1 .
		\]
	\end{enumerate}
\end{lemma}

We define
\begin{equation}\label{def:sigma0}
	\sigma_0:=\sup\left\{\sigma\in (0,1): \text{the conclusions of Lemma \ref{lem:exp-ATI} hold with exponent } \sigma\right\}.
\end{equation}

\begin{remark}
	Without loss of generality, we may assume that $\Lambda = 1/2$. Indeed, for each $k\in \mathbb Z$, there exists a unique $\eta(k)\in \mathbb Z$ such that
	\[
	\Lambda^{k+1}\le 2^{-\eta(k)} < \Lambda^{k}.
	\]
	Then, the family $\{\widetilde S_{k}:  \widetilde S_{k} = S_{\eta(k)}, k\in \mathbb Z\}$ satisfies all conditions (i)-(iii) in Lemma \ref{lem:exp-ATI} with $\Lambda = 1/2$.
	
	Furthermore, when $\X$ is a metric space, the exponent $\sigma\in (0,1)$ in Lemma \ref{lem:exp-ATI} can be chosen arbitrarily (see, for example, \cite{AH}). In this case, we have $\sigma_0=1$.
\end{remark}

	\begin{definition}
		Let $\{S_k\}_{k\in \mathbb Z}$ be the family as in Lemma \ref{lem:exp-ATI}. For $f\in L^1_{\rm loc}(\X)$, the radial maximal function $S^+_{\rho}(f)$ associated with $\rho$ is defined by
		$$S^+_{\rho}(f)(x)=\sup_{\{i\in\mathbb Z: 2^{-i}<\rho(x)\}} \left|S_i(f)(x)\right|,\quad x\in \X .$$
	\end{definition}

	\begin{remark}\label{20:52, 19/12/2023}
	It is straightforward to see that 
	\[
	S^+_{\rho}(f)\lesssim \mathcal Mf \quad \text{for all } f\in L^1_{\rm loc}(\X),
	\]
	where $\mathcal M$ is the classical Hardy--Littlewood maximal function.
	\end{remark}
	
	The following characterization is taken from \cite[Remark 5.8]{BDK18} (see also \cite[Theorem 4.1]{YZ}).
	
	\begin{theorem}\label{21:55, 19/12/2023}
		Let $\{S_k\}_{k\in \mathbb Z}$ be the family as in Lemma \ref{lem:exp-ATI}. Then  $f\in H^1_\rho(\X)$ if and only if $f \in L^1(\X)$ and  $S^+_{\rho}(f)\in L^1(\X)$. Moreover,
		$$\|f\|_{H^1_\rho}\simeq \|S^+_{\rho}(f)\|_{L^1}\quad\text{for all } f\in H^1_\rho(\X).$$ 
	\end{theorem}
	
	The following proposition plays a key role in the proof of Theorem \ref{thm-main thm 2 characterization BMO log}.

	\begin{proposition}\label{prop:radial-maximal-Krho-class}		
		Let $\sigma_0$ be as defined in \eqref{def:sigma0} and $\theta\in\big[0, \frac{\sigma_0}{i(\rho)+1}\big)$. Then, there exists an exponent $\sigma \in (0,\sigma_0)$ such that $\theta < \frac{\sigma}{i(\rho)+1}$ and a family $\{S_k\}_{k\in \mathbb Z}$ satisfies the conclusions of Lemma \ref{lem:exp-ATI} with this $\sigma$. Moreover, the radial maximal function operator $S^+_{\rho}$ associated with $\{S_k\}_{k\in \mathbb Z}$ belongs to the class $\K_{\rho,\theta,q}$ for all $q\in (1,\infty]$.
	\end{proposition}
	
	\begin{proof}
		By the definition of $\sigma_0$ in \eqref{def:sigma0}, there exists an exponent $\sigma \in (0,\sigma_0)$ such that $\theta < \frac{\sigma}{i(\rho)+1}$ and a family $\{S_k\}_{k\in \mathbb Z}$ satisfies the conclusions of Lemma \ref{lem:exp-ATI} with this $\sigma$.
		
		We now prove that $S^+_{\rho}\in \K_{\rho,\theta,q}$. Indeed, by Theorem \ref{21:55, 19/12/2023},  $S^+_{\rho}$ is bounded from $H^1_\rho(\X)$ into $L^1(\X)$. Hence, it suffices to prove that
		\begin{equation}\label{21:41, 19/12/2023}
			\|(b-b_B)S^+_{\rho}(\a)\|_{L^1}\lesssim \|b\|_{\mathrm{BMO}_{\rho,\theta}}
		\end{equation}
		for all $b\in \mathrm{BMO}_{\rho,\theta}(\X)$ and all $(H^1_\rho,q)$-atoms $\a$ related to the balls $B$. To see this, note that since $0\leq\theta<\frac{\sigma}{i(\rho)+1}$ and $\a$ is an $(H^1_\rho,q)$-atom related to the ball $B=B(x_0,r_0)$, there exists $k_0>i(\rho)$ such that \eqref{admissible function} holds and
		\begin{equation}\label{21:42, 19/12/2023}		
			\varepsilon:=  \sigma  -(k_0+1)\theta>0,\quad 0<r_0<\rho(x_0).
		\end{equation}		
		For any $j\in\mathbb Z^+$ and $x\in U_j(B)$, we consider the following two cases:
		
	\medskip

	\noindent 	\textit{ Case 1: $0<r_0<\frac{\rho(x_0)}{4}$.}  In this situation,   $\displaystyle \int_{\X} \a(y)d\mu(y)=0$. Therefore, by Lemma \ref{lem:exp-ATI}(ii), $d(x,y)\simeq d(x,x_0)\simeq (2\kappa)^j r_0$, and $$\mu(B(x,d(x,y)))\simeq\mu(B(x, d(x,x_0)))\simeq \mu(B(x_0, d(x,x_0)))\simeq \mu((2\kappa)^j B) \ \ \text{for all $y\in B$,}
	$$ we obtain
		\begin{align*}
			\left|S_i(\a)(x) \right|&=\left|\int_{B} (S_i(x,y)-S_i(x,x_0))\a(y)d\mu(y)\right|\\
			&\lesssim \int_{B}|\a(y)|  \left(\frac{d(y,x_0)}{2^{-i}+d(x,y)}\right)^{\sigma} \frac{1}{\mu(B(x, d(x,y)))} d\mu(y)\\
			&\lesssim (2\kappa)^{-j \sigma} \frac{1}{\mu((2\kappa)^j B)}
		\end{align*}
		for all $i\in\mathbb Z$. This implies that
		\begin{equation}\label{21:43, 19/12/2023}
			S^+_{\rho}(\a)(x)=\sup_{\{i\in\mathbb Z: 2^{-i}<\rho(x)\}}|S_i(\a)(x)|\lesssim  \frac{(2\kappa)^{-j \sigma}}{\mu((2\kappa)^j B)}.
		\end{equation}
		
\medskip

	\noindent	\textit{{Case 2:} $\frac{\rho(x_0)}{4}\leq r_0<\rho(x_0)$.}  Then, $d(x,y)\simeq d(x,x_0)\simeq (2\kappa)^j r_0\simeq (2\kappa)^j \rho(x_0)$ for all $y\in B=B(x_0,r_0)$. Therefore, by Lemma \ref{lem:exp-ATI}(i), the fact $$\mu(B(x, d(x,x_0)))\simeq \mu(B(x_0, d(x,x_0)))\simeq \mu((2\kappa)^j B),$$ and Remark \ref{20:49, 02/12/2023}, we obtain, for $2^{-i}<\rho(x)$,
	\begin{align*}
		\left|S_i(\a)(x) \right|&=\left|\int_{B} S_i(x,y)\a(y) d\mu(y)\right|\\
		&\lesssim \int_{B}|\a(y)|  \frac{1}{\mu(B(x, d(x,y)))} \left(\frac{2^{-i}}{2^{-i}+d(x,y)}\right)^{N} d\mu(y),
	\end{align*}
	where $N>0$ is a positive large number which will be fixed later.
	
	This, together with the fact $2^{-i}\le \rho(x)$, further implies 
		\begin{align*}
			\left|S_i(\a)(x) \right| &\lesssim \int_{B}|\a(y)| \frac{1}{\mu(B(x, d(x,x_0)))} \left(\frac{1}{1+\frac{d(x,x_0)}{\rho(x)}}\right)^{N} d\mu(y) \\
			&\lesssim \frac{1}{\mu((2\kappa)^j B)} \left(1+\frac{d(x,x_0)}{\rho(x_0)}\right)^{-\frac{N}{k_0+1}} \ \ \ \text{(since $\|\mathfrak{a}\|_{L^1(\X)}\le 1$)}\\
			& \lesssim(2\kappa)^{-j\frac{N}{k_0+1}} \frac{1}{\mu((2\kappa)^j B)}.
		\end{align*}

		Hence,
		\begin{equation}\label{21:44, 19/12/2023}
			S^+_{\rho}(\a)(x)=\sup_{\{i\in\mathbb Z: 2^{-i}<\rho(x)\}}|S_i(\a)(x)|\lesssim  \frac{(2\kappa)^{-j \frac{N}{k_0+1}}}{\mu((2\kappa)^j B)}.
		\end{equation}

		Taking $N$ such that $\f{N}{k_0+1}>\sigma$, by the H\"older inequality, Remark \ref{20:52, 19/12/2023}, the $L^q$-boundedness of $\mathcal M$, \eqref{21:43, 19/12/2023}, \eqref{21:44, 19/12/2023}, Lemma \ref{14:51, 06/10/2023}, and \eqref{21:42, 19/12/2023}, we get
		\begin{align*}
			\|(b-b_B) S^+_{\rho}(\a)\|_{L^1}&=\sum_{j=0}^{\infty} \|(b-b_B) S^+_{\rho}(\a)\|_{L^1(U_j)}\\
			&\lesssim \|b-b_B\|_{L^{q'}(U_0)}\|\mathcal M\a\|_{L^q(U_0)}  + \sum_{j=1}^{\infty} \|(b-b_B)\|_{L^1(U_j)}\frac{(2\kappa)^{-j\sigma}}{\mu((2\kappa)^j B)}\\
			&\lesssim \|b\|_{\mathrm{BMO}_{\rho,\theta}}+ \|b\|_{\mathrm{BMO}_{\rho,\theta}}\sum_{j=1}^{\infty} (j+2) (2\kappa)^{-j\varepsilon}\\
			& \lesssim \|b\|_{\mathrm{BMO}_{\rho,\theta}}.
		\end{align*}
		This proves \eqref{21:41, 19/12/2023}, and thus $S^+_{\rho}$ belongs to the class $\K_{\rho,\theta,q}$.
		
	\end{proof}
	
	We are now ready to give the proofs of Theorems \ref{thm-main thm 1 from H1 to weak L1} and \ref{thm-main thm 2 characterization BMO log}.
	\begin{proof}
		[Proof of Theorem \ref{thm-main thm 1 from H1 to weak L1}] For any $f\in H^1_\rho(\X)$, we can write  $f$  as  $f= \sum_{j=1}^\infty \lambda_j \a_j$, where the $\{a_j\}_{j=1}^\infty$ are $(H^1_\rho,q)$-atoms related to the balls $\{B_j\}_{j=1}^\infty$ and $\sum_{j=1}^{\infty} |\lambda_j|\leq 2 \|f\|_{H^1_\rho}$. Therefore,
		\begin{align*}
			\left|[b,T](f)(x)\right|&= \left|T\left(\left(b(x)-b(\cdot)\right)\sum_{j=1}^\infty \lambda_j \a_j(\cdot)\right)(x)\right|\\
			&\leq \sum_{j=1}^{\infty} |\lambda_j|\left|(b(x)-b_{B_j})T(\a_j)(x)\right| +\left|T\left(\sum_{j=1}^{\infty} \lambda_j (b(x)-b_{B_j})\a_j(x)\right)\right|,
		\end{align*}
		and thus, by $T\in\K_{\rho,\theta,q}$, $T$ is of weak type $(1,1)$ and Lemma \ref{17:53, 11/10/2023}\eqref{14:01, 12/10/2023},
		\begin{align*}
			\|[b,T](f)\|_{L^{1,\infty}}&\leq \sum_{j=1}^{\infty} |\lambda_j| \|(b-b_{B_j})T(\a_j)\|_{L^1} + \|T\|_{L^1\to L^{1,\infty}}\sum_{j=1}^{\infty} |\lambda_j|\|(b-b_{B_j})\a_j\|_{L^1}\\
			&\lesssim \|b\|_{\mathrm{BMO}_{\rho,\theta}}\sum_{j=1}^{\infty} |\lambda_j| + \|T\|_{L^1\to L^{1,\infty}}\|b\|_{\mathrm{BMO}_{\rho,\theta}}\sum_{j=1}^{\infty} |\lambda_j|\\
			&\lesssim \|b\|_{\mathrm{BMO}_{\rho,\theta}} \|f\|_{H^1_\rho}.
		\end{align*}
		This proves that $[b,T]$ is bounded from $H^1_\rho(\X)$ into $L^{1,\infty}(\X)$.
		
	\end{proof}
	
	\begin{proof}[Proof of Theorem \ref{thm-main thm 2 characterization BMO log}]
		Let $\sigma_0\in (0,1]$ be defined by \eqref{def:sigma0}, $\theta\in\big[0,\frac{\sigma_0}{i(\rho)+1}\big)$, $q\in (1,\infty]$, and $b\in \mathrm{BMO}_{\rho,\theta}(\X)$.
		
		\eqref{item:BMO-log-condition} $\Rightarrow $ \eqref{item:commutator-L1-bounded}: Suppose that $b\in \mathrm{BMO}^{\log}_{\rho,\theta}(\X)$. For any $(H^1_\rho,q)$-atom $\a$ related to a ball $B$, it follows from $T\in \K_{\rho,\theta,q}$ and Lemma \ref{07:01, 12/10/2023} that
		\begin{align*}
			\|[b,T](\a)\|_{L^1}&\leq \|(b-b_B)T\a\|_{L^1}+\|T(\a(b-b_B))\|_{L^1}\\
			&\lesssim \|b\|_{\mathrm{BMO}_{\rho,\theta}} + \|T\|_{H^1_\rho\to L^1}\|\a(b-b_B)\|_{H^1_\rho}\lesssim \|b\|_{\mathrm{BMO}^{\log}_{\rho,\theta}}.
		\end{align*}
		This, together with  \cite[Proposition 3.2]{YZ}, allows us to conclude that the commutator $[b,T]$ is bounded from $H^1_\rho(\X)$ into $L^1(\X)$.
		
		\eqref{item:commutator-L1-bounded} $\Rightarrow $ \eqref{item:BMO-log-condition}: Conversely, suppose that the commutator $[b,T]$ is bounded from $H^1_\rho(\X)$ into $L^1(\X)$ for all $T\in \K_{\rho,\theta,q}$. Let $\sigma\in (0,\sigma_0)$ and let $\{S_k\}_{k\in \mathbb Z}$ be as in Proposition \ref{prop:radial-maximal-Krho-class}.
		Then the associated radial maximal operator $S^+_{\rho}$ belongs to $\K_{\rho,\theta,q}$. Hence, by Theorem \ref{21:55, 19/12/2023},
		\begin{align*}
			\|\a(b-b_B)\|_{H^1_\rho}&\lesssim \|S^+_{\rho}(\a(b-b_B))\|_{L^1}\\
			&\leq \|[b,S^+_{\rho}](\a)\|_{L^1}+ \|(b-b_B) S^+_{\rho}(\a)\|_{L^1}\\
			&\lesssim \|[b, S^+_{\rho}]\|_{H^1_\rho\to L^1} + \|b\|_{\mathrm{BMO}_{\rho,\theta}}
		\end{align*}
		for all $(H^1_\rho,q)$-atoms $\a$ related to balls $B$. This, together with Lemma \ref{07:01, 12/10/2023}, implies that
		$$\|b\|_{\mathrm{BMO}^{\log}_{\rho,\theta}}\lesssim \|[b, S^+_{\rho}]\|_{H^1_\rho\to L^1} + \|b\|_{\mathrm{BMO}_{\rho,\theta}}.$$
		This completes our proof.
	\end{proof}

	In order to prove Theorems \ref{thm-main thm 3 chracterize BMO log by Tj} and \ref{thm-main thm 4 chracterize BMO log by Tj on Hardy space}, we also need the following two lemmas.
	
	\begin{lemma}\label{15:47, 09/12/2023}
		Let $1<q<s<\infty$, $\delta\in (0,1]$, $\theta\in [0,\frac{\delta}{i(\rho)+1})$, and let $T$ be an $(s,\delta,\rho)$-Calder\'on--Zygmund operator. Then, there exists a constant $C>0$ such that 
		\begin{enumerate}[\rm (i)]
			\item\label{09:31, 15/11/2025} for all $(H^1_\rho,q)$-atoms $\a$ related to the balls $B=B(x_0,r)$, and all  $j\in\mathbb Z^+_0$,
			\[\|T\a\|_{L^q(U_j(B))}\leq C (2\kappa)^{-j\delta} \mu((2\kappa)^{j}B)^{\frac{1}{q}-1};\]
			
			\item\label{09:32, 15/11/2025} for all $(H^1_\rho,s)$-atoms $\a$ related to the balls $B=B(x_0,r)$, and all $b\in \mathrm{BMO}_{\rho,\theta}(\X)$,
			\[\|(b-b_B)T\a\|_{L^1}\leq C \|b\|_{\mathrm{BMO}_{\rho,\theta}}.\]
		\end{enumerate}
	\end{lemma}

	\begin{proof}	
		\eqref{09:31, 15/11/2025} The case $j=0$ is trivial since $T$ is bounded on $L^q(\X)$ (see Remark \ref{rem 1.4}\eqref{10:19, 22/12/2023}). For $j\in\mathbb Z^+$, we consider the following two cases:
		
		\medskip
		
		{\sl Case 1:} $0<r<\frac{\rho(x_0)}{4}$. Then, $\int_{\X} \a(y)d\mu(y)=0$. Thus, by the H\"older inequality, the fact that $\|\a\|_{L^1}\leq 1$, the Fubini theorem, Remark \ref{rem 1.4}\eqref{10:19, 22/12/2023}, and Condition \eqref{11:09, 13/12/2023}, we obtain
		\begin{align}\label{22:41, 29/10/2023}
			\|T\a\|_{L^q(U_j(B))}&=  \left[\int_{(2\kappa)^j r\leq d(x,x_0)<(2\kappa)^{j+1} r} \left|\int_B (K(x,y)-K(x,x_0))\a(y)d\mu(y)\right|^q d\mu(x)\right]^{\frac{1}{q}}\nonumber\\
			&\leq \left[\int_B |\a(y)|d\mu(y) \int_{(2\kappa)^j r\leq d(x,x_0)<(2\kappa)^{j+1} r} |K(x,y)-K(x,x_0)|^q d\mu(x) \right]^{\frac{1}{q}}\nonumber\\
			&\lesssim (2\kappa)^{-j\delta} \mu((2\kappa)^j B)^{\frac{1}{q}-1}.
		\end{align}
		
		{\sl Case 2:} $\frac{\rho(x_0)}{4}\leq r< \rho(x_0)$. Then, by  H\"older's inequality, the fact that $\|\a\|_{L^1}\leq 1$, Fubini's theorem, Remark \ref{rem 1.4}\eqref{10:19, 22/12/2023}, Condition \eqref{11:08, 13/12/2023}, and the comparability $\rho(x_0) \simeq r_0$, we obtain
		\begin{align}\label{22:42, 29/10/2023}
			\|T\a\|_{L^q(U_j(B))}&\leq \left[\int_{(2\kappa)^j r\leq d(x,x_0)<(2\kappa)^{j+1} r} \left(\int_B |K(x,y)| |\a(y)|d\mu(y)\right)^q d\mu(x)\right]^{\frac{1}{q}}\nonumber\\
			&\leq \left[\int_B |\a(y)|d\mu(y) \int_{(2\kappa)^j r\leq d(x,x_0)<(2\kappa)^{j+1} r} |K(x,y)|^q d\mu(x) \right]^{\frac{1}{q}}\nonumber\\
			&\lesssim \frac{\left(1+\frac{(2\kappa)^j r}{\rho(x_0)}\right)^{-\delta}}{\mu((2\kappa)^j B)^{1-\frac{1}{q}}} \lesssim (2\kappa)^{-j\delta} \mu((2\kappa)^j B)^{\frac{1}{q}-1}.
		\end{align}
		
		Combining \eqref{22:41, 29/10/2023} and \eqref{22:42, 29/10/2023}, we deduce that
		\begin{equation}\label{09:44, 15/11/2025}
			\|T\a\|_{L^q(U_j(B))}\lesssim (2\kappa)^{-j\delta} \mu((2\kappa)^j B)^{\frac{1}{q}-1}.
		\end{equation}

		\bigskip

		\noindent \eqref{09:32, 15/11/2025} Let $q:=\frac{s+1}{2}\in (1,s)$. Then, since $\theta\in [0,\frac{\delta}{i(\rho)+1})$ and $\a$ is an $(H^1_\rho,q)$-atom related to the ball $B=B(x_0,r)$, there exists $k_0>i(\rho)$ such that \eqref{admissible function} holds and
		\begin{equation}\label{09:40, 15/11/2025}		
			\varepsilon:= \delta -(k_0+1)\theta>0,\quad 0<r<\rho(x_0).
		\end{equation}	
		
		It follows from H\"older's inequality, Lemma \ref{14:51, 06/10/2023}, \eqref{09:44, 15/11/2025}, and \eqref{09:40, 15/11/2025} that
		\begin{align*}
			&\|(b-b_B)T\a\|_{L^1}=  \sum_{j=0}^\infty \|(b-b_B)T\a\|_{L^1(U_j(B))}\\
			&\hskip2cm\leq \sum_{j=0}^\infty \|b-b_B\|_{L^{q'}(U_j(B))}\|T\a\|_{L^q(U_j(B))}\\
			&\hskip2cm\lesssim \sum_{j=0}^\infty \mu((2\kappa)^{j+1} B)^{\frac{1}{q'}}(j+2)(2\kappa)^{(j+1) (k_0+1)\theta}\|b\|_{\mathrm{BMO}_{\rho,\theta}} (2\kappa)^{-j\delta}\mu((2\kappa)^j B)^{\frac{1}{q}-1}\\
			&\hskip2cm\lesssim \|b\|_{\mathrm{BMO}_{\rho,\theta}} \sum_{j=0}^\infty (j+2) (2\kappa)^{-j\varepsilon} \lesssim \|b\|_{\mathrm{BMO}_{\rho,\theta}}.
		\end{align*}
		This completes our proof.
	\end{proof}
	
	\begin{lemma}\label{23:50, 07/10/2023}
		Let $\varepsilon>0$. Then, there exists a constant $C_{\kappa,\varepsilon}>0$ such that
		$$j+2\leq C_{\kappa,\varepsilon} (2\kappa)^{j\varepsilon}$$
		and 
		$$\log\left(e+(2\kappa)^{j+1} t\right)\leq C_{\kappa,\varepsilon} (2\kappa)^{j\varepsilon} \log\left(e+ t\right)$$
		for all $j\in\mathbb{Z}^+_0$ and all $t>0$.
	\end{lemma}
	
	\begin{proof}
		Since $\lim\limits_{i\to\infty} \left(\frac{i+2}{(2\kappa)^{i\varepsilon}}+\frac{\log\left(e+(2\kappa)^{i+1}\right)}{(2\kappa)^{i\varepsilon}}\right)=0$, we obtain that
		$$j+2\lesssim (2\kappa)^{j\varepsilon}$$
		and
		$$\log\left(e+(2\kappa)^{j+1} t\right)\leq \log\left(e+ t\right) + \log\left(e+(2\kappa)^{j+1}\right)\lesssim (2\kappa)^{j\varepsilon} \log\left(e+ t\right),$$
		as desired.
		
	\end{proof}
	
The following characterization of the boundedness of $(s,\delta,\rho)$-Calder\'on--Zygmund operators on $H^1_\rho(\X)$ is due to Ky \cite[Theorem 5.2]{Ky25}.	

\begin{proposition}\label{prop- bounded on H1 and T1}
	Let $s\in (1,\infty)$, $\delta\in (0,1]$, and let $T$ be an $(s,\delta,\rho)$-Calder\'on--Zygmund operator. Then $T$ is bounded on $H^1_\rho(\X)$ if and only if $T^*1\in \mathrm{BMO}^{\log}_{\rho}(\X)$.
\end{proposition}

	We now give the proofs of Theorems \ref{thm-main thm 3 chracterize BMO log by Tj} and \ref{thm-main thm 4 chracterize BMO log by Tj on Hardy space}.
	
	\begin{proof}
		[Proof of Theorem \ref{thm-main thm 3 chracterize BMO log by Tj}] Since $H^1_\rho(\X)$ admits a characterization in terms of $(s,\delta,\rho)$-Calder\'on--Zygmund operators $\{T_j\}_{j=1}^n$, we have
		\[H^1_\rho(\X)=\left\{f\in L^1(\X): T_j(f)\in L^1(\X) \text{ for all } 1\leq j\leq n\right\}\]
		and
		\begin{equation}\label{10:43, 15/11/2025}
			\|f\|_{H^1_\rho}\simeq \|f\|_{L^1}+\sum_{j=1}^{n} \|T_j(f)\|_{L^1}\quad\text{for all } f\in H^1_\rho(\X).
		\end{equation}   
		
		For every $1\leq j\leq n$ and every $(H^1_\rho,s)$-atom $\a$ related to a ball $B$, Lemma \ref{15:47, 09/12/2023}\eqref{09:32, 15/11/2025}, combined with the boundedness of the commutator $[b,T_j]$ from $H^1_\rho(\X)$ to $L^1(\X)$, implies that
		\begin{align}\label{10:41, 15/11/2025}
			\left\|T_j\left((b-b_B)\a\right)\right\|_{L^1}&\leq \|(b-b_B)T_j(\a)\|_{L^1} + \|[b,T_j](\a)\|_{L^1}\nonumber\\
			&\lesssim \|b\|_{\mathrm{BMO}_{\rho,\theta}}+\|[b,T_j]\|_{H^1_\rho\to L^1}\|\a\|_{H^1_\rho}\lesssim \|b\|_{\mathrm{BMO}_{\rho,\theta}} +\|[b,T_j]\|_{H^1_\rho\to L^1}.
		\end{align}
		Moreover, by Lemma \ref{17:53, 11/10/2023}\eqref{14:01, 12/10/2023},
		\begin{equation}\label{09:01, 02/01/2026}
			\|(b-b_B)\a\|_{L^1}\lesssim \|b\|_{\mathrm{BMO}_{\rho,\theta}}.
		\end{equation}
		
		Combining \eqref{10:43, 15/11/2025}--\eqref{09:01, 02/01/2026}, we obtain
		\[
		\left\|(b-b_B)\a\right\|_{H^1_\rho}\lesssim \|b\|_{\mathrm{BMO}_{\rho,\theta}} + \sum_{j=1}^{n}\|[b,T_j]\|_{H^1_\rho\to L^1}.
		\]

		This, together with Lemma \ref{07:01, 12/10/2023} and Theorem \ref{thm-main thm 2 characterization BMO log}, implies  that $b \in \mathrm{BMO}_{\rho,\theta}^{\log}(\X)$; moreover, we obtain the norm equivalence
		\[\|b\|_{\mathrm{BMO}_{\rho,\theta}^{\log}}\simeq \|b\|_{\mathrm{BMO}_{\rho,\theta}}+ \sum_{j=1}^{n}\|[b,T_j]\|_{H^1_\rho\to L^1},\]
		where the constants are independent of $b$. This completes the proof of Theorem \ref{thm-main thm 3 chracterize BMO log by Tj}.
		
	\end{proof}
	
	\begin{proof}[Proof of Theorem \ref{thm-main thm 4 chracterize BMO log by Tj on Hardy space}]
	 Let $q:=\frac{s+1}{2}\in (1,s)$. Then, by \cite[Proposition 3.2]{YZ}, it suffices to prove that
		\begin{equation}\label{07:39, 10/12/2023}
			\|[b,T](\a)\|_{H^1_\rho}\lesssim \|b\|_{\mathrm{BMO}^{\log}_{\rho,\theta}}
		\end{equation}
		for every $(H^1_\rho,q)$-atom $\a$ related to a ball $B$. Indeed, since $\theta\in \big[0,\frac{\delta}{i(\rho)+1}\big)$ and $\a$ is an $(H^1_\rho,q)$-atom related to the ball $B=B(x_0,r)$, there exists $k_0>i(\rho)$ such that \eqref{admissible function} holds and
		\begin{equation}\label{06:57, 06/12/2023}		
			\varepsilon:= \frac{\delta -(k_0+1)\theta}{2}>0, \quad 0<r<\rho(x_0).
		\end{equation}		
		
		Let $p:=\frac{q+1}{2}\in (1,\infty)$ and $m:=\frac{q(q+1)}{q-1}\in (1,\infty)$, then $\frac{1}{p}=\frac{1}{m} +\frac{1}{q}$. For any $j\in\mathbb{Z}^+_0$, by the  H\"older inequality, Lemma \ref{14:51, 06/10/2023}, Lemma \ref{15:47, 09/12/2023}\eqref{09:31, 15/11/2025}, and the estimate $(j+2)\lesssim (2\kappa)^{j\varepsilon}$ (see Lemma \ref{23:50, 07/10/2023}), we obtain
		\begin{align}\label{09:58, 08/10/2023}
			&\|(b-b_B)T\a\|_{L^p(U_j(B))}\leq \|b-b_B\|_{L^m(U_j(B))} \|T\a\|_{L^q(U_j(B))}\nonumber\\
			&\hskip1cm\lesssim \mu((2\kappa)^{j+1} B)^{\frac{1}{m}}(j+2) (2\kappa)^{(j+1) (k_0+1)\theta}\|b\|_{\mathrm{BMO}_{\rho,\theta}} (2\kappa)^{-j\delta}\mu((2\kappa)^j B)^{\frac{1}{q}-1}\nonumber\\
			&\hskip1cm\lesssim (2\kappa)^{-j\varepsilon} \mu((2\kappa)^j B)^{\frac{1}{p}-1} \|b\|_{\mathrm{BMO}^{\log}_{\rho,\theta}}.
		\end{align}
		
		On the other hand, by the H\"older inequality, Lemma \ref{15:47, 09/12/2023}\eqref{09:31, 15/11/2025}, Lemma \ref{14:50, 06/10/2023}, the estimate $\log\left(e+(2\kappa)^{j+1} t\right)\lesssim (2\kappa)^{j\varepsilon} \log\left(e+ t\right)$ (see Lemma \ref{23:50, 07/10/2023}), and \eqref{06:57, 06/12/2023}, we obtain
		\begin{align*}
			&\left|\int_{\X} (b(x)-b_B)T\a(x)d\mu(x)\right|\leq \sum_{j=0}^\infty \int_{U_j(B)} |b(x)-b_B| |T\a(x)|d\mu(x)\\
			&\hskip1cm\leq \sum_{j=0}^\infty \|b-b_B\|_{L^{q'}(U_j(B))} \|T\a\|_{L^q(U_j(B))}\\
			&\hskip1cm\lesssim \sum_{j=0}^\infty  \mu((2\kappa)^{j+1} B)^{\frac{1}{q'}}(j+2) \frac{(2\kappa)^{(j+1)(k_0+1)\theta}}{\log\left(e+\frac{\rho(x_0)}{(2\kappa)^{j+1} r}\right)} \|b\|_{\mathrm{BMO}^{\log}_{\rho,\theta}} (2\kappa)^{-j\delta}\mu((2\kappa)^j B)^{\frac{1}{q}-1}\\
			&\hskip1cm\lesssim \frac{\|b\|_{\mathrm{BMO}^{\log}_{\rho,\theta}}}{\log\left(e+\frac{\rho(x_0)}{r}\right)}\sum_{j=0}^\infty (j+2)(2\kappa)^{-j\varepsilon} \\
			&\hskip1cm\lesssim  \frac{\|b\|_{\mathrm{BMO}^{\log}_{\rho,\theta}}}{\log\left(e+\frac{\rho(x_0)}{r}\right)}.
		\end{align*}
		This, together with \eqref{09:58, 08/10/2023}, shows that $(b-b_B)T a$ is, up to the factor $C \|b\|_{\mathrm{BMO}^{\log}_{\rho,\theta}}$, an $(H^1_\rho,p,\varepsilon)$-molecule related to the ball $B=B(x_0,r)$. Therefore, by Lemma \ref{07:01, 14/10/2023},
		$$\|(b-b_B)T\a\|_{H^1_\rho}\lesssim \|b\|_{\mathrm{BMO}^{\log}_{\rho,\theta}}.$$
		Thus, by Lemma \ref{07:01, 12/10/2023} we have
		\begin{align*}
			\|[b,T](\a)\|_{H^1_\rho}&\leq \|(b-b_B)T\a\|_{H^1_\rho}+ \|T(\a(b-b_B))\|_{H^1_\rho}\\
			&\lesssim \|b\|_{\mathrm{BMO}^{\log}_{\rho,\theta}} + \|T\|_{H^1_\rho\to H^1_\rho} \|\a(b-b_B))\|_{H^1_\rho} \lesssim \|b\|_{\mathrm{BMO}^{\log}_{\rho,\theta}},
		\end{align*}
		where in the last inequality we used Proposition \ref{prop- bounded on H1 and T1}.
		
		This proves \eqref{07:39, 10/12/2023}, and thus completes our proof of the boundedness of the commutator $[b,T]$ on the Hardy space $H^1_\rho(\X)$.

		\medskip
		The reverse direction can be proved similarly to the proof of Theorem \ref{thm-main thm 3 chracterize BMO log by Tj}, and hence we omit the details.
		
		This completes the proof.

	\end{proof}

	\section{Some examples and applications}\label{16:59, 30/12/2025}
In this section, we present several examples of operators that fall within the scope of our theory. Although the list is not exhaustive, it demonstrates the effectiveness and applicability of our approach.

	\subsection{Sub-Laplacian Schr\"odinger operators on stratified Lie groups}\label{19:33, 01/02/2026}
	
	Let $\mathbb G$ be a stratified Lie group and let $\mathfrak{g}$ be its Lie algebra of dimension $n$. Namely, $\mathbb G$ is nilpotent, connected and simply connected, and the Lie algebra $\mathfrak{g}$ admits a vector space decomposition $\mathfrak{g} = V_1 \oplus \cdots \oplus V_m$ such that $[V_1, V_k] = V_{k+1}$ for $1 \leq k < m$ and $[V_1, V_m] = 0$. Let $X = \{X_1, \dots, X_{d_1}\}$, where $d_1=\dim V_1$, be left invariant vector fields on $\mathbb G$ satisfying the H\"ormander condition. Namely, $X$, together with its commutators of order $\leq m$, generates the tangent space of $\mathbb G$ at each point of $\mathbb G$. Let $d$ be the Carnot-Carath\'eodory (control) metric on $\mathbb G$ associated to $X$. The Haar measure $\mu$ on $\mathbb G$ is simply the Lebesgue measure on $\R^n$ under the identification of $\mathbb G$ with $\mathfrak{g}$ and the identification of $\mathfrak{g}$ with $\R^n$. Then $(\mathbb G, d,\mu)$ is a space of homogeneous type, and in fact an RD-space in the sense of Han, M\"uller and Yang \cite{HMY}.	Moreover,
	\begin{equation}
		\mu(B(x,r))\simeq r^Q\quad\text{for all $x\in\mathbb G$ and all $r>0$},
	\end{equation}
	where $Q=\sum_{k=1}^{m} k \dim V_k$ is the homogeneous dimension of $\mathbb G$ (see, for example, \cite[Subsection 4.2]{LHD}). Throughout this subsection, we assume that $Q\geq 3$.
	
	The sub-Laplacian is given by $\Delta_{\mathbb G} = \sum_{j=1}^{d_1} X_j^2$. The gradient operator $\nabla_{\mathbb G}$ is denoted by $\nabla_{\mathbb G} = (X_1, \dots, X_{d_1})$. Let $L = -\Delta_{\mathbb G} + V$ be a Schr\"odinger operator on $\mathbb G$, where $V$ is a nonnegative potential belonging to the reverse H\"older class $\mathrm{RH}_{Q/2}(\mathbb G, d,\mu)$; see \eqref{def:reverse-holder-class}. The reverse H\"older index of $V$ is then defined by
	\begin{equation}\label{14:47, 30/11/2023}
		\tau(V):= \sup\left\{r>1: V\in \mathrm{RH}_{r}(\mathbb G, d,\mu)\right\} \in (q,\infty].
	\end{equation}
	
	\begin{remark}
		If $V \in \mathrm{RH}_{q}(\mathbb G, d,\mu)$ for some $q \in (1, \infty)$, then $V \in \mathrm{RH}_{s}(\mathbb G, d,\mu)$ for every $s \in (1, q)$, and there exists $r \in (q,\infty)$ such that $V \in \mathrm{RH}_{r}(\mathbb G, d,\mu)$. See, for example, \cite[Chapter I]{ST}. 
	\end{remark}

	Following \cite{Sh, YZ}, for any $V\in \mathrm{RH}_{Q/2}(\mathbb G, d,\mu)$, we define
	\begin{equation}\label{eq-rho Lie group}
		\rho(x)=\sup\left\{r>0: \frac{1}{r^{Q-2}}\int_{B(x,r)} V(y) d\mu(y)\leq 1 \right\}
	\end{equation}
	for all $x\in \mathbb G$. Then, by \cite[Proposition 2.1]{YZ}, $\rho$ is an	admissible function on $\mathbb G$. Following \cite{LLL, YZ}, we define the Hardy space $H^1_{L}(\mathbb G)$ as the set of functions $f\in L^1(\mathbb G)$ such that 
	\[\|f\|_{H^1_{L}(\mathbb G)}:= \|\mathbb M_{L} f\|_{L^1(\mathbb G)}<\infty,\]
	where $\mathbb M_{L} f(x)= \sup_{t>0} \left|e^{- t L} f(x)\right|$ for all $x\in\mathbb G$. Then, by \cite[Theorem 5.2]{YZ}, we have
	\begin{equation}
		H^1_{L}(\mathbb G)=H^{1}_{\rho}(\mathbb G)
	\end{equation}
	with equivalent norms. Moreover, the space $H^1_{L}(\mathbb G)$ admits a characterization in terms of singular integral operators (see \cite[Section 7]{LLL}), as stated in the following theorem.
	
	\begin{theorem}\label{thm:H1L-Riesz-characterization}
		Let $V\in \mathrm{RH}_{Q/2}(\mathbb G, d,\mu)$ and let $R_j^{L}=X_j L^{-1/2}$ ($j = 1, \dots, d_1$) be the Riesz transforms. Then, a function $f$ belongs to $H^1_{L}(\mathbb G)$ if and only if $f\in L^1(\mathbb G)$ and $R_j^{L} f\in L^1(\mathbb G)$ for all $j = 1, \dots, d_1$. Moreover,
		\[\|f\|_{H^1_{L}(\mathbb G)}\simeq \|f\|_{L^1(\mathbb G)}+\sum_{j=1}^{d_1} \|R_j^{L} f\|_{L^1(\mathbb G)}\quad\text{for all } f\in H^1_{L}(\mathbb G).\]
	\end{theorem}

	We now present applications of our main results to singular integral operators associated to $L$. We first consider the Riesz transforms $R^L_j, j=1\ldots, d_1$. By combining \cite[TH\'EOR\`EME C]{HQLi}, the kernel estimates for $R_j^{L}$ established in \cite[Lemma 12]{JLiu}, and an argument similar to that in the proof of \cite[Theorem 1(iii)]{GLP}, we obtain the following: 
	\begin{proposition}
			Let $V\in \mathrm{RH}_{Q/2}(\mathbb{G}, d,\mu)$,   $0<\delta<\min\left\{2-\frac{Q}{\tau(V)},1\right\}$, and $1<s< p_0$, where
		\[p_0:=\begin{cases}
			\frac{Q \tau(V)}{Q-\tau(V)}&\text{if } \tau(V)\in (Q/2, Q),\\
			\infty &\text{if } \tau(V)\in [Q,\infty],
		\end{cases}\]
		  and $\rho$ be as in \eqref{eq-rho Lie group}. Then, for each $j=1,\ldots,d_1$, the Riesz transform $R_j^{L}$ is an $(s,\delta,\rho)$-Calder\'on--Zygmund operator and is bounded on
		$H^1_{L}(\mathbb G)$. 
		\end{proposition}
		
		This, together with  Remark \ref{rem 1.4}\eqref{10:19, 22/12/2023}, implies that  $R_j^{L} \in \K_{\rho,\theta,s}$ for all $\theta\in \bigl[0,\frac{\delta}{i(\rho)+1}\bigr)$. In addition, $(R_j^{L})^*1 =0$, and hence $(R_j^{L})^*1\in \mathrm{BMO}_{\rho}^{\log}(\mathbb G)$.

	Consequently, from Theorems \ref{thm-main thm 1 from H1 to weak L1}, \ref{thm-main thm 4 chracterize BMO log by Tj on Hardy space}, and \ref{thm:H1L-Riesz-characterization}, we obtain the following result.
	
	\begin{theorem}\label{07:53, 30/01/2026}
		Let $V\in \mathrm{RH}_{Q/2}(\mathbb{G}, d,\mu)$, $0\leq \theta<\min\left\{\frac{2-\frac{Q}{\tau(V)}}{i(\rho)+1},\frac{1}{i(\rho)+1}\right\}$ and $\rho$ be as in \eqref{eq-rho Lie group}. Then the following assertions hold:
		\begin{enumerate}[\rm (i)]
			\item If $b\in \mathrm{BMO}_{\rho,\theta}(\mathbb{G})$, then the commutator $[b, R_j^{L}]$ is bounded from $H^1_{L}(\mathbb{G})$ into $L^{1,\infty}(\mathbb{G})$ for all $j=1,\ldots, d_1$.
			
			\item\label{19:06, 28/01/2026} If $b\in \mathrm{BMO}_{\rho,\theta}^{\log}(\mathbb{G})$, then the commutator $[b, R_j^{L}]$ is bounded on $H^1_{L}(\mathbb{G})$ for all $j=1,\ldots, d_1$.
			
			\item Conversely, if $b\in \mathrm{BMO}_{\rho,\theta}(\mathbb{G})$ is such that all commutators $[b, R_j^{L}]$ are bounded on $H^1_{L}(\mathbb{G})$, then $b\in \mathrm{BMO}_{\rho,\theta}^{\log}(\mathbb{G})$. Moreover,		
			\[
			\|b\|_{{\rm BMO}_{\rho,\theta}^{\log}(\mathbb{G})}\simeq \|b\|_{{\rm BMO}_{\rho,\theta}(\mathbb{G})} + \sum_{j=1}^{d_1} \|[b,R_j^{L}]\|_{H^1_{L}(\mathbb{G})\to H^1_{L}(\mathbb{G})}.
			\]
		\end{enumerate}
	\end{theorem}

	Let $\mathrm{BMO}_{L}(\mathbb{G}):=\mathrm{BMO}_{\rho}(\mathbb{G})$ be defined as \eqref{eq-BMO}, which is  the dual space of $H^1_{L}(\mathbb{G})$.   As a consequence of Theorem \ref{07:53, 30/01/2026}, we obtain the following corollary.

	\begin{corollary}\label{14:35, 03/02/2026}
		Let $V\in \mathrm{RH}_{Q/2}(\mathbb{G}, d,\mu)$, $0\leq \theta<\min\left\{\frac{2-\frac{Q}{\tau(V)}}{i(\rho)+1},\frac{1}{i(\rho)+1}\right\}$  and $\rho$ be as in \eqref{eq-rho Lie group}. Then the following assertions hold:
		\begin{enumerate}[\rm (i)]		
			\item\label{14:28, 03/02/2026} If $b\in \mathrm{BMO}_{\rho,\theta}^{\log}(\mathbb{G})$, then the commutator $[b, (R_j^{L})^*]$ is bounded on $\mathrm{BMO}_{L}(\mathbb{G})$ for all $j=1,\ldots, d_1$.
			
			\item Conversely, if $b\in \mathrm{BMO}_{\rho,\theta}(\mathbb{G})$ satisfies the boundedness condition in \eqref{14:28, 03/02/2026}, then $b\in \mathrm{BMO}_{\rho,\theta}^{\log}(\mathbb{G})$. Moreover,		
			\[
			\|b\|_{{\rm BMO}_{\rho,\theta}^{\log}(\mathbb{G})}\simeq \|b\|_{{\rm BMO}_{\rho,\theta}(\mathbb{G})} + \sum_{j=1}^{d_1} \|[b, (R_j^{L})^*]\|_{\mathrm{BMO}_{L}(\mathbb{G})\to \mathrm{BMO}_{L}(\mathbb{G})}.
			\]
		\end{enumerate}
	\end{corollary}
	
	\begin{remark}	
		Theorem \ref{07:53, 30/01/2026} and Corollary \ref{14:35, 03/02/2026}
		are new and extend \cite[Theorem 3.6]{Ky15} and
		\cite[Theorem 2]{BHS}, respectively.
	\end{remark}

	For the Riesz transform $V^\alpha L^{-\alpha}$ with $\alpha\in (0,1]$, by combining the heat kernel estimates for $L$ from \cite[Subsection 6.1]{BDK18} and an argument similar to that in the proof of \cite[Theorem 1.7]{BLL}, we obtain:

	\begin{proposition} 
		Let $V\in \mathrm{RH}_{Q/2}(\mathbb G, d,\mu)$, $\alpha\in (0,1]$, $0<\delta<\min\left\{\alpha\left(2-\frac{Q}{\tau(V)}\right),1\right\}$, $s\in \big(1,\frac{\tau(V)}{\alpha}\big)$   and $\rho$ be as in \eqref{eq-rho Lie group}. Then, the Riesz transform $V^\alpha L^{-\alpha}$ is an $(s,\delta,\rho)$-Calder\'on--Zygmund operator and is bounded on $H^1_{L}(\mathbb G)$. 
	\end{proposition}
	
	Consequently, by Remark \ref{rem 1.4}(i) and Proposition \ref{prop- bounded on H1 and T1}, we have $V^\alpha L^{-\alpha} \in \K_{\rho,\theta,s}$ for all $\theta\in \bigl[0,\frac{\delta}{i(\rho)+1}\bigr)$, and $(V^\alpha L^{-\alpha})^*1\in \mathrm{BMO}_{\rho}^{\log}(\mathbb G)$. Therefore, applying Theorems \ref{thm-main thm 1 from H1 to weak L1} and \ref{thm-main thm 4 chracterize BMO log by Tj on Hardy space}, we deduce the following.
	
	\begin{theorem}\label{07:44, 30/01/2026}
		Let $V\in \mathrm{RH}_{Q/2}(\mathbb G, d,\mu)$, $\alpha\in (0,1]$,  $0\leq \theta<\min\left\{\frac{\alpha\left(2-\frac{Q}{\tau(V)}\right)}{i(\rho)+1},\frac{1}{i(\rho)+1}\right\}$   and $\rho$ be as in \eqref{eq-rho Lie group}. Then the following assertions hold:
		\begin{enumerate}[\rm (i)]
			\item If $b\in \mathrm{BMO}_{\rho,\theta}(\mathbb G)$, then the commutator $[b, V^{\alpha}L^{-\alpha}]$ is bounded from $H^1_{L}(\mathbb G)$ into $L^{1,\infty}(\mathbb G)$ .
			
			\item If $b\in \mathrm{BMO}^{\log}_{\rho,\theta}(\mathbb G)$, then the commutator $[b, V^{\alpha}L^{-\alpha}]$ is bounded on $H^1_{L}(\mathbb G)$.
		\end{enumerate}
		
	\end{theorem}
	
	\begin{remark}
		Theorem \ref{07:44, 30/01/2026} is new and extends \cite[Theorem 1.2]{HW}.
	\end{remark}
	
	We next take care of the second-order Riesz transform $\nabla_{\mathbb G}^2 L^{-1}$. By combining the heat kernel estimates for $L$ from \cite[Subsection 6.1]{BDK18} and an argument similar to that in the proof of \cite[Theorem 1.6]{BLL}, we have:

	\begin{proposition} 	
		Let $V\in \mathrm{RH}_{Q/2}(\mathbb G, d,\mu)$, $0<\delta<\min\left\{2-\frac{Q}{\tau(V)},1\right\}$,  $1<s<\tau(V)$   and $\rho$ be as in \eqref{eq-rho Lie group}. Then, the Riesz transform $\nabla_{\mathbb G}^2 L^{-1}$ is an $(s,\delta,\rho)$-Calder\'on--Zygmund operator and is bounded on $H^1_{L}(\mathbb G)$.
	\end{proposition}
	
	Consequently, by Remark \ref{rem 1.4}(i) and Proposition \ref{prop- bounded on H1 and T1}, we have $\nabla_{\mathbb G}^2 L^{-1} \in \K_{\rho,\theta,s}$ for all $\theta\in \bigl[0,\frac{\delta}{i(\rho)+1}\bigr)$, and $(\nabla_{\mathbb G}^2 L^{-1})^*1\in \mathrm{BMO}_{\rho}^{\log}(\mathbb G)$. Therefore, by using Theorems \ref{thm-main thm 2 characterization BMO log}, \ref{thm-main thm 3 chracterize BMO log by Tj}, we have:
	
	\begin{theorem}\label{08:31, 30/01/2026}
		Let $V\in \mathrm{RH}_{Q/2}(\mathbb G, d,\mu)$, $0\leq \theta<\min\left\{\frac{2-\frac{Q}{\tau(V)}}{i(\rho)+1},\frac{1}{i(\rho)+1}\right\}$   and $\rho$ be as in \eqref{eq-rho Lie group}.  Then the following assertions hold:
		\begin{enumerate}[\rm (i)]
			\item If $b\in \mathrm{BMO}_{\rho,\theta}(\mathbb G)$, then the commutator $[b, \nabla_{\mathbb G}^2 L^{-1}]$ is bounded from $H^1_{L}(\mathbb G)$ into $L^{1,\infty}(\mathbb G)$ .
			
			\item If $b\in \mathrm{BMO}^{\log}_{\rho,\theta}(\mathbb G)$, then the commutator $[b, \nabla_{\mathbb G}^2 L^{-1}]$ is bounded on $H^1_{L}(\mathbb G)$.
		\end{enumerate}
		
	\end{theorem}
	
	\begin{remark}
		Theorem \ref{08:31, 30/01/2026} is new. Moreover, Theorems \ref{07:53, 30/01/2026}, \ref{07:44, 30/01/2026}, and \ref{08:31, 30/01/2026}  extend and improve \cite[Theorems 3.5 and 4.1]{LP}.
	\end{remark}

We conclude this subsection by considering the the \emph{Laplace transform type multiplier} of $L$.  Given a bounded function $a: [0,\infty) \rightarrow \mathbb{C}$, we define the \emph{Laplace transform type multiplier $m_a(L)$} by
	\begin{equation}
		\label{eq-Spectral}
		m_a(L)f(x)=\int_0^\infty a(t)Le^{-tL}f(x)dt,
	\end{equation}
	which is bounded on $L^2(\mathbb G)$. A typical example is provided by the imaginary powers $m_a(L)=L^{i\nu}$, which correspond to the kernel $a(t)=-\frac{1}{\Gamma(i\nu)}t^{-i\nu}$ for $\nu\in\mathbb R$.	
	
	By combining the heat kernel estimates for $L$ from \cite[Subsection 6.1]{BDK18} and an argument similar to that in the proof of \cite[Theorem 1.5]{BLL}, we obtain the following proposition.
	
	\begin{proposition} 
		Let $V\in \mathrm{RH}_{Q/2}(\mathbb G, d,\mu)$, $0<\delta<\min\left\{2-\frac{Q}{\tau(V)},1\right\}$,   and $\rho$ be as in \eqref{eq-rho Lie group}. Let $m_a(L)$ be defined as in \eqref{eq-Spectral}. Then, $m_a(L)$ is a $(\delta,\rho)$-Calder\'on--Zygmund operator and is bounded on $H^1_{L}(\mathbb G)$. 
	\end{proposition}
	
	By Remark \ref{rem 1.4}\eqref{06:22, 26/12/2023} and Proposition \ref{prop- bounded on H1 and T1}, we have $m_a(L) \in \K_{\rho,\theta,s}$ for all $\theta\in \big[0,\frac{\delta}{i(\rho)+1}\big)$ and  $s\in (1,\infty)$ and  $(m_a(L))^*1\in \mathrm{BMO}_{\rho}^{\log}(\mathbb G)$. As a direct consequence of  Theorems \ref{thm-main thm 1 from H1 to weak L1} and \ref{thm-main thm 4 chracterize BMO log by Tj on Hardy space}, have:
	
	\begin{theorem}\label{10:28, 30/01/2026}
		Let $V\in \mathrm{RH}_{Q/2}(\mathbb G, d,\mu)$, $0\leq \theta<\min\left\{\frac{2-\frac{Q}{\tau(V)}}{i(\rho)+1},\frac{1}{i(\rho)+1}\right\}$  and $\rho$ be as in \eqref{eq-rho Lie group}. Let $m_a(L)$ be defined as in \eqref{eq-Spectral}. Then the following assertions hold:
		\begin{enumerate}[\rm (i)]
			\item If $b\in \mathrm{BMO}_{\rho,\theta}(\mathbb G)$, then the commutator $[b, m_a(L)]$ is bounded from $H^1_{L}(\mathbb G)$ into $L^{1,\infty}(\mathbb G)$ .
			
			\item If $b\in \mathrm{BMO}^{\log}_{\rho,\theta}(\mathbb G)$, then the commutator $[b, m_a(L)]$ is bounded on $H^1_{L}(\mathbb G)$.
		\end{enumerate}
		
	\end{theorem}

	\subsection{Riesz transforms associated to Laguerre operators}
	
	Let $n\geq 1$ be a fixed integer and let $\alpha = (\alpha_1, \dots, \alpha_n) \in (-1,\infty)^n$. We consider the space $\mathcal{X}=(0,\infty)^n$ endowed with the measure $d\mu_{\alpha}(x) = x^{\alpha} dx$ and the standard Euclidean distance $d(x,y)=|x-y|$. Then, the following volume estimate holds
	\begin{equation}\label{volume-Bessel}
		\mu_{\alpha}(B(x,r))\simeq r^n
		\prod_{j=1}^n(x_j+r)^{\alpha_j}
	\end{equation}
	for all $x=(x_1,\ldots,x_n)\in \X$ and all $r>0$. This implies that
	\begin{equation}\label{doubling2}
		\lambda^\gamma \mu_{\alpha}(B(x,r)) \lesssim\mu_\alpha(B(x,\lambda r))\lesssim \lambda^D \mu_{\alpha}(B(x,r))
	\end{equation}
	for all $x\in\X$, $r>0$ and $\lambda\geq 1$, where $\gamma:=n+\sum_{j=1}^{n} \min\{\alpha_j,0\}>0$ and $D:=n+\sum_{j=1}^{n} \max\{\alpha_j,0\}$. Thus, $(\X,d,\mu_\alpha)$ is a space of homogeneous type, and in fact an RD-space in the sense of Han, M\"uller and Yang \cite{HMY}. Moreover,
	\begin{equation}\label{doubling3}
		\mu_{\alpha}(B(x,r))\lesssim \Big(1+\frac{d(x,y)}{r}\Big)^D \mu_{\alpha}(B(y,r))\quad\text{for all } x,y\in \X,\ r>0.
	\end{equation}

	Given a multi-index $\mathbf{k} = (k_1, \dots, k_n) \in \mathbb{N}^n$, we consider the \emph{multivariate Laguerre functions} defined by the tensor product
	\begin{equation}
		\Psi_{\mathbf{k}}^{(\alpha-1)/2}(x) 
		:= \prod_{j=1}^n \psi_{k_j}^{(\alpha_j-1)/2}(x_j), \quad x=(x_1,\ldots,x_n) \in \X,
	\end{equation}
	where each $\psi_{k_j}^{(\alpha_j-1)/2}$ is the one-dimensional Laguerre function
	\begin{equation}
		\psi_{k_j}^{(\alpha_j-1)/2}(x_j) 
		= \Bigg(\frac{2k_j!}{\Gamma(k_j+\alpha_j/2+1/2)}\Bigg)^{1/2} 
		L_{k_j}^{(\alpha_j-1)/2}(x_j^2) e^{-x_j^2/2}.
	\end{equation}
	
	The \emph{multivariate Laguerre operator} $L_\alpha$ is then defined by
	\begin{equation}
		L_\alpha = - \sum_{j=1}^n \frac{\partial^2}{\partial x_j^2} 
		- \sum_{j=1}^n \frac{\alpha_j}{x_j} \frac{\partial}{\partial x_j} 
		+ \sum_{j=1}^n x_j^2.
	\end{equation}
	Its domain is given spectrally by
	\begin{equation}
		D(L_\alpha) = \Bigg\{ f \in L^2(\X) : 
		\sum_{\mathbf{k} \in \mathbb{N}^n} 
		\Big( \sum_{j=1}^n (4k_j + \alpha_j +1) \Big)^2 
		\big| \langle f, \Psi_{\mathbf{k}}^{(\alpha-1)/2} \rangle \big|^2 < \infty \Bigg\}.
	\end{equation}
	Moreover, the following spectral property holds:
	\begin{equation}
		L_\alpha \Psi_{\mathbf{k}}^{(\alpha-1)/2} = 
		\Big( \sum_{j=1}^n (4k_j + \alpha_j +1) \Big) \Psi_{\mathbf{k}}^{(\alpha-1)/2},
		\quad \mathbf{k}=(k_1,\ldots,k_n) \in \mathbb{N}^n.
	\end{equation}

	It is known  that $L_\alpha$ is a non-negative self-adjoint operator on $L^2(\X, d\mu_\alpha)$.
	Moreover, let $p^\alpha_t(x,y)$ denote the kernel of the semigroup $e^{-tL_\alpha}$. Then
	\[
	p^\alpha_t(x,y) =\prod_{j=1}^n p^{\alpha_j}_t(x_j,y_j)
	\]
	for all $t>0$, $x=(x_1,\ldots, x_n) ,y=(y_1,\ldots,y_n)\in \X$,  where
	\begin{equation}\label{pt Laguerre}
		p^{\alpha_j}_t(x_j,y_j)=\f{2e^{-2t}}{1-e^{-4t}}\exp{\Big(-\f{1}{2}\f{1+e^{-4t}}{1-e^{-4t}}(x_j^2+y_j^2)\Big)}(x_jy_j)^{-(\alpha_j-1)/2}I_{(\alpha_j-1)/2}\Big(\f{2e^{-2t}}{1-e^{-4t}}x_jy_j\Big)
	\end{equation}
	for all  $j=1,\ldots,n$, where  $I_{(\alpha-1)/2}$ is the Bessel function. See, for example, \cite{Le}.
	
	We define the admissible function $\rho$ on $\X$ by
	\begin{equation}\label{eq-rho laguer}
		\rho(x)=\frac{1}{8} \min\{1, |x|^{-1}\},\quad x\in\X.
	\end{equation}
	It is straightforward to verify that the admissibility index of $\rho$ is
	\begin{equation}\label{eq-rho laguer_index}
		i(\rho)	= 1.
	\end{equation}
	
	For each $j=1,\ldots, n$,  define
	\[
	\delta_j = \f{\partial}{\partial x_j} + x_j.
	\]
	It is well known that the (formal) adjoint of $\delta_j$ in $L^2(\X, d\mu_\alpha)$ is given by
	\[
	\delta_j^*=-\f{\partial}{\partial x_j} + x_j-\f{\alpha_j}{x_j}.
	\]
We start by proving some kernel estimates for the Laguerre operator $L_\alpha$.	
	\begin{proposition}\label{prop-derivaitive of pt}
		\begin{enumerate}[\rm (i)]
			
			\item\label{prop-derivaitive of pt_1} For any multi-index $\beta$  and $N>0$, we have 
			\[
			|  x^\beta p^{\alpha}_t(x,y)|\lesssim \f{1}{t^{|\beta|/2}\mu_{\alpha}(B(x,\sqrt{t}))}\exp\Big(-\f{|x-y|^2}{ct}\Big)\Big(1+\f{\sqrt{t}}{\rho(x)}+\f{\sqrt{t}}{\rho(y)}\Big)^{-N}
			\] 
			for all $x,y\in \X$ and $t>0$.
			
			\item\label{prop-derivaitive of pt_2} For any $j\in \{1,\ldots, n\}$, $k\in \{0,1,2\}$, and $N>0$, we have 
			\[
			| \partial_{j}^k p^{\alpha}_t(x,y)|\lesssim \f{1}{t^{k/2}\mu_{\alpha}(B(x,\sqrt{t}))}\exp\Big(-\f{|x-y|^2}{ct}\Big)\Big(1+\f{\sqrt{t}}{\rho(x)}+\f{\sqrt{t}}{\rho(y)}\Big)^{-N}
			\]
			for all $x,y\in \X$ and $t>0$.
		\end{enumerate}
	\end{proposition}
	
	\begin{proof}
		\eqref{prop-derivaitive of pt_1} We first recall some basic properties of the Bessel functions $I_\nu$ with $\nu>-1$. It is known (see, e.g., \cite{Le}) that
		\begin{equation}\label{eq1-Bessel}		
			z^{-\nu}I_\nu(z)\simeq 2^{-\nu},\quad z\in (0,1];
		\end{equation}
		\begin{equation}\label{eq2-Bessel}
			I_\nu(z)\lesssim z^{-1/2}e^z,\quad z\geq 1;
		\end{equation}
		moreover, for all $\nu>-1, z>0$,
		\begin{equation}\label{eq3-Bessel}
			\f{d}{dz}(z^{-\nu}I_\nu(z))=z^{-\nu}I_{\nu+1}(z),
		\end{equation}
		\begin{equation}
			\label{eq4-Bessel}
			|I_\nu(z)-I_{\nu+1}(z)|<(4\nu +6)\f{I_{\nu+1}(z)}{z}.
		\end{equation}
		
		Since $p^{\alpha}_t(x,y) =\prod_{j=1}^n p^{\alpha_j}_t(x_j,y_j)$, it suffices to prove the proposition for $n=1$. Indeed, it was proved in \cite{BDK18} that for any $N>0$, there exists $C=C(N)$ such that	
		\begin{equation}\label{eq-pt}
			|p^{\alpha}_t(x,y)|\leq \f{C}{\mu_\alpha(B(x,\sqrt{t}))}\exp\Big(-\f{|x-y|^2}{ct}\Big)\Big(1+\f{\sqrt{t}}{\rho(x)}+\f{\sqrt{t}}{\rho(y)}\Big)^{-N}
		\end{equation}
		for all $x,y\in (0,\vc)$ and $t>0$.
		
		This estimate, together with the fact $x\rho(x)\lesssim 1$ for all $x\in (0,\vc)$, implies that
		\begin{equation}\label{eq-xy pt}
			| x^\beta y^\gamma p^{\alpha}_t(x,y)|\leq \f{C(N)}{t^{(\beta+\gamma)/2}\mu_\alpha(B(x,\sqrt{t}))}\exp\Big(-\f{|x-y|^2}{ct}\Big)\Big(1+\f{\sqrt{t}}{\rho(x)}+\f{\sqrt{t}}{\rho(y)}\Big)^{-N},
		\end{equation}
		for all $x,y\in (0,\vc)$ and $t, N>0$, which proves \eqref{prop-derivaitive of pt_1}.
		
		We turn to the proof of \eqref{prop-derivaitive of pt_2}. Recall from \eqref{pt Laguerre} that
		\[
		p^{\alpha}_t(x,y)=\f{2e^{-2t}}{1-e^{-4t}}\exp{\Big(-\f{1}{2}\f{1+e^{-4t}}{1-e^{-4t}}(x^2+y^2)\Big)}(xy)^{-(\alpha-1)/2}I_{(\alpha-1)/2}\Big(\f{2e^{-2t}}{1-e^{-4t}}xy\Big).
		\]	
		The case $k=0$ in \eqref{prop-derivaitive of pt_2} follows from \eqref{eq-pt}. We now prove \eqref{prop-derivaitive of pt_2} for $k=1$. To this end, we consider two cases.
		
		\textit{Case 1: $t\ge 1$.}  Using the product rule and \eqref{eq3-Bessel}, we obtain
		\begin{equation}\label{eq-derivative of pt}
			\begin{aligned}
				\partial_x p^{\alpha}_t(x,y) &= -\f{1+e^{-4t}}{1-e^{-4t}} x p_t^\alpha(x,y)\\
				& \ \  +\f{2e^{-2t}}{1-e^{-4t}}y\times \f{2e^{-2t}}{1-e^{-4t}}\exp{\Big(-\f{1}{2}\f{1+e^{-4t}}{1-e^{-4t}}(x^2+y^2)\Big)}(xy)^{-(\alpha-1)/2}I_{(\alpha+1)/2}\Big(\f{2e^{-2t}}{1-e^{-4t}}xy\Big)\\
				&=-\f{1+e^{-4t}}{1-e^{-4t}} x p_t^{\alpha}(x,y) +\f{2e^{-2t}}{1-e^{-4t}}xy^2 p_t^{\alpha+2}(x,y).
			\end{aligned}
		\end{equation}
		Note that   $e^{-2t}\le 1$ and $1-e^{-4t}\le  1$ for $t\ge 1$. Hence, we further imply
		\[
		|\partial_x p^{\alpha}_t(x,y)|\lesssim x p_t^{\alpha}(x,y)+xy^2 p_t^{\alpha+2}(x,y).
		\]
		This, together with \eqref{eq-xy pt}, implies that for all $x,y\in(0,\vc)$ and $t\ge 1$,
		\[
		\begin{aligned}
			|\partial_x p^{\alpha}_t(x,y)| &\lesssim \Big[\f{1}{t^{1/2}\mu_\alpha(B(x,\sqrt t))}+\f{1}{t^{3/2}\mu_{\alpha+2}(B(x,\sqrt t))}\Big]\exp\Big(-\f{|x-y|^2}{ct}\Big)\Big(1+\f{\sqrt{t}}{\rho(x)}+\f{\sqrt{t}}{\rho(y)}\Big)^{-N}\\
			&\lesssim  \f{1}{t^{1/2}\mu_\alpha(B(x,\sqrt t))} \exp\Big(-\f{|x-y|^2}{ct}\Big)\Big(1+\f{\sqrt{t}}{\rho(x)}+\f{\sqrt{t}}{\rho(y)}\Big)^{-N}, 
		\end{aligned}
		\] 
		as desired, where in the last inequality we used 
		\[
		\mu_{\alpha+2}(B(x,\sqrt t))\simeq (|x|+\sqrt t)^2\mu_{\alpha}(B(x,\sqrt t))\ge \mu_{\alpha}(B(x,\sqrt t)), \ \ t\ge 1. 
		\]	
		
		\textit{Case 2: $t\in (0,1)$.} In this case, $e^{-4t}$ and $e^{-2t}$ are comparable to $1$, and $1-e^{-4t}\simeq t$. Combining these facts with \eqref{eq-derivative of pt}, we obtain
		\begin{equation}\label{eq-derivative of pt t small}
			\begin{aligned}
				|\partial_x p^{\alpha}_t(x,y)| &\lesssim \f{x}{t} p_t^{\alpha}(x,y) +\f{xy^2}{t}  p_t^{\alpha+2}(x,y).
			\end{aligned}
		\end{equation}
		We now consider three subcases.
		
		\underline{\textit{Subcase 2.1: $x>2y$ or $y>2x$}}. In this case $|x-y|\gtrsim \max\{x,y\}$. Hence, by \eqref{eq-pt},
		\[
		\begin{aligned}
			|\partial_x p^{\alpha}_t(x,y)| &\lesssim \Big[\f{x}{t}\f{1}{\mu_\alpha(B(x,\sqrt t))}+ \f{xy^2}{t}\f{1}{\mu_{\alpha+2}(B(x,\sqrt t))}\Big]\exp\Big(-\f{|x-y|^2}{ct}\Big)\Big(1+\f{\sqrt{t}}{\rho(x)}+\f{\sqrt{t}}{\rho(y)}\Big)^{-N}\\
			&\lesssim \Big[\f{x}{t}\f{\sqrt t}{|x-y|}\f{1}{\mu_\alpha(B(x,\sqrt t))}+ \f{xy^2}{t}\f{t^{3/2}}{|x-y|^3}\f{1}{\mu_{\alpha+2}(B(x,\sqrt t))}\Big]\\
			& \ \ \times \exp\Big(-\f{|x-y|^2}{2ct}\Big)\Big(1+\f{\sqrt{t}}{\rho(x)}+\f{\sqrt{t}}{\rho(y)}\Big)^{-N}\\
			&\lesssim \Big[ \f{1}{\sqrt t\mu_\alpha(B(x,\sqrt t))}+  \f{\sqrt t}{\mu_{\alpha+2}(B(x,\sqrt t))}\Big] \exp\Big(-\f{|x-y|^2}{2ct}\Big)\Big(1+\f{\sqrt{t}}{\rho(x)}+\f{\sqrt{t}}{\rho(y)}\Big)^{-N}\\
			&\lesssim   \f{1}{\sqrt t\mu_\alpha(B(x,\sqrt t))} \exp\Big(-\f{|x-y|^2}{2ct}\Big)\Big(1+\f{\sqrt{t}}{\rho(x)}+\f{\sqrt{t}}{\rho(y)}\Big)^{-N}
		\end{aligned}
		\]
		for all $x,y\in (0,\vc)$ such that $x>2y$ or $y>2x$, where in the last inequality we used
		\begin{equation}\label{eq- mu alpha 2 and alpha}
			\mu_{\alpha+2}(B(x,\sqrt t))\simeq (\sqrt t+|x|)^2\mu_{\alpha}(B(x,\sqrt t))\gtrsim t\mu_{\alpha}(B(x,\sqrt t)).
		\end{equation}

		\medskip
		
		\underline{\textit{Subcase 2.2: $y/2\le x\le 2y$ and $x< \sqrt t.$}} From \eqref{eq-derivative of pt t small},  \eqref{eq-pt} and \eqref{eq- mu alpha 2 and alpha}, we obtain
		\[
		\begin{aligned}
			|\partial_x p^{\alpha}_t(x,y)| 			&\lesssim \Big[ \f{1}{\sqrt t\mu_\alpha(B(x,\sqrt t))}+  \f{\sqrt t}{\mu_{\alpha+2}(B(x,\sqrt t))}\Big] \exp\Big(-\f{|x-y|^2}{2ct}\Big)\Big(1+\f{\sqrt{t}}{\rho(x)}+\f{\sqrt{t}}{\rho(y)}\Big)^{-N}\\
			&\lesssim   \f{1}{\sqrt t\mu_\alpha(B(x,\sqrt t))} \exp\Big(-\f{|x-y|^2}{2ct}\Big)\Big(1+\f{\sqrt{t}}{\rho(x)}+\f{\sqrt{t}}{\rho(y)}\Big)^{-N}.
		\end{aligned}
		\]
		
		\medskip
		
		\underline{\textit{Subcase 2.3: $y/2\le x\le 2y$ and $x\ge \sqrt t.$}} In this case, $x\simeq y \gtrsim \sqrt t$. From \eqref{eq-derivative of pt}, we obtain
		\begin{equation*} 
			\begin{aligned}
				|\partial_x p^{\alpha}_t(x,y)|		&\le  \f{1-2e^{-2t}+e^{-4t}}{1-e^{-4t}} x p_t^{\alpha}(x,y) +\f{2e^{-2t}}{1-e^{-4t}}|x-y|p_t^{\alpha}(x,y) \\
				& \ \ +\f{2e^{-2t}}{1-e^{-4t}}y |p_t^{\alpha}(x,y)-xyp_t^{\alpha+2}(x,y)|\\
				&\le  \f{1-e^{-2t}}{1+e^{-2t}} x p_t^{\alpha}(x,y) +\f{2e^{-2t}}{1-e^{-4t}}|x-y|p_t^{\alpha}(x,y)  +\f{2e^{-2t}}{1-e^{-4t}}y \left|p_t^{\alpha}(x,y)-xyp_t^{\alpha+2}(x,y)\right|\\
				&=:E_1+E_2+E_3.
			\end{aligned}
		\end{equation*}
		
		Combining the estimates $e^{-2t}\simeq 1$ and $1-e^{-2t}\simeq 1-e^{-4t}\simeq t$
		with \eqref{eq-xy pt}, \eqref{eq-pt}, and the estimates $x\simeq y\gtrsim \sqrt t$, we have
		\[
		\begin{aligned}
			E_1+E_2&\simeq tx p_t^\alpha(x,y) +\f{|x-y|}{t}p_t^\alpha(x,y)\\
			&\lesssim \Big[\f{\sqrt t}{\mu_\alpha(B(x,\sqrt t))}+\f{1}{\sqrt t\mu_\alpha(B(x,\sqrt t))}\Big] \exp\Big(-\f{|x-y|^2}{2ct}\Big)\Big(1+\f{\sqrt{t}}{\rho(x)}+\f{\sqrt{t}}{\rho(y)}\Big)^{-N}\\
			&\lesssim \f{1}{\sqrt t\mu_\alpha(B(x,\sqrt t))}  \exp\Big(-\f{|x-y|^2}{2ct}\Big)\Big(1+\f{\sqrt{t}}{\rho(x)}+\f{\sqrt{t}}{\rho(y)}\Big)^{-N}.
		\end{aligned}
		\]
		It remains to estimate $E_3$. Using \eqref{eq4-Bessel}, we obtain
		\[
		\begin{aligned}
			|p_t^{\alpha}(x,y)-xyp_t^{\alpha+2}(x,y)|=& \f{2e^{-2t}}{1-e^{-4t}}\exp{\Big(-\f{1}{2}\f{1+e^{-4t}}{1-e^{-4t}}(x^2+y^2)\Big)}(xy)^{-(\alpha-1)/2}\\
			& \ \ \ \times \Big|I_{(\alpha+1)/2}\Big(\f{2e^{-2t}}{1-e^{-4t}}xy\Big)-I_{(\alpha-1)/2}\Big(\f{2e^{-2t}}{1-e^{-4t}}xy\Big)\Big|\\
			&\lesssim \f{1-e^{-4t}}{2e^{-2t}}\times \f{2e^{-2t}}{1-e^{-4t}}\exp{\Big(-\f{1}{2}\f{1+e^{-4t}}{1-e^{-4t}}(x^2+y^2)\Big)}(xy)^{-(\alpha+1)/2} I_{(\alpha+1)/2}\Big(xy\Big) \\
			&\lesssim tp_t^{\alpha+2}(x,y),
		\end{aligned}
		\]
		where in the last inequality we used $e^{-2t}\simeq 1$ and $1-e^{-4t}\simeq t$ as $t\in (0,1)$.
		
		Plugging this into the expression of $E_3$ and using the fact $e^{-2t}\simeq 1$ and $1-e^{-4t}\simeq t$ as $t\in (0,1)$, we further obtain
		\[
		\begin{aligned}
			|E_3|\lesssim y p_t^{\alpha+2}(x,y),
		\end{aligned}
		\]
		which, together with \eqref{eq-pt} and \eqref{doubling3}, implies
		\[
		\begin{aligned}
			|E_3|&\lesssim \f{y}{\mu_{\alpha+2}(B(y,\sqrt t))}  \exp\Big(-\f{|x-y|^2}{2ct}\Big)\Big(1+\f{\sqrt{t}}{\rho(x)}+\f{\sqrt{t}}{\rho(y)}\Big)^{-N}\\
			&\lesssim \f{1}{\sqrt t\mu_{\alpha}(B(y,\sqrt t))}  \exp\Big(-\f{|x-y|^2}{2ct}\Big)\Big(1+\f{\sqrt{t}}{\rho(x)}+\f{\sqrt{t}}{\rho(y)}\Big)^{-N}\\
			&\lesssim \f{1}{\sqrt t\mu_{\alpha}(B(x,\sqrt t))}  \exp\Big(-\f{|x-y|^2}{2ct}\Big)\Big(1+\f{\sqrt{t}}{\rho(x)}+\f{\sqrt{t}}{\rho(y)}\Big)^{-N},
		\end{aligned}
		\]
		where in the second  inequality we used
		\[
		\mu_{\alpha+2}(B(y,\sqrt t))\simeq (\sqrt t + y)^2\mu_{\alpha}(B(y,\sqrt t))\gtrsim y\sqrt t\mu_{\alpha}(B(y,\sqrt t)). 
		\]
		This completes the proof of \eqref{prop-derivaitive of pt_2} for $k=1$.

		It remains to prove \eqref{prop-derivaitive of pt_2} for $k=2$. By \eqref{eq-derivative of pt}, we have
		\begin{equation}
			\begin{aligned}
				|\partial^2_x p^{\alpha}_t(x,y)|  
				&\le\Big|-\f{1+e^{-4t}}{1-e^{-4t}}  p_t^{\alpha}(x,y)+\f{2e^{-2t}}{1-e^{-4t}}y^2 p_t^{\alpha+2}(x,y)\Big|\\
				&\ \   +\Big|-\f{1+e^{-4t}}{1-e^{-4t}} x \partial_x p_t^{\alpha}(x,y)+\f{2e^{-2t}}{1-e^{-4t}}xy^2 \partial_x p_t^{\alpha+2}(x,y)\Big|\\
				&=:F_1 +F_2.
			\end{aligned}
		\end{equation}	
		Using a similar argument as in the estimate of the first derivative
		$\partial_x p^{\alpha}_t(x,y)$ for the case $k=1$, we obtain
		\[
		|\partial^2_x p^{\alpha}_t(x,y) |\lesssim \f{1}{ t\mu_{\alpha}(B(x,\sqrt t))}  \exp\Big(-\f{|x-y|^2}{2ct}\Big)\Big(1+\f{\sqrt{t}}{\rho(x)}+\f{\sqrt{t}}{\rho(y)}\Big)^{-N},
		\]
		for all $t>0$ satisfying one of the following conditions:
		$t\ge1$; or $t\in(0,1)$ with $x>2y$ or $y>2x$; or
		$t\in(0,1)$ with $y/2\le x\le 2y$ and $x\le\sqrt{t}$. Hence, it remains to consider the case where $y/2 \le x \le 2y$ and $x > \sqrt{t}$. In this case, again by \eqref{eq-derivative of pt}, we have
		\[	F_1 = \f{1}{x}|\partial_x p_t^\alpha(x,y)|\lesssim \f{1}{ t\mu_{\alpha}(B(x,\sqrt t))}  \exp\Big(-\f{|x-y|^2}{2ct}\Big)\Big(1+\f{\sqrt{t}}{\rho(x)}+\f{\sqrt{t}}{\rho(y)}\Big)^{-N}.\]
		
		It remains to estimate $F_2$. To do this, we write
		\[
		\begin{aligned}
			F_2&\le \Big|\f{1-2e^{-2t}+e^{-4t}}{1-e^{-4t}} x \partial_x p_t^{\alpha}(x,y)\Big| +|\f{ 2e^{-2t} }{1-e^{-4t}} |x-y| |\partial_x p_t^{\alpha}(x,y)| \\
			& \ \  +\f{2e^{-2t}}{1-e^{-4t}}y\big| \partial_x p_t^{\alpha}(x,y)-xy \partial_x p_t^{\alpha+2}(x,y)\big|\\
			&\le \Big|\f{1- e^{-2t} }{1+e^{-2t}} x \partial_x p_t^{\alpha}(x,y)\Big| +|\f{ 2e^{-2t} }{1-e^{-4t}} |x-y| |\partial_x p_t^{\alpha}(x,y)| \\
			& \ \  +\f{2e^{-2t}}{1-e^{-4t}}y\big| \partial_x p_t^{\alpha}(x,y)-xy \partial_x p_t^{\alpha+2}(x,y)\big|\\
			&\lesssim t x \partial_x p_t^{\alpha}(x,y)  +\f{|x-y|}{t} |\partial_x p_t^{\alpha}(x,y)| +\f{y}{t}\big| \partial_x p_t^{\alpha}(x,y)-xy \partial_x p_t^{\alpha+2}(x,y)\big|\\
			&=:F_{21}+F_{22}+F_{23}.
		\end{aligned}
		\]
		Similarly to $E_1$ và $E_2$, we have
		\[
		F_{21}+F_{22}\lesssim \f{1}{ t\mu_{\alpha}(B(x,\sqrt t))}  \exp\Big(-\f{|x-y|^2}{2ct}\Big)\Big(1+\f{\sqrt{t}}{\rho(x)}+\f{\sqrt{t}}{\rho(y)}\Big)^{-N}.
		\]
		For $F_{23}$, using the expression \eqref{eq-pt}, the product rule, \eqref{eq3-Bessel} and the facts $e^{-2t}\simeq e^{-4t}\simeq 1$ and $1-e^{-4t}\simeq t$, we have
		\[
		\begin{aligned}
			F_{23}&\lesssim \f{y}{t}\f{1+e^{-4t}}{1-e^{-4t}} x |p_t^\alpha(x,y)-xyp_t^{\alpha+2}(x,y)| + \f{y^2}{t}\f{2e^{-2t}}{1-e^{-4t}}|A-B|\\
			&=:F_{231}+F_{232},
		\end{aligned}
		\]
		where
		\[
		A=\f{2e^{-2t}}{1-e^{-4t}}\exp{\Big(-\f{1}{2}\f{1+e^{-4t}}{1-e^{-4t}}(x^2+y^2)\Big)}(xy)^{-(\alpha-1)/2}I_{(\alpha+1)/2}\Big(\f{2e^{-2t}}{1-e^{-4t}}xy\Big)
		\]
		and
		\[
		B=\f{2e^{-2t}}{1-e^{-4t}}\exp{\Big(-\f{1}{2}\f{1+e^{-4t}}{1-e^{-4t}}(x^2+y^2)\Big)}(xy)^{-(\alpha-1)/2}I_{(\alpha+3)/2}\Big(\f{2e^{-2t}}{1-e^{-4t}}xy\Big).
		\]
		Similarly to the estimate of $E_3$, 
		\[
		F_{231}\lesssim \f{1}{ t\mu_{\alpha}(B(x,\sqrt t))}  \exp\Big(-\f{|x-y|^2}{2ct}\Big)\Big(1+\f{\sqrt{t}}{\rho(x)}+\f{\sqrt{t}}{\rho(y)}\Big)^{-N}.
		\]
		
		For the bound of $F_{232}$, using \eqref{eq4-Bessel} and \eqref{eq2-Bessel}, 
		\[
		\begin{aligned}
			|A-B|&\lesssim \f{1}{xy}\f{1-e^{-4t}}{2e^{-2t}}\Big(\f{2e^{-2t}}{1-e^{-4t}}xy\Big)^{-1/2}\exp\Big(\f{2e^{-2t}}{1-e^{-4t}}xy\Big)\\
			& \ \ \ \times \f{2e^{-2t}}{1-e^{-4t}}\exp{\Big(-\f{1}{2}\f{1+e^{-4t}}{1-e^{-4t}}(x^2+y^2)\Big)}(xy)^{-(\alpha-1)/2}.
		\end{aligned}
		\]
		Using the facts $x\simeq y$, $1-e^{-4t}\simeq t$, $e^{-2t}\simeq 1$, and \eqref{eq-rho laguer}, we further obtain
		\[
		\begin{aligned}
			|A-B|&\lesssim \f{\sqrt t}{x^{2+\alpha}} \exp\Big(\f{2e^{-2t}}{1-e^{-4t}}xy\Big) \exp{\Big(-\f{1}{2}\f{1+e^{-4t}}{1-e^{-4t}}(x^2+y^2)\Big)}\\
			&\lesssim \f{\sqrt t}{x^{2+\alpha}} \exp\Big(\f{2e^{-2t}}{1-e^{-4t}}xy-\f{e^{-2t}}{1-e^{-4t}}(x^2+y^2)\Big) \exp{\Big(-\f{1}{2}\f{1-2e^{-2t}+e^{-4t}}{1-e^{-4t}}(x^2+y^2)\Big)}\\
			&\lesssim \f{\sqrt t}{x^{2+\alpha}} \exp\Big( -\f{e^{-2t}}{1-e^{-4t}}|x-y|^2\Big) \exp{\Big(-\f{1}{2}\f{1-e^{-2t} }{1+e^{-2t}}(x^2+y^2)\Big)}\\
			&\lesssim \f{\sqrt t}{x^{2+\alpha}} \exp\Big( -c\f{|x-y|^2}{t}\Big) \exp{\Big(-ct(x^2+y^2)\Big)}\\
			&\lesssim \f{\sqrt t}{x^{2+\alpha}}\exp\Big(-\f{|x-y|^2}{2ct}\Big)\Big(1+\f{\sqrt{t}}{\rho(x)}+\f{\sqrt{t}}{\rho(y)}\Big)^{-N}.
		\end{aligned}
		\]	
		Inserting this into the expression of $F_{232}$ and using the facts that $x\simeq y$, $1-e^{-4t}\simeq t$, $e^{-2t}\simeq 1$, and $x>\sqrt{t}$,
		\[
		\begin{aligned}
			F_{232}&\lesssim \f{x^2}{t^2}\f{\sqrt t}{x^{2+\alpha}}\exp\Big(-\f{|x-y|^2}{2ct}\Big)\Big(1+\f{\sqrt{t}}{\rho(x)}+\f{\sqrt{t}}{\rho(y)}\Big)^{-N}\\
			&\lesssim \f{1}{t}\f{1}{\sqrt{t} x^{\alpha}}\exp\Big(-\f{|x-y|^2}{2ct}\Big)\Big(1+\f{\sqrt{t}}{\rho(x)}+\f{\sqrt{t}}{\rho(y)}\Big)^{-N}\\
			&\simeq \f{1}{\mu_\alpha(B(x,\sqrt t))}\exp\Big(-\f{|x-y|^2}{2ct}\Big)\Big(1+\f{\sqrt{t}}{\rho(x)}+\f{\sqrt{t}}{\rho(y)}\Big)^{-N},
		\end{aligned}
		\]
		where in the last inequality we used 
		\[
		\mu_\alpha(B(x,\sqrt t))\simeq \sqrt t(\sqrt t +x)^\alpha \simeq \sqrt{t} x^{\alpha}. 
		\]
		
		This completes the proof of \eqref{prop-derivaitive of pt_2} for $k=2$, and thus the proof of Proposition \ref{prop-derivaitive of pt}.
	\end{proof}
	
	\begin{theorem}\label{thm-all Riesz are rho CZ}
		For each $j=1,\ldots,n$, the Riesz transforms $x_j L_\alpha^{-1/2}$, $\partial_j L_\alpha^{-1/2}$, and $\delta_j L_\alpha^{-1/2}$ are $(\delta, \rho)$-Calder\'on--Zygmund operators with  $\delta=1$.
	\end{theorem}	
	
	\begin{proof}
		We first prove that the Riesz transform $x_j L_\alpha^{-1/2}$ is a $(\delta, \rho)$-Calder\'on--Zygmund operator with $\delta=1$.
		
		Using the subordination formula, we can write
		\[
		L_\alpha^{-1/2}=c\int_0^\vc \sqrt t e^{-tL_\alpha}\f{dt}{t},
		\] 
		which implies
		\begin{equation}\label{eq-Riesz xL}
			x_jL_\alpha^{-1/2}f(x)=c\int_0^\vc \sqrt t x_j e^{-tL_\alpha}f(x)\f{dt}{t}.
		\end{equation}
		By \eqref{eq-pt},
		\[
		\begin{aligned}
			|x_jL_\alpha^{-1/2}f(x)|&\lesssim \int_{\X}\int_0^\vc \sqrt t x_j  \f{1}{ \mu_{\alpha}(B(x,\sqrt t))}  \exp\Big(-\f{|x-y|^2}{ct}\Big)\Big(1+\f{\sqrt{t}}{\rho(x)}+\f{\sqrt{t}}{\rho(y)}\Big)^{-2}|f(y)|\f{dt}{t}d\mu_\alpha(y).
		\end{aligned}
		\]
		
		It is straightforward that
		\[
		\begin{aligned}
			\int_0^\vc \sqrt t x_j  &\f{1}{ \mu_{\alpha}(B(x,\sqrt t))}  \exp\Big(-\f{|x-y|^2}{ct}\Big)\Big(1+\f{\sqrt{t}}{\rho(x)}+\f{\sqrt{t}}{\rho(y)}\Big)^{-2}|f(y)|\f{dt}{t}\\
			&\lesssim  |x-y|x_j\f{1}{ \mu_{\alpha}(B(x,|x-y|))}\Big(1+\f{|x-y|}{\rho(x)}\Big)^{-2}\\
			&\lesssim  \f{|x-y|}{\rho(x)}\f{1}{ \mu_{\alpha}(B(x,|x-y|))}\Big(1+\f{|x-y|}{\rho(x)}\Big)^{-2},
		\end{aligned}
		\]
		where in the last inequality we used $x_j\lesssim \f{1}{\rho(x)}$. Consequently,
		\[
		\begin{aligned}
			|x_jL_\alpha^{-1/2}f(x)|&\lesssim \int_{\X}\f{|x-y|}{\rho(x)}\f{1}{ \mu_{\alpha}(B(x,|x-y|))}\Big(1+\f{|x-y|}{\rho(x)}\Big)^{-2}|f(y)| d\mu_\alpha(y)\\
			&= \int_{|x-y|<\rho(x)}\ldots +\int_{|x-y|\ge \rho(x)}\ldots\\
			&=:E(x)+F(x).
		\end{aligned}
		\]
		
		For the first term, we have
		\[
		\begin{aligned}
			E(x)&\simeq \int_{|x-y|<\rho(x)}\f{|x-y|}{\rho(x)}\f{1}{ \mu_{\alpha}(B(x,|x-y|))}|f(y)| d\mu_\alpha(y)\\
			&\lesssim \mathcal M f(x),
		\end{aligned}
		\]
		where $\mathcal M$ is the Hardy--Littlewood maximal operator.
		
		For the second term, using \eqref{doubling2}, we have
		\[
		\begin{aligned}
			F(x)&\simeq \int_{|x-y|\ge \rho(x)} \f{\rho(x)}{|x-y|} \f{1}{ \mu_{\alpha}(B(x,|x-y|))}|f(y)| d\mu_\alpha(y)\\
			&\lesssim \mathcal M f(x).
		\end{aligned}
		\]
		
		Consequently,
		\[
		|x_jL_\alpha^{-1/2}f(x)|\lesssim \mathcal M f(x).
		\]
		This, together with the $L^p$-boundedness of the maximal function $\mathcal M$ for $1<p<\vc$, implies that the Riesz transform is bounded on $L^p(\X)$ in this range.	
		
		It remains to show that $x_jL^{-1/2}$ satisfies \eqref{11:06, 13/12/2023} and \eqref{20:56, 15/12/2023} with $\delta=1$. 
		
		Let $K_j(x,y)$ be the associated kernel of $x_j L^{-1/2}$. By \eqref{eq-Riesz xL} and Proposition \ref{prop-derivaitive of pt}\eqref{prop-derivaitive of pt_1}, we obtain, for $N>0$,
		\begin{align}\label{eq-size xL}
			|K_j(x,y)|&\lesssim \int_0^\vc \f{1}{ \mu_{\alpha}(B(x,\sqrt t))}  \exp\Big(-\f{|x-y|^2}{ct}\Big)\Big(1+\f{\sqrt{t}}{\rho(x)}+\f{\sqrt{t}}{\rho(y)}\Big)^{-N}|f(y)|\f{dt}{t}\nonumber\\
			&\lesssim \f{1}{ \mu_{\alpha}(B(x,|x-y|))}   \Big(1+\f{|x-y|}{\rho(x)}+\f{|x-y|}{\rho(y)}\Big)^{-N}.
		\end{align}	
		Next, using Proposition \ref{prop-derivaitive of pt}\eqref{prop-derivaitive of pt_2} and the estimate $x_j\rho(x)\lesssim 1$, we obtain, for $k\neq j$ and $N>0$,
		\[
		\begin{aligned}
			|\partial_{k}x_jp^\alpha_t(x,y)|&=|x_j\partial_{k}p_t(x,y)|\\
			&\lesssim x_j\f{1}{\sqrt t \mu_{\alpha}(B(x,\sqrt t))}  \exp\Big(-\f{|x-y|^2}{ct}\Big)\Big(1+\f{\sqrt{t}}{\rho(x)}+\f{\sqrt{t}}{\rho(y)}\Big)^{-N-1}\\
			&\lesssim \f{1}{ t\mu_{\alpha}(B(x,\sqrt t))}  \exp\Big(-\f{|x-y|^2}{ct}\Big)\Big(1 +\f{\sqrt{t}}{\rho(y)}\Big)^{-N}.
		\end{aligned}
		\]
		If $k=j$, then we have
		\[
		\partial_{j}[x_jp^\alpha_t(x,y)]=p^\alpha_t(x,y) +x_j\partial_jp^\alpha_t(x,y).
		\]
		Similarly to the previous estimate,
		\[
		|x_j\partial_jp^\alpha_t(x,y)|\lesssim  \f{1}{ t\mu_{\alpha}(B(x,\sqrt t))}  \exp\Big(-\f{|x-y|^2}{ct}\Big)\Big(1 +\f{\sqrt{t}}{\rho(y)}\Big)^{-N}.
		\]
		In addition, using \eqref{eq-pt} and the fact $\rho(x)\lesssim 1$, 
		\[
		\begin{aligned}
			| p^\alpha_t(x,y)| &\lesssim  \f{1}{  \mu_{\alpha}(B(x,\sqrt t))}  \exp\Big(-\f{|x-y|^2}{ct}\Big)\Big(1+\f{\sqrt{t}}{\rho(x)}+\f{\sqrt{t}}{\rho(y)}\Big)^{-N-2}\\
			&\lesssim \f{1}{ t\mu_{\alpha}(B(x,\sqrt t))}  \exp\Big(-\f{|x-y|^2}{ct}\Big)\Big(1 +\f{\sqrt{t}}{\rho(y)}\Big)^{-N}.
		\end{aligned}
		\]
		Consequently,
		\[
		|\partial_{j}[x_jp^\alpha_t(x,y)]|\lesssim \f{1}{ t\mu_{\alpha}(B(x,\sqrt t))}  \exp\Big(-\f{|x-y|^2}{ct}\Big)\Big(1 +\f{\sqrt{t}}{\rho(y)}\Big)^{-N}.
		\]
		Therefore, for any $N>0$,
		\[
		|\nabla_x[x_jp^\alpha_t(x,y)]|\lesssim \f{1}{ t\mu_{\alpha}(B(x,\sqrt t))}  \exp\Big(-\f{|x-y|^2}{ct}\Big)\Big(1 +\f{\sqrt{t}}{\rho(y)}\Big)^{-N}.
		\]
		Using this together with \eqref{eq-Riesz xL}, we obtain, for any $N>0$,
		\begin{align}\label{eq-Holder xL}
			|\nabla_x K_j(x,y)|&=c\Big|\int_0^\vc \sqrt t \nabla[x_jp^\alpha_t(x,y)]\f{dt}{t}\Big|\nonumber\\
			&\lesssim \int_0^\vc \f{1}{ \sqrt t\mu_{\alpha}(B(x,\sqrt t))}  \exp\Big(-\f{|x-y|^2}{ct}\Big)\Big(1 +\f{\sqrt{t}}{\rho(y)}\Big)^{-N}\f{dt}{t}\nonumber\\
			&\lesssim \f{1}{|x-y|}\f{1}{ \mu_{\alpha}(B(x,|x-y|))}\Big(1 +\f{|x-y|}{\rho(y)}\Big)^{-N}.
		\end{align}
		Similarly,
		\[
		|\nabla_y K_j(x,y)| \lesssim \f{1}{|x-y|}\f{1}{ \mu_{\alpha}(B(x,|x-y|))}\Big(1 +\f{|x-y|}{\rho(x)}\Big)^{-N}.
		\]
		It follows that $K_j(x,y)$ satisfies \eqref{20:56, 15/12/2023} with $\delta=1$.
		
		Thus, we have proved that $x_j L_\alpha^{-1/2}$ is a
		$(1,\rho)$-Calder\'on--Zygmund operator.
		
		We next prove that the Riesz transform $\partial_j L_\alpha^{-1/2}$ is a $(1,\rho)$-Calder\'on–Zygmund operator. First, we note that
		\[
		\partial_j L_\alpha^{-1/2} = \delta_j L_\alpha^{-1/2} -x_jL_\alpha^{-1/2}.
		\]
		As proved above, the Riesz transform $x_j L_\alpha^{-1/2}$ is
		bounded on $L^2(\X)$, while the $L^2$-boundedness of the Riesz transform
		$\delta_j L_\alpha^{-1/2}$ was established in \cite{NS2}.
		Consequently, $\partial_j L_\alpha^{-1/2}$ is bounded on $L^2(\X)$. Moreover, using the formula 
		\[
		\partial_j L_\alpha^{-1/2}=c\int_0^\vc \sqrt t \partial_j e^{-t L_\alpha} \f{dt}{t}
		\]
		and Proposition \ref{prop-derivaitive of pt}\eqref{prop-derivaitive of pt_2}, we deduce that the associated kernel of $\partial_j L_\alpha^{-1/2}$ satisfies \eqref{11:06, 13/12/2023} and \eqref{20:56, 15/12/2023} with $\delta=1$. 
		
		Thus, we have proved that $\partial_j L_\alpha^{-1/2}$ is a
		$(1,\rho)$-Calder\'on--Zygmund operator.	
		
		Finally, since
		\[
		\delta_j L_\alpha^{-1/2} = \partial_j L_\alpha^{-1/2} +x_j L_\alpha^{-1/2}, 
		\]
		it follows that $\delta_j L_\alpha^{-1/2}$ is a $(1,\rho)$-Calder\'on--Zygmund operator.
		
		This completes our proof.
		
	\end{proof}

	From   Theorems \ref{thm-all Riesz are rho CZ}, \ref{thm-main thm 1 from H1 to weak L1} and \ref{thm-main thm 2 characterization BMO log}, we have:
	\begin{theorem}\label{thm-thm 1 Laguerre}
		Let  $\theta\in [0,\frac{1}{2})$ and $\rho$ be as in \eqref{eq-rho laguer}.  Then the following assertions hold:
		\begin{enumerate}[\rm (i)]
			\item If $b\in \mathrm{BMO}_{\rho,\theta}(\X)$, then the commutator $[b, T]$ is bounded from $H^1_\rho(\X)$ to $L^{1,\vc}(\X)$, where $T$ is any of the following Riesz transforms $\partial_j L_\alpha^{-1/2}$, $x_j L_\alpha^{-1/2}$, or $\delta_j L_\alpha^{-1/2}$ for $j=1,\ldots,n$.
			
			\item If $b\in \mathrm{BMO}^{\log}_{\rho,\theta}(\X)$, then the commutator $[b, T]$ is bounded from $H^1_\rho(\X)$ to $L^{1}(\X)$, where $T$ is any of the following Riesz transforms $\partial_j L_\alpha^{-1/2}$, $x_j L_\alpha^{-1/2}$, or $\delta_j L_\alpha^{-1/2}$ for $j=1,\ldots,n$.
		\end{enumerate}
			
	\end{theorem}
	In addition, we have the following results regarding the Riesz transforms $\partial_j L_\alpha^{-1/2}$.
\begin{theorem}
	Let  $\theta\in [0,\frac{1}{2})$ and $\rho$ be as in \eqref{eq-rho laguer}.  If $b\in \mathrm{BMO}_{\rho,\theta}(\X)$ is such that all commutators $[b, \partial_j L_\alpha^{-1/2}], j=1,\ldots,n,$ are bounded from $H^1_\rho(\X)$ to $L^1(\X)$, then $b\in \mathrm{BMO}_{\rho,\theta}^{\log}(\X)$. Moreover, 
	\[
	\|b\|_{{\rm BMO}_{\rho,\theta}^{\log}}\simeq \|b\|_{{\rm BMO}_{\rho,\theta}} + \sum_{j=1}^{n} \|[b, \partial_j L_\alpha^{-1/2}]\|_{H^1_\rho\to L^1}.
	\]	
\end{theorem}
	\begin{proof}
		It was proved in \cite{Preisner} that the Hardy space $H^1_\rho(\X)$ admits a characterization in terms of $\{\partial_j L_\alpha^{-1/2}\}_{j=1}^n$. Hence, the theorem follows directly from Theorem \ref{thm-main thm 3 chracterize BMO log by Tj}.
	\end{proof}
Finally, we have the boundedness on $H^1_\rho(\X)$ for the commutator of the Riesz transform $x_j L_\alpha^{-1/2}$, which improves the corresponding result in Theorem \ref{thm-thm 1 Laguerre}(ii).
	\begin{theorem}
		Let  $\theta\in [0,\frac{1}{2})$ and $\rho$ be as in \eqref{eq-rho laguer}.  If $b\in \mathrm{BMO}_{\rho,\theta}^{\log}(\X)$, then the commutator $[b, x_j L_\alpha^{-1/2}]$ is bounded on $H^1_\rho(\X)$ for all $j=1,\ldots,n$.
	\end{theorem}
	\begin{proof}
		From Proposition \ref{prop-derivaitive of pt} and Theorem \ref{thm-all Riesz are rho CZ}, and by an argument analogous to the proof of \cite[Theorem 1.7]{BLL}, one can verify that the Riesz transform $x_j L_\alpha^{-1/2}$ is bounded on $H^1_\rho(\X)$ for each $j$. Proposition \ref{prop- bounded on H1 and T1} then implies that $(x_j L_\alpha^{-1/2})^*1 \in \mathrm{BMO}_{\rho}^{\log}(\X)$. The theorem thus follows directly from Theorem \ref{thm-main thm 4 chracterize BMO log by Tj on Hardy space}.
	\end{proof}
	\subsection{Singular integral operators associated with admissible functions}
	
	Let $s\in (1,\infty)$, $\delta\in (0,1]$ and  $\eta\in C^1(\R)$ be such that $0\leq \eta\leq 1$, $\eta(t)=1$ as $|t|<1$ and $\eta(t)=0$ as $|t|>2$. Following Yang and Zhou \cite{YZ}, we consider a class of singular integral operators $T_{\rho,\eta}$ that are bounded on $L^s(\X)$ and whose kernels $K_{\rho,\eta}(x,y)$ satisfy the following conditions:
	\begin{enumerate}[\rm (i)]
		\item for all $f\in L^\infty(\X)$ with bounded support and all $x\notin$\,supp\,$f$,
		\begin{equation}\label{10:32, 26/12/2023}
			T_{\rho,\eta}(f)(x)=\int_{\X} K_{\rho,\eta}(x,y) f(y) d\mu(y)= \int_{\X} K(x,y)\eta\left(\frac{d(x,y)}{\rho(x)}\right) f(y) d\mu(y);
		\end{equation}
		
		\item there exist constants $C_1, C_2>0$ such that, for all $x\ne y$,
		\begin{equation}\label{12:59, 26/12/2023}
			|K(x,y)|\leq \frac{C_1}{\mu(B(x,d(x,y)))};
		\end{equation}
		and, for all $x,y,y'\in\X$ with $d(y,y')\leq \frac{d(x,y)}{2\kappa}$,
		\begin{equation}\label{13:00, 26/12/2023}
			|K(x,y)-K(x,y')|\leq \frac{C_2}{\mu(B(x,d(x,y)))} \left(\frac{d(y,y')}{d(x,y)}\right)^{\delta}.
		\end{equation}
	\end{enumerate}
	We note that there are many examples of singular integral operators of the form $T_{\rho,\eta}$. For example, on $\mathbb{R}^n$ with $1 \leq j \leq n$, the operator defined by
	$$T_{\rho,\eta}(f)(x)\equiv \text{p.v.}\int_{\R^n} \frac{x_j-y_j}{|x-y|^{n+1}}\eta\left(\frac{|x-y|}{\rho(x)}\right) f(y) dy$$
	is of the form considered in \eqref{10:32, 26/12/2023}. We also note that the motivation for studying these singular integral operators $T_{\rho,\eta}$ is the boundedness of Riesz transforms associated with Schr\"odinger operators and the local Riesz transforms (see \cite[Section 5]{YZ}).
	
	The following result follows from the proof of Theorem 4.10 in \cite{BDK}.
	
	\begin{lemma}\label{13:42, 26/12/2023}
		Let $s\in (1,\infty)$, $\delta\in (0,1]$,  and let $T_{\rho,\eta}$ be as in \eqref{10:32, 26/12/2023}. Then:
		\begin{enumerate}[\rm (i)]		
			\item\label{16:07, 26/12/2023} $T_{\rho,\eta}$ is bounded on $H^1_\rho(\X)$, or equivalently, $T_{\rho,\eta}^*(1)\in \mathrm{BMO}^{\log}_\rho(\X)$.
			
			\item\label{17:15, 26/12/2023} For each $N>0$, there exists a constant $C(N)>0$ such that
			$$|K_{\rho,\eta}(x,y)|\leq \frac{C(N)}{\mu(B(x,d(x,y)))} \left(1+\frac{d(x,y)}{\rho(x)}+ \frac{d(x,y)}{\rho(y)}\right)^{-N}$$
			for all $x\ne y$; and there exists a constant $C>0$ such that
			$$|K_{\rho,\eta}(x,y)-K_{\rho,\eta}(x,y')|\leq \frac{C}{\mu(B(x,d(x,y)))} \left(\frac{d(y,y')}{d(x,y)}\right)^{\delta}$$ 
			for all $x,y,y'\in \X$ with $d(y,y')\leq \frac{d(x,y)}{2\kappa}$.
			
			\item\label{16:30, 26/12/2023} For each $N>0$ and $\delta_0\in (0,\delta)$, there exists a constant $C(N,\delta_0)>0$ such that
			$$|K_{\rho,\eta}(x,y)-K_{\rho,\eta}(x,y')|\leq \frac{C(N,\delta_0)}{\mu(B(x,d(x,y)))} \left(\frac{d(y,y')}{d(x,y)}\right)^{\delta_0} \left(1+\frac{d(x,y)}{\rho(x)}\right)^{-N}$$
			for all $x,y,y'\in\X$ with $d(y,y')\leq \frac{d(x,y)}{2\kappa}$.
		\end{enumerate}
	\end{lemma}
	
	As a consequence of Remark \ref{rem 1.4}\eqref{10:19, 22/12/2023}, Theorem \ref{thm-main thm 4 chracterize BMO log by Tj on Hardy space}, and Lemma \ref{13:42, 26/12/2023}, we obtain the following result.
	
	\begin{theorem}\label{06:39, 27/12/2023}
		Let $s\in (1,\infty)$, $\delta\in (0,1]$, and let $T_{\rho,\eta}$ be as in \eqref{10:32, 26/12/2023}. Then:	 
		\begin{enumerate}[\rm (i)]
			\item $T_{\rho,\eta}$ is an $(s,\delta,\rho)$-Calder\'on--Zygmund operator, and thus belongs to $\K_{\rho,\theta,s}$ for all $\theta\in \big[0,\frac{\delta}{i(\rho)+1}\big)$.
			
			\item The commutator $[b,T_{\rho,\eta}]$ is bounded on $H^1_\rho(\X)$ for all $\theta\in \big[0,\frac{\delta}{i(\rho)+1}\big)$ and all $b\in \mathrm{BMO}^{\log}_{\rho,\theta}(\X)$.
		\end{enumerate}	 
	\end{theorem}
	
	Moreover, the conclusions of Theorem \ref{06:39, 27/12/2023} remain valid if the condition $\theta\in \big[0,\frac{\delta}{i(\rho)+1}\big)$ is replaced by the weaker assumption $\theta\in [0,\infty)$. More precisely, we have the following improvement of Theorem \ref{06:39, 27/12/2023}.
	
	\begin{theorem}\label{22:31, 26/12/2023}
		Under the same hypotheses as in Theorem \ref{06:39, 27/12/2023}, we have:	 
		\begin{enumerate}[\rm (i)]
			\item\label{15:43, 26/12/2023} $T_{\rho,\eta}$ belongs to the class $\K_{\rho,\theta,s}$ for all $\theta\in [0,\infty)$.
			
			\item\label{15:44, 26/12/2023} The commutator $[b,T_{\rho,\eta}]$ is bounded on $H^1_\rho(\X)$ for all $\theta\in [0,\infty)$ and all $b\in \mathrm{BMO}^{\log}_{\rho,\theta}(\X)$.
		\end{enumerate}	
	\end{theorem}
	
	\begin{proof}
		\eqref{15:43, 26/12/2023} Since $T_{\rho,\eta}$ is bounded on $H^1_\rho(\X)$ (see Lemma \ref{13:42, 26/12/2023}\eqref{16:07, 26/12/2023}), it suffices to prove that
		\begin{equation}\label{16:14, 26/12/2023}
			\|(b-b_B)T_{\rho,\eta}(\a)\|_{L^1}\lesssim \|b\|_{\mathrm{BMO}_{\rho,\theta}}
		\end{equation}
		for all $b\in \mathrm{BMO}_{\rho,\theta}(\X)$ and all $(H^1_\rho,s)$-atoms $\a$ related to balls $B$. Indeed, since $\a$ is an $(H^1_\rho,s)$-atom related to the ball $B=B(x_0,r)$, we have
		\begin{equation}\label{16:15, 26/12/2023}
			\text{supp\,$\a\subset B$\quad and \quad $0<r<\rho(x_0)$}.	
		\end{equation}		
		
		Let $k_0>i(\rho)$ be such that \eqref{admissible function} holds. For any $j\in \mathbb{Z}^+$ and any $x\in U_j(B)$, we consider the following two cases:
		
		{\sl Case 1:} $0<r<\frac{\rho(x_0)}{4}$. Then $\int_{\X} \a(x) d\mu(x)=0$. Thus, by Lemma \ref{13:42, 26/12/2023}\eqref{16:30, 26/12/2023}, Remark \ref{20:49, 02/12/2023}, and the comparability $d(x,x_0)\simeq (2\kappa)^{j+1} r$, we obtain
		\begin{align}\label{17:37, 26/12/2023}
			|T_{\rho,\eta}(\a)(x)|&=\left|\int_{B} (K_{\rho,\eta}(x,y)-K_{\rho,\eta}(x,x_0)) \a(y) d\mu(y)\right|\nonumber\\
			&\lesssim \int_{B} \frac{1}{\mu(B(x,d(x,x_0)))} \left(\frac{d(y,x_0)}{d(x,x_0)}\right)^{\delta/2} \left(1+\frac{d(x,x_0)}{\rho(x)}\right)^{-(k_0+1)^2 \theta} |\a(y)| d\mu(y)\nonumber\\
			&\lesssim \frac{1}{\mu(B(x_0,d(x,x_0)))} \left(\frac{r}{d(x,x_0)}\right)^{\delta/2} \left(1+\frac{d(x,x_0)}{\rho(x_0)}\right)^{-(k_0+1) \theta} \|\a\|_{L^1}\nonumber\\
			&\lesssim \mu((2\kappa)^{j+1} B)^{-1} (2\kappa)^{-j \delta/2} \left(1+\frac{(2\kappa)^{j+1} r}{\rho(x_0)}\right)^{-(k_0+1) \theta}.
		\end{align}
		
		{\sl Case 2:} $\frac{\rho(x_0)}{4}\leq r<\rho(x_0)$. Then, by Lemma \ref{13:42, 26/12/2023}\eqref{17:15, 26/12/2023}, Remark \ref{20:49, 02/12/2023}, and the comparabilities $\rho(x_0)\simeq r$ and $d(x,y)\simeq d(x,x_0)\simeq (2\kappa)^{j+1} r$ for all $y\in B$, we obtain
		\begin{align}\label{17:38, 26/12/2023}
			|T_{\rho,\eta}(\a)(x)|&=\left|\int_{B} K_{\rho,\eta}(x,y) \a(y) d\mu(y)\right|\nonumber\\
			& \lesssim \frac{1}{\mu(B(x,d(x,x_0)))}\left(1+\frac{d(x,x_0)}{\rho(x)}\right)^{-[(k_0+1)\delta/2 + (k_0+1)^2 \theta]} \|\a\|_{L^1}\nonumber\\
			&\lesssim \frac{1}{\mu(B(x_0,d(x,x_0)))}\left(1+\frac{d(x,x_0)}{\rho(x_0)}\right)^{-\delta/2}  \left(1+\frac{d(x,x_0)}{\rho(x_0)}\right)^{-(k_0+1) \theta} \nonumber\\
			&\lesssim \mu((2\kappa)^{j+1} B)^{-1} (2\kappa)^{-j \delta/2} \left(1+\frac{(2\kappa)^{j+1} r}{\rho(x_0)}\right)^{-(k_0+1) \theta}.
		\end{align}
		
		By H\"older's inequality, Lemma \ref{14:51, 06/10/2023}, \eqref{16:15, 26/12/2023}, the $L^s$-boundedness of $T_{\rho,\eta}$, \eqref{17:37, 26/12/2023}, and \eqref{17:38, 26/12/2023}, we deduce that
		\begin{align*}
			&\|(b-b_B)T_{\rho,\eta}(\a)\|_{L^1}=  \|(b-b_B)T_{\rho,\eta}(\a)\|_{L^1(U_0(B))} + \sum_{j=1}^\infty \|(b-b_B)T_{\rho,\eta}(\a)\|_{L^1(U_j(B))}\\
			&\leq \|b-b_B\|_{L^{s'}((2\kappa) B)} \|T_{\rho,\eta}(\a)\|_{L^s} + \sum_{j=1}^\infty \|b-b_B\|_{L^{s'}(U_j(B))}\|T_{\rho,\eta}(\a)\|_{L^s(U_j(B))}\\
			&\lesssim \mu((2\kappa)B)^{\frac{1}{s'}}\|b\|_{\mathrm{BMO}_{\rho,\theta}}  \mu(B)^{\frac{1}{s}-1} + \sum_{j=1}^\infty \mu((2\kappa)^{j+1} B)^{\frac{1}{s'}}(j+2)\left(1+\frac{(2\kappa)^{j+1} r}{\rho(x_0)}\right)^{(k_0+1)\theta}\|b\|_{\mathrm{BMO}_{\rho,\theta}}\\
			& \hskip 7cm\times \mu((2\kappa)^{j+1} B)^{\frac{1}{s}-1} (2\kappa)^{-j \delta/2} \left(1+\frac{(2\kappa)^{j+1} r}{\rho(x_0)}\right)^{-(k_0+1) \theta}\\
			&\lesssim \|b\|_{\mathrm{BMO}_{\rho,\theta}} \left(1+\sum_{j=1}^\infty (j+2) (2\kappa)^{-j\delta/2}\right) \lesssim \|b\|_{\mathrm{BMO}_{\rho,\theta}}.
		\end{align*}
		This proves \eqref{16:14, 26/12/2023}, and thus $T_{\rho,\eta}$ belongs to the class $\K_{\rho,\theta,s}$.
		
		\eqref{15:44, 26/12/2023} By \cite[Proposition 3.2]{YZ}, it suffices to prove that
		\begin{equation}\label{22:21, 26/12/2023}
			\|[b,T_{\rho,\eta}](\a)\|_{H^1_\rho}\lesssim \|b\|_{\mathrm{BMO}^{\log}_{\rho,\theta}}
		\end{equation}
		for all $(H^1_\rho,s)$-atoms $\a$ related to the balls $B$. Indeed, since $\a$ is an $(H^1_\rho,s)$-atom related to the ball $B=B(x_0,r)$, we have
		\begin{equation}\label{21:16, 26/12/2023}
			\text{supp\,$\a\subset B$\quad and \quad $0<r<\rho(x_0)$}.	
		\end{equation}				
		
		Let $p:=\frac{s+1}{2}\in (1,\infty)$ and $m:=\frac{s(s+1)}{s-1}\in (1,\infty)$. Then $\frac{1}{p}=\frac{1}{m} +\frac{1}{s}$. By  the generalized H\"older inequality, Lemma \ref{14:51, 06/10/2023}, \eqref{21:16, 26/12/2023}, and the $L^s$-boundedness of $T_{\rho,\eta}$, we obtain
		\begin{align}\label{14:41, 27/12/2023}
			\|(b-b_B)T_{\rho,\eta}(\a)\|_{L^p(U_0(B))}&\leq \|b-b_B\|_{L^m(U_0(B))} \|T_{\rho,\eta}(\a)\|_{L^s(U_0(B))}\nonumber\\
			&\lesssim \mu((2\kappa)B)^{\frac{1}{m}} \|b\|_{\mathrm{BMO}_{\rho,\theta}} \mu(B)^{\frac{1}{s}-1}\lesssim \mu(B)^{\frac{1}{p}-1} \|b\|_{\mathrm{BMO}^{\log}_{\rho,\theta}}
		\end{align}	
		For any $j\in\mathbb{Z}^+$, by the generalized H\"older inequality, Lemma \ref{14:51, 06/10/2023}, \eqref{17:37, 26/12/2023}, \eqref{17:38, 26/12/2023}, and the estimate $(j+2)\lesssim (2\kappa)^{j\delta/4}$ (see Lemma \ref{23:50, 07/10/2023}), we obtain
		\begin{align}\label{22:22, 26/12/2023}
			\|(b-b_B)T_{\rho,\eta}(\a)\|_{L^p(U_j(B))}&\leq \|b-b_B\|_{L^m(U_j(B))} \|T\a\|_{L^s(U_j(B))}\nonumber\\
			&\lesssim \mu((2\kappa)^{j+1} B)^{\frac{1}{m}} (j+2)\left(1+\frac{(2\kappa)^{j+1} r}{\rho(x_0)}\right)^{(k_0+1)\theta}\|b\|_{\mathrm{BMO}_{\rho,\theta}}\nonumber\\
			&\hskip1cm \times \mu((2\kappa)^{j+1} B)^{\frac{1}{s}-1} (2\kappa)^{-j \delta/2} \left(1+\frac{(2\kappa)^{j+1} r}{\rho(x_0)}\right)^{-(k_0+1) \theta}\nonumber\\
			&\lesssim (2\kappa)^{-j\delta/4} \mu((2\kappa)^j B)^{\frac{1}{p}-1} \|b\|_{\mathrm{BMO}^{\log}_{\rho,\theta}}.
		\end{align}	
		
		On the other hand, by H\"older's inequality, Lemma \ref{14:50, 06/10/2023}, \eqref{21:16, 26/12/2023}, the $L^s$-boundedness of $T_{\rho,\eta}$, \eqref{17:37, 26/12/2023}, \eqref{17:38, 26/12/2023}, and the estimate $\log\left(e+(2\kappa)^{j+1} t\right)\lesssim (2\kappa)^{j\delta/4} \log\left(e+ t\right)$ (see Lemma \ref{23:50, 07/10/2023}), it follows that
		\begin{align*}
			&\left|\int_{\X} (b(x)-b_B)T_{\rho,\eta}(\a)(x)d\mu(x)\right|\leq \sum_{j=0}^\infty \int_{U_j(B)} |b(x)-b_B| |T\a(x)|d\mu(x)\\
			&\leq \|b-b_B\|_{L^{s'}((2\kappa)B)} \|T_{\rho,\eta}(\a)\|_{L^s((2\kappa)B)} + \sum_{j=1}^\infty  \|b-b_B\|_{L^{s'}(U_j(B))}\|T_{\rho,\eta}(\a)\|_{L^s(U_j(B))}\\
			&\lesssim  \frac{\mu((2\kappa)B)^{\frac{1}{s'}}\|b\|_{\mathrm{BMO}^{\log}_{\rho,\theta}}}{\log\left(e+\frac{\rho(x_0)}{(2\kappa) r}\right)} \|\a\|_{L^s} + \sum_{j=1}^\infty \mu((2\kappa)^{j+1} B)^{\frac{1}{q'}}(j+2) \frac{\left(1+\frac{(2\kappa)^{j+1} r}{\rho(x_0)}\right)^{(k_0+1)\theta}}{\log\left(e+\frac{\rho(x_0)}{(2\kappa)^{j+1} r}\right)} \|b\|_{\mathrm{BMO}^{\log}_{\rho,\theta}} \\
			&\hskip6cm  \times  \mu((2\kappa)^{j+1} B)^{\frac{1}{q}-1} (2\kappa)^{-j \delta/2} \left(1+\frac{(2\kappa)^{j+1} r}{\rho(x_0)}\right)^{-(k_0+1) \theta}\\
			&\lesssim \frac{\|b\|_{\mathrm{BMO}^{\log}_{\rho,\theta}}}{\log\left(e+\frac{\rho(x_0)}{r}\right)}\left(1+\sum_{j=1}^\infty (j+2)(2\kappa)^{-j\delta/4}\right)\\
			&\lesssim  \frac{\|b\|_{\mathrm{BMO}^{\log}_{\rho,\theta}}}{\log\left(e+\frac{\rho(x_0)}{r}\right)}. 
		\end{align*}
		This, together with \eqref{14:41, 27/12/2023} and \eqref{22:22, 26/12/2023}, shows that $(b-b_B)T_{\rho,\eta}(\a)$ is, up to a factor $C \|b\|_{\mathrm{BMO}^{\log}_{\rho,\theta}}$, an $(H^1_\rho,p,\delta/4)$-molecule related to the ball $B$. Therefore, by Lemma \ref{07:01, 14/10/2023},
		$$\|(b-b_B)T_{\rho,\eta}(\a)\|_{H^1_\rho}\lesssim \|b\|_{\mathrm{BMO}^{\log}_{\rho,\theta}}.$$
		Thus, by Lemma \ref{07:01, 12/10/2023} and Lemma \ref{13:42, 26/12/2023}\eqref{16:07, 26/12/2023},
		\begin{align*}
			\|[b,T_{\rho,\eta}](\a)\|_{H^1_\rho}&\leq \|(b-b_B)T_{\rho,\eta}(\a)\|_{H^1_\rho}+ \|T_{\rho,\eta}(\a(b-b_B))\|_{H^1_\rho}\\
			&\lesssim \|b\|_{\mathrm{BMO}^{\log}_{\rho,\theta}} + \|T_{\rho,\eta}\|_{H^1_\rho\to H^1_\rho} \|\a(b-b_B))\|_{H^1_\rho} \lesssim \|b\|_{\mathrm{BMO}^{\log}_{\rho,\theta}}.
		\end{align*}
		This proves \eqref{22:21, 26/12/2023}, and thus completes the proof of Theorem \ref{22:31, 26/12/2023}.
		
	\end{proof}

	\subsection{Grand maximal operators $\M_{\rho,\epsilon,\beta,\gamma}$}

\begin{definition}\label{definition for test functions}
	Let $x_0\in\mathcal X$, $r>0$, $\beta\in (0,1]$ and $\gamma >0$. A function $\phi$ is said to belong to the space of test functions $\mathcal G(x_0,r,\beta,\gamma)$ if there exists a constant $C_\phi>0$ such that
	\begin{enumerate}[\rm (i)]
		\item\label{16:03, 02/12/2023} $|\phi(x)| \leq  \frac{C_\phi}{\mu(B(x_0,r)) + \mu(B(x_0, d(x,x_0)))}\left(\frac{r}{r+ d(x,x_0)}\right)^\gamma$ for all $x\in\mathcal X$;
		
		\item\label{16:02, 02/12/2023} $|\phi(x) - \phi(y)|\leq \frac{C_\phi}{\mu(B(x_0,r)) + \mu(B(x_0, d(x,x_0)))}\left(\frac{r}{r+ d(x,x_0)}\right)^\gamma \left(\frac{d(x,y)}{r+ d(x,x_0)}\right)^\beta$ for all $x,y\in \mathcal X$ satisfying $d(x,y)\leq \frac{r + d(x,x_0)}{2\kappa}$.
		
	\end{enumerate}
	Moreover, for any $\phi\in \mathcal G(x_0,r,\beta,\gamma)$, we define its norm by	
	$$\|\phi\|_{\mathcal G(x_0,r,\beta,\gamma)}:= \inf \left\{C_\phi: \mbox{\eqref{16:03, 02/12/2023} and \eqref{16:02, 02/12/2023} hold}\right\}.$$
\end{definition}

Throughout this subsection, we fix $x_0\in \mathcal X$. In Definition \ref{definition for test functions}, it is easy to see that $\mathcal G(x_0,1,\beta,\gamma)$ is a Banach space. Furthermore, for any  $x\in \mathcal X$ and $r>0$, we have $\mathcal G(x,r,\beta,\gamma)= \mathcal G(x_0,1,\beta,\gamma)$  with equivalent norms (but of course the constants are depending on $x$ and $r$). For simplicity, we write $\mathcal G(\beta,\gamma)$ instead of $\mathcal G(x_0,1,\beta,\gamma)$.

Given $\epsilon\in (0,1]$ and $\beta,\gamma\in (0,\epsilon]$, we define the space $\mathcal G^\epsilon_0(\beta,\gamma)$ to be the completion of $\mathcal G(\epsilon,\epsilon)$ in $\mathcal G(\beta,\gamma)$, and denote by $(\mathcal G^\epsilon_0(\beta,\gamma))'$ the space of all continuous linear functionals on $\mathcal G^\epsilon_0(\beta,\gamma)$. We say that $f$ is a distribution if $f\in (\mathcal G^\epsilon_0(\beta,\gamma))'$. For a distribution $f$, the  grand maximal function  $\M_{\rho,\epsilon,\beta,\gamma}(f)$ is defined by
$$\M_{\rho,\epsilon,\beta,\gamma}(f)(x) = \sup\{|\langle f,\phi \rangle|: \phi\in \mathcal G^\epsilon_0(\beta,\gamma), \|\phi\|_{\mathcal G(x,r,\beta,\gamma)}\leq 1\; \mbox{for some}\; r\in (0,\rho(x))\}.$$

\begin{definition}\label{11:42, 31/12/2025}
	Let $\epsilon\in (0,1)$ and $\beta,\gamma\in (0,\epsilon)$. The Hardy space $H^1_{\rho,\epsilon,\beta,\gamma}(\mathcal X)$ is then defined as the space of all distributions $f\in (\mathcal G^\epsilon_0(\beta,\gamma))'$ such that
	$$\|f\|_{H^1_{\rho,\epsilon,\beta,\gamma}}:=\|\M_{\rho,\epsilon,\beta,\gamma}(f)\|_{L^1}<\infty.$$
\end{definition}

We recall the following result from \cite[Theorem 3.2]{YZ}.

\begin{theorem}\label{thm:H1-equiv}
	Let $\epsilon\in (0,1)$ and $\beta,\gamma\in (0,\epsilon)$. Then $H^1_{\rho,\epsilon,\beta,\gamma}(\mathcal X)= H^1_\rho(\X)$ with equivalent norms.
\end{theorem}

The following proposition provides an additional example of operators 
belonging to the class $\K_{\rho,\theta,q}$.

\begin{proposition}\label{prop:grand-maximal-Kclass}
	For all $\epsilon\in (0,1)$, $\beta,\gamma\in (0,\epsilon)$, $0\leq\theta<\min\left\{\frac{\beta}{i(\rho)+1},\frac{\gamma}{[i(\rho)+1]^2}\right\}$, and $q\in (1,\infty]$, the grand maximal operators $\M_{\rho,\epsilon,\beta,\gamma}$ belong to the class $\K_{\rho,\theta,q}$ .
\end{proposition}

\begin{proof}
	By Theorem \ref{thm:H1-equiv}, $\M_{\rho,\epsilon,\beta,\gamma}$ is bounded from $H^1_\rho(\X)$ into $L^1(\X)$. Thus, it suffices to prove that
	\begin{equation}\label{14:09, 14/12/2023}
		\|(b-b_B)\M_{\rho,\epsilon,\beta,\gamma}(\a)\|_{L^1}\lesssim \|b\|_{\mathrm{BMO}_{\rho,\theta}}
	\end{equation}
	for all $b\in \mathrm{BMO}_{\rho,\theta}(\X)$ and all $(H^1_\rho,q)$-atoms $\a$ related to the balls $B$. Indeed, since $0\leq\theta<\min\left\{\frac{\beta}{i(\rho)+1},\frac{\gamma}{[i(\rho)+1]^2}\right\}$ and $\a$ is an $(H^1_\rho,q)$-atom related to the ball $B=B(x_0,r_0)$, there exists $k_0>i(\rho)$ such that \eqref{admissible function} holds and
	\begin{equation}\label{17:24, 09/12/2023}		
		\varepsilon_0:= \min\left\{\beta,\frac{\gamma}{k_0+1}\right\} -(k_0+1)\theta>0,\quad 0<r_0<\rho(x_0).
	\end{equation}		
	For any $j\in\mathbb Z^+$ and $x\in U_j(B)$, we consider the following two cases:
	
	{\sl Case 1:} $0<r_0<\frac{\rho(x_0)}{4}$. Then $\int_{\X} \a(y)d\mu(y)=0$. Hence, from Definition \ref{definition for test functions}\eqref{16:02, 02/12/2023} and the comparability $\mu(B(x, d(x,x_0)))\sim \mu(B(x_0, d(x,x_0)))\sim \mu((2\kappa)^j B)$, we obtain
	\begin{align*}
		\left|\langle\a,\phi\rangle \right|&=\left|\int_{B} \a(y)(\phi(y)-\phi(x_0))d\mu(y)\right|\\
		&\leq \int_{B}|\a(y)| \|\phi\|_{\G(x,r,\beta,\gamma)} \left(\frac{d(y,x_0)}{r+d(x,x_0)}\right)^\beta \frac{1}{\mu(B(x, d(x,x_0)))} d\mu(y)\\
		&\lesssim (2\kappa)^{-j \beta} \frac{1}{\mu((2\kappa)^j B)}
	\end{align*}
	holds for all $\phi\in \mathcal G^\epsilon_0(\beta,\gamma)$, $\|\phi\|_{\mathcal G(x,r,\beta,\gamma)}\leq 1$ for some $r\in (0,\rho(x))$. This implies that
	\begin{equation}\label{21:22, 02/12/2023}
		\M_{\rho,\epsilon,\beta,\gamma}(\a)(x)\lesssim  \frac{(2\kappa)^{-j \beta}}{\mu((2\kappa)^j B)}.
	\end{equation}

	{\sl Case 2:} $\frac{\rho(x_0)}{4}\leq r_0<\rho(x_0)$. Then $d(x,y)\sim d(x,x_0)\sim (2\kappa)^j r_0\sim (2\kappa)^j \rho(x_0)$ for all $y\in B=B(x_0,r_0)$. Therefore, by Definition \ref{definition for test functions}\eqref{16:03, 02/12/2023}, the comparability $\mu(B(x, d(x,x_0)))\sim \mu(B(x_0, d(x,x_0)))\sim \mu((2\kappa)^j B)$, and Remark \ref{20:49, 02/12/2023}, we obtain
	\begin{align*}
		\left|\langle\a,\phi\rangle \right|&=\left|\int_{B} \a(y)\phi(y) d\mu(y)\right|\\
		&\leq \int_{B}|\a(y)| \|\phi\|_{\G(x,r,\beta,\gamma)} \frac{1}{\mu(B(x, d(x,y)))} \left(\frac{r}{r+d(x,y)}\right)^\gamma d\mu(y)\\
		&\lesssim \int_{B}|\a(y)| \frac{1}{\mu(B(x, d(x,x_0)))} \left(\frac{1}{1+\frac{d(x,x_0)}{\rho(x)}}\right)^\gamma d\mu(y)\\
		&\lesssim \frac{1}{\mu((2\kappa)^j B)} \left(1+\frac{d(x,x_0)}{\rho(x_0)}\right)^{-\frac{\gamma}{k_0+1}} \lesssim(2\kappa)^{-j\frac{\gamma}{k_0+1}} \frac{1}{\mu((2\kappa)^j B)}
	\end{align*}
	holds for all $\phi\in \mathcal G^\epsilon_0(\beta,\gamma)$, $\|\phi\|_{\mathcal G(x,r,\beta,\gamma)}\leq 1$ for some $r\in (0,\rho(x))$. This implies that
	\begin{equation}\label{21:23, 02/12/2023}
		\M_{\rho,\epsilon,\beta,\gamma}(\a)(x)\lesssim  \frac{(2\kappa)^{-j\frac{\gamma}{k_0+1}}}{\mu((2\kappa)^j B)}.
	\end{equation}
	
	By H\"older's inequality, Remark \ref{20:52, 19/12/2023}, the $L^q$-boundedness of $M$, together with \eqref{21:22, 02/12/2023}, \eqref{21:23, 02/12/2023}, Lemma \ref{14:51, 06/10/2023}, and \eqref{17:24, 09/12/2023}, we obtain
	\begin{align*}
		\|(b-b_B) \M_{\rho,\epsilon,\beta,\gamma}(\a)\|_{L^1}&=\sum_{j=0}^{\infty} \|(b-b_B) \M_{\rho,\epsilon,\beta,\gamma}(\a)\|_{L^1(U_j)}\\
		&\lesssim \|b-b_B\|_{L^{q'}(U_0)}\|M\a\|_{L^q(U_0)} +\\
		&\hskip1cm + \sum_{j=1}^{\infty} \|(b-b_B)\|_{L^1(U_j)}\frac{(2\kappa)^{-j\min\left\{\beta,\frac{\gamma}{k_0+1}\right\}}}{\mu((2\kappa)^j B)}\\
		&\lesssim \|b\|_{\mathrm{BMO}_{\rho,\theta}}+ \|b\|_{\mathrm{BMO}_{\rho,\theta}}\sum_{j=1}^{\infty} (j+2) (2\kappa)^{-j\varepsilon_0} \lesssim \|b\|_{\mathrm{BMO}_{\rho,\theta}}.
	\end{align*}
	This proves \eqref{14:09, 14/12/2023}, and thus $\M_{\rho,\epsilon,\beta,\gamma}$ belongs to the class $\K_{\rho,\theta,q}$.
	
\end{proof}

The following result can be proved by arguments similar to those 
used in the proof of Theorem \ref{thm-main thm 2 characterization BMO log}; 
we omit the details.

\begin{theorem}
	Let $\epsilon\in (0,1)$, $\beta,\gamma\in (0,\epsilon)$, $0\leq\theta<\min\left\{\frac{\beta}{i(\rho)+1},\frac{\gamma}{[i(\rho)+1]^2}\right\}$, and $b\in \mathrm{BMO}_{\rho,\theta}(\X)$. Then the following statements are equivalent:
	\begin{enumerate}[\rm (i)]
		\item $b\in \mathrm{BMO}_{\rho,\theta}^{\log}(\X)$;
		
		\item The commutator $[b, \M_{\rho,\epsilon,\beta,\gamma}]$ is bounded from $H^1_\rho(\X)$ into $L^1(\X)$.
	\end{enumerate}
	Moreover, in this case,
	\[\|b\|_{\mathrm{BMO}_{\rho,\theta}^{\log}}\simeq \|b\|_{\mathrm{BMO}_{\rho,\theta}} + \|[b, \M_{\rho,\epsilon,\beta,\gamma}]\|_{H^1_\rho\to L^1}.\]
\end{theorem}
	
	\begin{acknowledgement}
		The Anh Bui and Xuan Thinh Duong were supported by the research grant ARC DP260101083 from the Australian Research Council. Luong Dang Ky is supported by the Vietnam National Foundation for Science and Technology Development (NAFOSTED) under grant number 101.02-2023.09.
	\end{acknowledgement}


\begin{thebibliography}{MTW1}
		\bibitem{AH}
		P. Auscher and T. Hyt\"onen, \textit{Orthonormal bases of regular wavelets in spaces of homogeneous type}. Appl. Comput. Harmon. Anal. 34 (2013), no. 2, 266--296.
		
		\bibitem{BPQ} F. Berra, G. Pradolini and P. Quijano, \textit{Mixed inequalities for operators associated to critical radius functions with applications to Schr\"odinger type operators}. Potential Anal. 60 (2024), no. 1, 253--283. 
		
		\bibitem{BHQ} B. Bongioanni, E. Harboure and P. Quijano, \textit{Weighted inequalities for Schr\"odinger type singular integrals}. J. Fourier Anal. Appl. 25 (2019), no. 3, 595--632.
		
		\bibitem{BHS} B. Bongioanni, E. Harboure and O. Salinas,  \textit{Commutators of Riesz transforms related to Schr\"odinger operators}. J. Fourier Anal. Appl. 17 (2011), no. 1, 115--134.	
		
		\bibitem{BDK} T. A. Bui, X. T. Duong and L. D. Ky, \textit{Hardy spaces associated to critical functions and applications to  $T1$  theorems}. J. Fourier Anal. Appl. 26 (2020), no. 2, Paper No. 27, 67 pp.
		
		\bibitem{BDK18} T. A. Bui, X. T. Duong and F. K. Ly, \textit{Maximal function characterizations for new local Hardy type spaces on spaces of homogeneous type}. Trans. Amer. Math. Soc. 370 (2018), no. 10, 7229--7292.
		
		\bibitem{BLL}  T. A. Bui, J. Li and F. K. Ly, \textit{$T1$ criteria for generalised Calder\'on--Zygmund type operators on Hardy and $\mathrm{BMO}$ spaces associated to Schr\"odinger operators and applications}. Ann. Sc. Norm. Super. Pisa Cl. Sci. (5) 18 (2018), no. 1, 203--239.
		
		
		\bibitem{CRW} R. R. Coifman, R. Rochberg and G. Weiss,
		\textit{Factorization theorems for Hardy spaces in several variables}.	Ann. of Math. (2) 103 (1976), no. 3, 611--635.
		
		\bibitem{CW} R. R. Coifman and G. Weiss,  \textit{Extensions of Hardy spaces and their use in analysis}. Bull. Amer. Math. Soc. 83 (1977), no. 4, 569--645.  
		
		\bibitem{DLPV} G. Dafni, C. H. Lau, T. Picon and C. Vasconcelos, \textit{Inhomogeneous cancellation conditions and Calder\'on--Zygmund type operators on $h^p$}. Nonlinear Anal. 225 (2022), Paper No. 113110, 22 pp.
		
		
		\bibitem{DZ99} J. Dziuba\'nski and J. Zienkiewicz, \textit{Hardy space $H^1$ associated to Schr\"odinger operator with potential satisfying reverse H\"older inequality}. Rev. Mat. Ibero. 15 (1999), 279--296. 
		
		\bibitem{DZ03} J. Dziuba\'nski and J. Zienkiewicz, \textit{$H^p$ spaces associated with Schr\"odinger operators with potentials from reverse H\"older classes}. Colloq. Math. 98 (2003), no. 1, 5--38.
		
		%
		%
		\bibitem{F} C. Fefferman, \textit{The uncertainty principle}. Bull. Amer. Math. Soc. 9 (1983), 129--206.
		
		\bibitem{GLY} L. Grafakos, L. Liu and D. Yang, \textit{Maximal function characterizations of Hardy sapces on RD-spaces and their applications}. Sci. China Ser. A 51 (2008), 2253--2284.
		
		\bibitem{GLP} Z. Guo, P. Li and L. Peng, \textit{$L^p$ boundedness of commutators of Riesz transforms associated to Schr\"odinger operator}. J. Math. Anal. Appl. 341 (2008), no. 1, 421--432.
		
		\bibitem{HMY} Y. Han, D. M\"uller and  D. Yang,  \textit{A theory of Besov and Triebel-Lizorkin spaces on metric measure spaces modeled on Carnot-Carath\'eodory spaces}. Abstr. Appl. Anal. 2008, Art. ID 893409, 250 pp.
		
		\bibitem{HST} E. Harboure, C. Segovia and J. L. Torrea, \textit{Boundedness of commutators of fractional and singular integrals for the extreme values of $p$}. Illinois J. Math. 41 (1997), no. 4, 676--700. 
		
		\bibitem{HW} Y. Hu and Y. Wang, \textit{Hardy type estimates for Riesz transforms associated with Schr\"odinger operators}. Anal. Math. Phys. 9 (2019), no. 1, 275--287.
		
		\bibitem{JLiu} G. Jiang and Y. Liu, \textit{$L^p$ estimates for commutators of Riesz transforms associated with Schr\"odinger operators on stratified groups}. Anal. Math. Phys. 9 (2019), no. 1, 531--553.
		
		 
		\bibitem{Ky13} L. D. Ky, \textit{Bilinear decompositions and commutators of singular integral operators}. Trans. Amer. Math. Soc. 365 (2013), no. 6, 2931--2958.		
		
		\bibitem{Ky15} L. D. Ky,  \textit{Endpoint estimates for commutators of singular integrals related to Schr\"odinger operators}. Rev. Mat. Iberoam. 31 (2015), no. 4, 1333--1373.
		
		\bibitem{Ky25} L. D. Ky, \textit{Generalized Calder\'on--Zygmund operators on the Hardy space $H^1_\rho(\X)$}. Banach J. Math. Anal. 19 (2025), no. 2, Paper No. 20.
		
		\bibitem{Le} N. N. Lebedev, \textit{Special functions and their applications}. Dover Publications, Inc., New York, 1972.
		
		\bibitem{HQLi} H. Q. Li,  \textit{Estimations $L^p$ des op\'erateurs de Schr\"odinger sur les groupes nilpotents}. J. Funct. Anal. 161 (1999), no. 1, 152--218.
		
		\bibitem{LP} P. Li and L. Peng, \textit{Endpoint estimates for commutators of Riesz transforms associated with Schr\"odinger operators}. Bull. Aust. Math. Soc. 82 (2010), no. 3, 367--389.
		
		\bibitem{LKY} Y. Liang, L. D. Ky and D. Yang, \textit{Weighted endpoint estimates for commutators of Calderón-Zygmund operators}. Proc. Amer. Math. Soc. 144 (2016), no. 12, 5171--5181.
		
		\bibitem{LLL} C-C. Lin, H. Liu and Y. Liu, \textit{Hardy spaces associated with Schr\"odinger operators on the Heisenberg group}, arXiv:1106.4960.
		
		
		
		\bibitem{LD} Y. Liu and J. Dong, \textit{Some estimates of higher order Riesz transform related to Schr\"odinger type operators}. Potential Anal. 32 (2010), no. 1, 41--55.
		
		\bibitem{LHD} Y. Liu, J. Huang and J. Dong, \textit{Commutators of Calder\'on--Zygmund operators related to admissible functions on spaces of homogeneous type and applications to Schr\"odinger operators}. Sci. China Math. 56 (2013), no. 9, 1895--1913. 
		
		
		\bibitem{MSTZ} T. Ma, P. R. Stinga, J. L. Torrea and C. Zhang, \textit{Regularity estimates in H\"older spaces for Schr\"odinger operators via a  $T1$  theorem}. Ann. Mat. Pura Appl. (4) 193 (2014), no. 2, 561--589.
		
		\bibitem{MS} R. A. Macias and C. Segovia, \textit{Lipschitz functions on spaces of homogeneous type}. Adv. in Math. 33 (1979), no. 3, 257--270.
		
		
		\bibitem{Na} E. Nakai, \textit{A generalization of Hardy spaces $H^p$ by using atoms}. Acta Math. Sin. (Engl. Ser.) 24 (2008), no. 8, 1243--1268.
		
		\bibitem{NS2} A. Nowak and K. Stempak,  \textit{Riesz transforms for multi-dimensional Laguerre function expansions}. Adv. Math. 215 (2007), no. 2, 642--678.
		
		
		\bibitem{Pe} C. P\'erez, \textit{Endpoint estimates for commutators of singular integral operators}. J. Funct. Anal. 128 (1995), no. 1, 163--185.
		
		\bibitem{Preisner} M. Preisner, \textit{Riesz transform characterization of $H^1$ spaces associated with certain Laguerre expansions}. J. Approx. Theory 164 (2012), no. 2, 229--252.
		
		
		\bibitem{Sh} Z. Shen, \textit{$L^p$ estimates for Schr\"odinger operators with certain potentials}. Ann. Inst. Fourier (Grenoble) 45 (1995), no. 2, 513--546.
		
		
		
		\bibitem{ST} J-O. Str\"omberg and A. Torchinsky, \textit{Weighted Hardy spaces}. Lecture Notes in Math., 1381 Springer-Verlag, Berlin, 1989.
		
		\bibitem{YYZ} Da. Yang,  Do. Yang and  Y. Zhou,  \textit{Localized Morrey-Campanato spaces on metric measure spaces and applications to Schr\"odinger operators}. Nagoya Math. J. 198 (2010), 77--119.
		
		\bibitem{YZ} D. Yang and Y. Zhou, \textit{Localized Hardy spaces $H^1$ related to admissible functions on RD-spaces and applications to Schr\"odinger operators}. Trans. Amer. Math. Soc. 363 (2011), no. 3, 1197--1239.
		
		
	\end{thebibliography}
\end{document}